\documentclass[onefignum,onetabnum]{siamart171218}

\usepackage{lipsum}
\usepackage{amsfonts}
\usepackage{graphicx}
\usepackage{epstopdf}
\usepackage{algorithmic}
\ifpdf
  \DeclareGraphicsExtensions{.eps,.pdf,.png,.jpg}
\else
  \DeclareGraphicsExtensions{.eps}
\fi

\usepackage{orcidlink}
\usepackage{nicefrac}
\usepackage{caption}
\usepackage{siunitx} % units
\usepackage{booktabs} % nice tables
\usepackage{multirow}
\usepackage{multicol}
\allowdisplaybreaks

\usepackage[font=footnotesize]{caption}
\usepackage[numbers]{natbib} % more citation tools

\usepackage{wrapfig} % For the neighbourhood plot

\newsiamremark{example}{Example}
\newsiamremark{remark}{Remark}
\newsiamremark{hypothesis}{Hypothesis}
\crefname{hypothesis}{Hypothesis}{Hypotheses}
\newsiamthm{claim}{Claim}

\headers{Efficient NFFT--assisted nonlocal image denoising}{A. Miniguano-Trujillo, J. W. Pearson, and B. D. Goddard}

\title{Efficient nonlocal linear image denoising: Bilevel optimization with Nonequispaced Fast Fourier Transform and matrix-free preconditioning\thanks{Submitted to the editors DATE.
\funding{A.M-T. acknowledges support from the MAC--MIGS CDT Scholarship under EPSRC grant EP/S023291/1. J.W.P. acknowledges support from EPSRC grant
EP/S027785/1.}}}

\author{Andrés Miniguano-Trujillo%
    \thanks{Maxwell Institute for Mathematical Sciences, The University of Edinburgh and Heriot-Watt University, Bayes Centre, Edinburgh, United Kingdom
    (\email{Andres.Miniguano-Trujillo@ed.ac.uk} \orcidlink{0000-0002-0877-628X})}
\and John W. Pearson\thanks{School of Mathematics and Maxwell Institute for Mathematical Sciences, The University of Edinburgh, Edinburgh, United Kingdom 
  (\email{j.pearson@ed.ac.uk} \orcidlink{0000-0002-6063-1766}, \email{b.goddard@ed.ac.uk} \orcidlink{0000-0002-8781-014X})}
\and Benjamin D. Goddard\footnotemark[3]}

\usepackage{amsopn}

\usepackage{geometry}
\usepackage{amsfonts}
\usepackage{amsmath}
\usepackage{amssymb}
\usepackage{mathtools}
\usepackage{stmaryrd} % nice brackets
\usepackage{arydshln} % Dashed matrices
\usepackage{subcaption}

\usepackage{ wasysym }
\usepackage{csquotes}
\usepackage{commath}
\usepackage{pdfpages}

\newcommand{\MaxT}{$\mathsf{\times}$}
\newcommand{\EE}[1][]{\(10^{#1}\)}

\newcommand{\R}{\mathbb{R}}
\newcommand{\C}{\mathbb{C}}
\newcommand{\Z}{\mathbb{Z}}
\newcommand{\N}{\mathbb{N}}
\newcommand{\F}{\mathbb{F}}

\DeclareMathOperator*{\argmin}{arg\,min}
\DeclareMathOperator*{\argmax}{arg\,max}
\DeclareMathOperator*{\spn}{span}

\newcommand{\bfa}{\mathbf{a}}
\newcommand{\bfb}{\mathbf{b}}

\newcommand{\bfe}{\mathbf{e}}
\newcommand{\bff}{\mathbf{f}}
\newcommand{\bfg}{\mathbf{g}}
\newcommand{\bfu}{\mathbf{u}}
\newcommand{\bfv}{\mathbf{v}}
\newcommand{\bfw}{\mathbf{w}}
\newcommand{\bfq}{\mathbf{q}}
\newcommand{\bfr}{\mathbf{r}}
\newcommand{\bfx}{\mathbf{x}}
\newcommand{\bfy}{\mathbf{y}}

\newcommand{\etab}{\pmb{\eta}}
\newcommand{\UnitSim}[2][U]{ \mathsf{#2}_{#1} }
\newcommand{\Idn}[1][n]{\mathsf{I}_{#1}}
\newcommand{\Ones}[1][n]{\mathbf{1}_{#1}}
\newcommand{\Zeros}[1][n]{\mathbf{0}_{#1}}
\newcommand{\Prec}[1][]{ \mathsf{P}_{\text{#1}}}
\DeclareMathOperator{\Rayleigh}{\mathcal{R}}

\newcommand{\sfL}{\mathsf{L}}

\newcommand{\llb}{\llbracket}
\newcommand{\rrb}{\rrbracket}

\DeclareMathOperator{\dive}{div}
\DeclareMathOperator{\diag}{diag}
\DeclareMathOperator{\proj}{proj}

\usepackage{enumitem}       % not needed after finishing the draft

\makeatletter
\newcommand*{\addFileDependency}[1]{% argument=file name and extension
  \typeout{(#1)}% latexmk will find this if $recorder=0 (however, in that case, it will ignore #1 if it is a .aux or .pdf file etc and it exists! if it doesn't exist, it will appear in the list of dependents regardless)
  \@addtofilelist{#1}% if you want it to appear in \listfiles, not really necessary and latexmk doesn't use this
  \IfFileExists{#1}{}{\typeout{No file #1.}}% latexmk will find this message if #1 doesn't exist (yet)
}
\makeatother

\newcommand*{\myexternaldocument}[1]{%
    \externaldocument{#1}%
    \addFileDependency{#1.tex}%
    \addFileDependency{#1.aux}%
}
\ifpdf
\hypersetup{
  pdftitle={TBA},
  pdfauthor={}
}
\fi

\myexternaldocument{ex_supplement}

\begin{document}

\maketitle

% REQUIRED
\begin{abstract}
  We present a new approach for nonlocal image denoising, based around the application of an unnormalized extended Gaussian ANOVA kernel within a bilevel optimization algorithm. A critical bottleneck when solving such problems for finely--resolved images is the solution of huge--scale, dense linear systems arising from the minimization of an energy term. We tackle this using a Krylov subspace approach, with a Nonequispaced Fast Fourier Transform utilized to approximate matrix--vector products in a matrix--free manner. We accelerate the algorithm using a novel change of basis approach to account for the (known) smallest eigenvalue--eigenvector pair of the matrices involved, coupled with a simple but frequently very effective diagonal preconditioning approach. We present a number of theoretical results concerning the eigenvalues and predicted convergence behavior, and a range of numerical experiments which validate our solvers and use them to tackle parameter learning problems. These demonstrate that very large problems may be effectively and rapidly denoised with very low storage requirements on a computer.
\end{abstract}

% REQUIRED

% REQUIRED

% ------------------------------------------------------------------------------------------- %
% ------------------------------------------------------------------------------------------- %
% ------------------------------------------------------------------------------------------- %

\section{Introduction}

Image denoising is the process of reducing or removing unwanted noise in an image, and forms a crucial step for image processing and image analysis \cite{Scherzer2009,Chan2005,Bouman2022}. This paper focuses on fast, effective, and storage--efficient numerical methods for the resolution of such problems.

Variational methods are a powerful approach to study and formalize denoising models from a functional perspective \cite{Schoenlieb2015,Scherzer2009}. In particular, the minimization of the total variation seminorm \cite{Rudin_1992} and the success of the Chambolle--Pock algorithm \cite{Chambolle2011} have inspired an extensive catalog of improved methods and algorithms. In this work, we consider nonlocal image denoising, which was inspired by image filter methods incorporating information from generalized neighborhoods of a pixel to reconstruct it. This way, the amount of redundancy of data in a digital image becomes a crucial tool for guiding the reconstruction \cite{Kindermann2005,Gilboa2009}. In particular, the nonlocal means filter \cite{Buades2005b} and its many extensions \cite{Buades2010,Dabov2007} have been effective in removing noise, while preserving textures and avoiding common pitfalls of local models. Crucially, its extension to a variational setting in \cite{Gilboa2007} led to a range of new developments, including a rich analytical theory \cite{Brezis2018,Davoli2023}, by making use of tools from nonlocal calculus \cite{Gilboa2009}.

Two caveats are encountered when employing nonlocal tools for denoising: parameter identification and computational cost. The former is a classical question of denoising methods, due to the parametric nature of models that weight regularization against fidelity and underlying information. Bilevel optimization has proven an effective tool to determine high--quality reconstructions from trained parameters; see \cite{Davoli2023} for a review. The latter is instead inherent to nonlocal terms. The use of such terms yields dense discretizations of operators, the computation of which can be computationally cumbersome on its own, and thus limits the usage of these tools for large--scale or more complex applications. 

The parameter calibration of the nonlocal means model was studied in \cite{D’Elia2021}. There, the nonlocal kernel was approximated using a large--neighborhood localized approach that was suitable for small-- and medium--sized images and which resulted in a multi--banded semi--sparse matrix. This sparsification technique is just another approach to improve the performance of the reconstruction originally proposed in \cite{Buades2005b}. For instance, \cite{Gilboa2007} used a semilocal approach to reduce the amount of comparisons between patches; see also \cite{Mahmoudi2005,Darbon2008,Froment2014}.

Nonlocal operators in variational imaging arise as continuous extensions of kernel--based methods that encode an underlying space of features. Reproducible kernel Hilbert spaces (RKHSs) have gained significant attention due to their key connection with Gaussian processes \cite{Foucart2022, Rasmussen2006,Vaart2008}, their flexibility in representing complex function spaces \cite{Kennedy2013,MoriartyOsborne2024},  and their effectiveness in various inverse problem settings \cite{Latz2024,Bai2024,Teckentrup2020}. A notable example is the ANOVA kernel, which is designed to decompose interactions among different feature components, allowing for a structured representation of multi--dimensional dependencies \cite{Durrande2013,Gunn2002,Berlinet2004}. The kernel is often presented as
\begin{equation}\label{eq:ANOVA_baseline}
	\gamma^t_{\textsf{ANOVA}}(\bfx, \bfy) \coloneqq \sum_{1 \leq i_1 < i_2 < \ldots < i_t \leq d} \, \prod_{j=1}^t \gamma_{i_j} (x_{i_j}, y_{i_j}),
\end{equation}
where \(\bfx,\bfy \in \R^d\) and \(t\leq d\) \cite{Duvenaud2011,ShaweTaylor2004,Vapnik2000,Stitson1998}. The term ANOVA is an abbreviation for analysis of variance. Known also as the \(t\)--th order additive kernel, \cref{eq:ANOVA_baseline} compares the variance of the data in detail by incorporating all one--dimensional interactions from the kernels \(\gamma_{i_j}\). Multiple base kernels are multiplied to cover higher order feature interactions according to the order \(t\) of the additive kernel which allow us to rely on fewer features \cite[\S 9.2]{ShaweTaylor2004}. 

Letting all underlying kernels \(\gamma_{i_j}\) in \cref{eq:ANOVA_baseline} be the squared--exponential  \( \gamma(x,y) = e^{-\sigma^{-2} |x-y|^2}\) with the same shape parameter \(\sigma > 0\), we obtain, by the power law, that for the \(d\)--th order ANOVA it holds: 
\[%\label{eq:ANOVA_Squared_Exponential}
	\gamma^d_{\textsf{ANOVA}}(\bfx, \bfy) = \exp\cbr{ -\sigma^{-2} \norm{\bfx - \bfy}^2 }.
\]
In other words, the \(d\)--th order ANOVA kernel is nothing else than the squared--exponential kernel, also known as the Gaussian kernel \cite[\S 3]{ShaweTaylor2004}, evaluated at all feature dimensions at once \cite{Wagner2024}. Summing all the \(t\)--th order ANOVA kernels, we obtain a full additive kernel involving all order at the possible expense of over--determination. Instead, if we combine terms only relying on a small number of features can be more promising for generalization \cite{Stitson1998}. In \cite{Nestler2022}, this idea motivated the introduction of the \emph{extended Gaussian ANOVA kernel} 
\begin{equation}\label{eq:Original_ANOVA}
	\gamma_{\textsf{e--ANOVA}}(\bfx,\bfy) = \frac{1}{\sfL} \sum_{\ell=1}^\sfL \exp \cbr{ {-}\sigma^{-2} \big\| \mathsf{W}_{\ell} [\bfx] - \mathsf{W}_{\ell} [\bfy]  \big\|_2^2 } 
	\eqqcolon
	\frac{1}{\sfL} \sum_{\ell=1}^{\sfL} \gamma_{\ell}(\bfx, \bfy)
	.
\end{equation}
Here, the sub--kernels \(\gamma_\ell\) depend only on low--dimensional feature interactions which are given by windows of features \( W_\ell \). For each subgroup \(\ell\), we have that \( W_\ell[\bfx] \in \R^t\) and often \( t \ll d\).

The kernel in \cref{eq:Original_ANOVA} is also known as the additive Gaussian kernel \cite{Wagner2024}, and we say it is unnormalized whenever its diagonal \( \gamma_{\textsf{e--ANOVA}}(\bfx,\bfx)\) is fixed to a non--unit value.

This work introduces an efficient computational framework, suitable for large images and for the incorporation of a large number of features without the need for localization or sparsity enforcement. The method is based on the fast summation approach and feature splitting for applying the \emph{unnormalized extended Gaussian ANOVA kernel} \cref{eq:Original_ANOVA}. The summation algorithm is in turn based on the Nonequispaced Fast Fourier Transform (NFFT) \cite{Kunis2006a,Potts2003,Potts2004}, which we deploy for image denoising problems in this work. This algorithm yields an operator that is used to solve the denoising problem for a given parameter configuration. In our setting, this leads to a linear system where the nonlocal kernel is used to assemble an abstract unnormalized graph Laplacian shifted by a constant scaling of the identity. The abstract linear operator setting is suitable for matrix--free solution schemes, including suitable Krylov subspace iterative methods that require a fixed (moderate) number of matrix--vector multiplications.

Our proposed framework extends the applicability of \cite{D’Elia2021} to large--scale imaging tasks.  Specifically, a broader class of kernels is introduced to incorporate full covariate information without enforcing sparsity, thus capturing broad nonlocal interactions. This step avoids the introduction of approximation errors from localization while gaining computational efficiency through fast summation techniques. Moreover, in contrast to \cite{D’Elia2021}, where the numerical challenges posed by ill--conditioning are not addressed, this work analyzes and integrates dedicated preconditioning strategies. Additionally, we provide a rigorous analysis of discrete nonlocal operators, enabling the use of conjugate gradient (CG) as an efficient linear solver. 
Thus, our combined contributions provide an efficient and portable procedure for solving discretized nonlocal systems arising in similar imaging applications. The structure of the linear systems involved may also arise in applications outside imaging, for instance feature training for neural networks \cite{HeNe2024}.

The solution of linear systems governed by graph Laplacians using such methods is greatly aided by effective preconditioning. A number of approaches have been developed and refined when the full matrix is available, see for instance \cite{Spielman2011,Napov2016,Gao2023}, however besides diagonal--based preconditioners \cite{Takapoui2016,Qu2022} there is limited work on methods for matrix--free systems. The sparse approximate inverse approach for $M$--matrices is an exception, where \((\mathsf{I}+S)\)--type preconditioners have proven a computationally cheap alternative \cite{Simons1998,Tam2005,Jin2006,Zhang2009}; see also \cite{SaberiNajafi2014}. Notwithstanding, in practice, many of these reduce to the classical Jacobi diagonal preconditioner for the graph Laplacian. This work extends the Jacobi preconditioner, and its quadratic variant applied in \cite{D’Elia2021}, using a matrix--free approach that would otherwise result in dense preconditioners under a change of basis that singles out the smallest eigenvalue of the system. Motivated by deflation, exact operator formulae are derived for the change of basis, whose applicability extends to Hermitian finite dimensional operators. The transformation displays linear computational complexity and can be used efficiently to project the action of any preconditioner under the new basis for a given subblock of the linear system.

The key contributions of this article are summarized below:
\begin{itemize}
\item We embed an NFFT routine into bilevel optimization for image denoising for the first time.

\item We devise a novel strategy based on a change of basis and preconditioning, to accelerate the matrix--free solution of the resulting linear systems with Krylov subspace methods.

\item We present new analytic results concerning spectral properties of graph Laplacians and shifted equations governed by them. We also present results on the effectiveness of the preconditioners, specifically by localizing the eigenvalues of the preconditioned system.

\item We complement our analysis with numerical tests to demonstrate the effectiveness of our NFFT and deflation by change of basis strategies, with near--constant iteration counts exhibited.

\item We also train a denoising parameters from a dataset of images featuring patterns, textures, and sudden intensity jumps, showcasing the effectiveness of the kernel to obtain reconstructions under the presence of such structures.

\item We provide open--source code for the above tests, available at  %\newline
\href{https://github.com/andresrmt/Prec_GLs_NFFT_BLO}{\texttt{edin.ac/3zh86hT}}.
\end{itemize}

This paper is structured as follows. In \cref{sec:Preliminaries} we provide necessary background on variational denoising models, the choice of similarity kernel, and the finite element discretization. In \cref{sec:NFFT_Gauss} we describe the NFFT and its use in computing the Fast Gauss Transform. In \cref{sec:PreconditioningUnderChangeOfBasis} we outline our preconditioning methodology, including the bespoke change of basis and its application to graph Laplacian matrices, diagonal preconditioners, theoretical results on spectral properties of graph Laplacians and the resulting preconditioned matrices, and a comparison of the properties of different preconditioners. In \cref{sec:Numerics} we provide a range of numerical experiments to further validate our approach, and in \cref{sec:Conc} we present our conclusions.

% ----------------------------------------------------- %
\subsection{Notation and concepts}
We will use bold notation for vectors whose components are values of some underlying variables that are relevant in this paper. However, we will keep in mind that any bold quantity may result from the discretization of a continuous or higher--dimensional object.

For \(m,n \in \N\), let the integer interval between \(m\) and \(n\) be denoted by \( \llb m,n \rrb \coloneqq \{m, m+1, \ldots, n-1, n \} \). We denote by \( \mathcal{M}_n (\F)\) the space of matrices of size \(n \times n\) with entries in the field \(\F \in \{\R,\C\}\). In particular, we denote by \(\mathsf{I}_n\) the identity matrix in \( \mathcal{M}_n (\F)\). The columns  of \(\mathsf{I}_n\) define the canonical basis \( \{ \bfe_i \}_{i\in \llb 1,n\rrb} \), \( \Ones \) denotes a vector of ones of size \(n\), while \(\Ones[n,n]\) is the \(n\)--dimensional square matrix of all ones. Similarly, \( \Zeros = 0 \, \Ones\) and \( \Zeros[n,n] = 0 \, \Ones[n,n] \). 

Let \( \bfu = (u_i)_{i \in \llb 1,n \rrb} \) and \( \bfv = (v_i)_{i \in \llb 1,n \rrb}\) be two vectors in \(\F^n\). The operator \( \diag : \F^n \to \mathcal{M}_n (\F)\) maps \(\bfu\) to the diagonal matrix \( \diag(\bfu)\) whose diagonal coincides with the vector \( \bfu\). We will \emph{overload} this operator by also defining \( \diag: \mathcal{M}_n(\F) \to \F^n \) such that \( \diag(A) \) is the \(n\) dimensional vector whose entries correspond to the diagonal of the matrix \(A\). The \emph{Hadamard product} \( \circ: \F^n \times \F^n \to \F^n \) computes the component--wise product of its arguments; i.e., \( (\bfu \circ \bfv)_i = u_i v_i \) for all \( i \in \llb 1,n\rrb\). We will also \emph{overload} this operator with its natural extension on matrices of the same size with \( A \circ B \) being the component--wise product between the entries of the matrices \(A, B \in \mathcal{M}_{m,n} (\F)\). The \emph{Kronecker product} \( \otimes: \mathcal{M}_{m,n}(\F) \times \mathcal{M}_{p,q}(\F) \to \mathcal{M}_{mp, nq}(\F) \) is a block matrix \( A \otimes B\) where each element of \(A \in \mathcal{M}_{m,n}(\F)\) is replaced by a copy of the entire matrix \(B \in \mathcal{M}_{p,q}(\F)\), scaled by the value of that element. In particular, we have that \( \bfu \otimes \bfv^\dagger = \bfu \bfv^\dagger\). For any integer \(p\), the upper triangular operator \( \mathrm{Triu}_{p}: \mathcal{M}_{m,n} (\F) \to \mathcal{M}_{m,n} (\F) \) takes a matrix \(A\) and returns its upper triangular part \( \mathrm{Triu}_{p}(A) \), which is equal to \(A\) from its \(p\)--th diagonal and above, but zero elsewhere.

Letting \(S\) be a set, the indicator function of \(S\) at \(x\) is given by \( \iota_S \), taking value \(1\) if \( x\in S\) and \(0\) otherwise. In case \(S\) is a singleton, we will omit the argument of \( \iota_S \). The canonical projection of a vector on the finite set \(S\), \( \pi_S (\bfu) \), consists of all entries of \(\bfu\) such that \( \iota_S\) is positive. As a shorthand, we define \( \bfu_{ m:n } \coloneqq \pi_{ \llb m,n\rrb } (\bfu)\), but omit the colon when \(m = n\), with \( u_m \coloneqq \pi_{\{m\}} (\bfu) \). We extend this to matrix indexing in a similar fashion; i.e., for \( A \in \mathcal{M}_n( \F )\) we have that \( A_{\ell:m, p:q}\) is the \( (m-\ell+1) \times (q-p+1) \) sub--matrix consisting of the elements in the \(\ell\)--to--\(m\) rows and the \(p\)--to--\(q\) columns of \(A\).

Throughout this document, \( \| \bfu \| \) represents the norm of \( \bfu \). We will make use of different norms, particularly \( p \)--norms. In case there is a need to specify the particular \( p \)--norm employed, it will be denoted as \( \| \cdot \|_p \), with \( p \in [1,\infty] \). The choice of \( p \) depends on the analysis being conducted. If no specific norm is mentioned, it can be inferred from context. In particular, the \(2\)--norm is induced by the inner product \( \langle \bfu, \bfv\rangle = \bfv^\dagger \bfu\) for all \( \bfu,\bfv \in \F^n\). Here \( \bfv^\dagger\) is the conjugate transpose of \(\bfv\).

For a Hermitian matrix \( A \in \mathcal{M}_n (\F)\), the spectrum of \(A\) is the set of eigenvalues \( \Sigma(A) = \{ \lambda_i \}_{i \in \llb 1,n \rrb} \) where each eigenvalue is indexed in such a way that \( |\lambda_1| \leq |\lambda_2| \leq \ldots \leq  |\lambda_n|\). We denote the largest eigenvalue (in modulus) \( \rho(A) \coloneqq \lambda_n\) as the \emph{spectral radius of \(A\)}. The second smallest eigenvalue (in modulus) is labeled \( a(A) \coloneqq \lambda_2\) and corresponds to the \emph{algebraic connectivity of \(A\)}.

% ------------------------------------------------------------------------------------------- %
% ------------------------------------------------------------------------------------------- %
% ------------------------------------------------------------------------------------------- %
\section{Preliminaries on variational denoising and control}\label{sec:Preliminaries}

Our work is motivated by the practical computation of a particular class of nonlocal noise filters. Image denoising is the process of reducing or removing unwanted noise in an image. It is used either to enhance the visual appeal of images or as a preliminary step for image analysis and feature extraction. Noise typically corrupts an observed image due to instrument manipulation or medium artifacts. There are many models for noise \cite[\S2.3]{Scherzer2009}, where Gaussian additive noise is the most suitable for digital images \cite[\S1]{Lebrun2012}. The model states that a clean discretized image \(\bfu\in \R^n\) is distorted by a Gaussian signal \(\mathbf{h}\in \mathcal{N}(\Zeros,\hat{\sigma}^2 \Idn)\), yielding \(\bff = \bfu + \mathbf{h}\) as a registered image. Finding \(\bfu\) from this relation is a difficult task, due to the random nature of \(\mathbf{h}\). As a result, the main task of denoising is to identify and remove the noise while preserving the most important information and structures. Thus, there are several approaches to analyzing and tackling this problem from the perspectives of stochastic processes, calculus of variations, partial differential equations, and wavelet theory \cite{Lebrun2012,Buades2008}. We refer the reader to \cite{Bouman2022,Schoenlieb2015,Scherzer2009,Chan2005} for comprehensive introductions and reviews on imaging and denoising methods.

% ----------------------------------------------------- %
\subsection{Variational denoising}

Variational imaging casts an image as a signal \(u\) defined on a continuous medium. In this context, a registered noisy sample \(\bff\) is just a quantization of a noisy signal \(f\). Here, denoising methods use image priors and minimize an energy function to calculate the denoised image:
\begin{equation}\label{incl:Energy_method}
	u \in \argmin_{v\in V} \mathcal{E}(v),
\end{equation}
where \(V\) is a suitable function space. The motivation for \cref{incl:Energy_method} is the \emph{maximum a posteriori} (MAP) probability estimate. From a Bayesian perspective, the MAP estimate of \(v\) can be written as
\begin{equation*}%\label{eq:MAP_Estimate}
	u = \argmax_{v\in V} \, \log \mathrm{P}(f|v) + \log \mathrm{P}(v) ,
\end{equation*}
where the first term \( \mathrm{P}(f|v)\) is a likelihood function of \(v\), and the second term \(\mathrm{P}(v)\) represents the image prior. For additive white Gaussian noise, the energy \(\mathcal{E}\) is often selected as 
\begin{equation*}%\label{eq:Energy_Fidelity}
	u = \argmin_{v\in V} \, \mu R(v) + \frac{\lambda}{2} \| v -f \|^2_{L^2},
\end{equation*}
where \(R\) denotes a regularization term weighted by \(\mu > 0\), and \( \frac{1}{2} \| v -f \|^2_{L^2}\) is a data fidelity weight denoting the distance between the reconstructed and noisy samples weighted by \(\lambda > 0\). We refer the reader to \cite{Piening2024} for a comprehensive review on the design of priors inspired by the underlying statistical properties of the noise model. In this work, we will focus on a prior that exploits the variational structure of the problem. This perspective naturally connects to the use of differential operators and their nonlocal counterparts, which we now discuss.

The variational standpoint allows us to properly define local differential operators and energies (e.g., Total Variation, Laplacians, Mean Curvature, and so on) and extend them to their nonlocal counterparts \cite{Kindermann2005,Gilboa2009}. 
The reasoning behind local operators is that they are good for smoothing out regions of the image, as showcased in the Rudin--Osher--Fatemi model of Total Variation Minimization (ROF) \cite{Rudin_1992}. ROF allows for \(u\) to be piece--wise discontinuous, thus an effective choice for edge preservation. However, the ROF model and many of its variants can fail to preserve the fine structure; i.e., details and textures of an image. Due to the regularity assumptions of local operators, fine structures are smoothed out because they behave, under the action of the operator, as noise. 
Nonlocal imaging is a valuable alternative that aims to preserve most of the fine details (e.g., contours, texture, high-contrast edges, flat areas) by considering lower regularity assumptions on \(u\) and the large amount of redundant information in textured images \cite{RosaVargas2016,Rosa2017}.

% ----------------------------------------------------- %
\subsection{Preliminaries on nonlocal calculus}

We follow the approach of \cite{Gilboa2007}. Let \(\Omega\) be a bounded domain in \(\R^{d_1}\) and \(u: \Omega \to \R\) be a real function. The notion of a derivative can be generalized to a nonlocal framework by introducing the \emph{nonlocal derivative}:
\begin{equation}\label{eq:nonlocal_derivative}
	\partial_\bfy u(\bfx) \coloneqq
	\frac{ u(\bfy) - u(\bfx) }{ \delta(\bfy,\bfx)},
\end{equation}
where \( \bfx,\bfy \in \Omega\) and \(\delta(\bfy,\bfx) \in (0,\infty]\) is a positive and symmetric quantity. Let us define the weight function \( \gamma \coloneqq \delta^{-2}\) so that \( \gamma(\bfx,\bfy) \in [0,\infty)\) and \(\gamma(\bfx,\bfy) = \gamma(\bfy,\bfx)\).
Then \cref{eq:nonlocal_derivative} can be written as
\(
	\partial_\bfy u(\bfx) \coloneqq [ u(\bfy) - u(\bfx) ] \sqrt{\gamma(\bfx,\bfy)}.
\)
Moreover, we define the \emph{nonlocal gradient} \( \nabla_\gamma u(\bfx): \Omega \to \Omega \times \Omega\) as the family of nonlocal derivatives 
\[
	\nabla_\gamma u(\bfx) = \{ \partial_\bfy u(\bfx) : \, \bfy \in \Omega \}.
\]
Alternatively, we have the pointwise representation \( \nabla_\gamma u: \Omega\times \Omega \to \R\) given by
\(
	\nabla_\gamma u (\bfx,\bfy) = [ u(\bfy) - u(\bfx) ] 
	\newline 
	\sqrt{\gamma(\bfx,\bfy)}.
\)
Let us also define the \emph{nonlocal divergence} for vector functions \( v: \Omega \times \Omega \to \R\) given by
\[
	(\dive_\gamma v) (\bfx) \coloneqq \int\limits_\Omega [ v(\bfx,\bfy) - v(\bfy,\bfx) ] \sqrt{\gamma (\bfx,\bfy)} \dif y.
\]
We define the \emph{nonlocal diffusion} of \(u\) as the operator \( \Delta_\gamma: \Omega \to \R\) such that
\begin{align*}
	\Delta_\gamma u(\bfx) &\coloneqq  \frac{1}{2}  \dive_\gamma \big( \nabla_\gamma u (\bfx) \big)
	= 
	\dfrac{1}{2}    \int\limits_\Omega [ \partial_\bfy u(\bfx) - \partial_\bfx u(\bfy) ] \sqrt{\gamma(\bfx,\bfy)} \dif \bfy
	= 
	\int\limits_\Omega [ u(\bfy) - u(\bfx) ] \gamma(\bfx,\bfy) \dif \bfy.
\end{align*}
Notice that, for the previous operators to be well defined, we require that at least \( \gamma \in L^\infty (\Omega,\Omega)\) and \( u \in L^2(\Omega)\).

Many properties from classical calculus extend to this nonlocal extension of the derivative. In particular, we have that: (a) the adjoint of the nonlocal gradient is the negative of the nonlocal divergence (i.e., \( \nabla_\gamma^* = -\dive_\gamma \)), (b) a zero flux property (divergence theorem: \( \langle \dive_\gamma v, 1 \rangle = 0 \)), and (c) the self--adjoint property for the nonlocal Laplacian (i.e., \(\Delta_\gamma^* = \Delta_\gamma\)). See \cite[\S 12]{AMT_Th_2024} for detailed proofs.

% ----------------------------------------------------- %
\subsection{Nonlocal denoising}

Originally inspired by convolution filters, nonlocal methods predict the value of a pixel based on generalized neighborhoods of the same image that represent or approximate the local behavior of the pixel. In this class of methods, the nonlocal means (NLM) filter introduced by \cite{Buades2005b} has proven to be successful, yielding several variants \cite{Buades2010} as the BM3D filter \cite{Dabov2007}. Inspired by the effectiveness of the NLM, \cite{Gilboa2007} proposed an effective and general variational framework for nonlocal operators as the action of the nonlocal regularizing functional
\[
    R(u) \coloneqq \frac{1}{4} \iint\limits_{\Omega \times \Omega} \big[ u(\bfy) - u(\bfx) \big]^2 \gamma(\bfx,\bfy) \dif \bfy \dif \bfx 
    \, = \frac{1}{2} \langle -\Delta_\gamma  u, u \rangle
    ,
\]
where \( \Omega \), known as the image domain, is a subset of \(\R^{d_1}\), and the weight function \(\gamma\) is a non--negative and symmetric scalar field for any \( (\bfx,\bfy) \in \Omega\times\Omega\). The choice of \(d_1\) is application--driven, commonly based on structured data; for instance, \(d_1\) will take the value of \(1\) for flat signals, \(2\) for static images, \(3\) for movies or animations, and higher values are used for hyperspectral imaging. We will focus on the case \( d_1 = 2\), but our analysis can be generalized to the other named cases. For image processing tasks the weight function is based on image features and can be understood as the \emph{similarity} or \emph{proximity} between two points \(\bfx\) and \(\bfy\), based on features in their neighborhood.

\begin{proposition}\label{prop:All_The_Nice_Properties_of_NL_Energy}
	Let \(\gamma \in L^\infty (\Omega\times\Omega)\), then 
	the functional \(R:L^2(\Omega) \to \R\) is differentiable and convex in \(L^2(\Omega)\), and its derivative is continuous and self--adjoint.
\end{proposition}

A direct proof can be found in \cite[\S 12]{AMT_Th_2024}. By construction, the functional \(R\) defines an energy seminorm in \(L^2(\Omega)\) induced by the positive semi--definite Hermitian form
\begin{equation}\label{eq:Nonlocal_Dot_Product}
    R'(u)v = \langle -\Delta_\gamma  u, v \rangle = \iint\limits_{\Omega\times \Omega} v(\bfx) \big[ u(\bfx) - u(\bfy) \big] \gamma(\bfx, \bfy) \dif \bfy \dif \bfx
    .
\end{equation}
We will use the notational shorthand \( \mathcal{L} \coloneqq \Delta_\gamma \) for brevity. The operator \(-\mathcal{L}\) is the continuous version of a graph Laplacian, and generalizes the classical Laplacian \cite{Gilboa2007,Gilboa2009}.

Inspired by the classical ROF model, and based on the maximum a posteriori probability estimate criterion \cite[\S4.5.4]{Chan2005}, the nonlocal denoising problem can be posed as the minimization of an energy consisting of the combination of a regularizing term and a fidelity term \cite{Gilboa2007}:
\begin{equation}
\label{eq:Min_Energy}
    \argmin_{u \in V} \mathcal{E}(u;\lambda) \coloneqq \mu R(u) + \frac{\lambda}{2} \| u - f \|^2_{L^2(\Omega)},
\end{equation}
where \(V\) is a suitable closed subspace of \(L^2 (\R^2)\), and \(\mu,\lambda\) are positive. The quotient \(\nicefrac{\lambda}{\mu}\) balances the fidelity term against the nonlocal regularizer. In principle, the regularizer weight can be omitted, yet it plays an essential rôle for numerical computations as the values of \(R\) can be large. Furthermore, the value of the fidelity weight determines how much of the noisy sample \(f\) is preserved when minimizing \(\mathcal{E}\) with respect to \(u\). Thus, finding an optimal value of \(\lambda\) in \cref{eq:Min_Energy} is crucial if we want to implement nonlocal denoising inside the scope of large--scale image recovery algorithms.

% ----------------------------------------------------- %
\subsection{Parameter learning via bilevel optimization}

Bilevel optimization has played an important rôle in image recovery and parameter calibration \cite{CarlosDelosReyes2013,Calatroni2017,DelosReyes2021,DelosReyes2023}. In the context of nonlocal denoising, this was first showcased in \cite{D’Elia2021} for optimizing the fidelity weight \(\lambda\) and kernel selection. 
If a clean and noise--less version of \(f\), labeled \( u_{\textsf c}\),  is available, then an optimal fidelity weight can be trained by solving
\begin{subequations}\label{eq:BilevelProblem_Continuous}
\begin{align}
    &\qquad \min_{u \in V, \lambda \in \Lambda} J(u;\lambda) \coloneqq \frac{1}{2} \| u_{\textsf c} - u \|^2_{L^2(\Omega)}
    \\[-0.5em]
    \text{subject to} &  \notag
    \\
    & u = \argmin_{v \in V} \mathcal{E}(v;\lambda) \in \{ v\in V:\, {-}\mu \mathcal{L} v + \lambda v = \lambda f \},      \label{eq:Euler-Lagrange}
    \\
    & \lambda \in \Lambda \coloneqq [\Lambda_{\min}, \Lambda_{\max}].
\end{align}
\end{subequations}
The linear equation in \cref{eq:Euler-Lagrange} is just the Euler--Lagrange equation of \cref{eq:Min_Energy}, where we have used \cref{eq:Nonlocal_Dot_Product}. A solution \(u = \argmin\limits_{v \in V} \mathcal{E}(v;\lambda) \) is the best reconstruction of the image for a given regularization level \(\lambda\) and a given regularization kernel embedded through \(\mathcal{L}\). The choice of \(\lambda\) is then influenced by comparing \(u\) with the ground truth \( u_{\textsf c}\).

It turns out that if \(0 < \Lambda_{\min} \leq \Lambda_{\max}\), then the solution set of the lower--level problem \cref{eq:Euler-Lagrange} contains a single point. Furthermore, depending on the choice of \(V\) and \(\gamma\), we can even extend \(\Lambda\) to include the zero fidelity case, which we do not cover here. For simplicity, let us consider \( V = L^2(\Omega) \) and kernels with support fully contained in \(\Omega\), thus we also fix \( \Lambda_{\min} > 0\).

\begin{theorem}\label{th:Uniqueness_LLP}
	For every \(\lambda \in \Lambda\), there exists a unique solution to \cref{eq:Euler-Lagrange}. Moreover, if \(f\) is bounded, it holds that \( \min\limits_{\bfx \in \Omega} f(\bfx) \leq u \leq \max\limits_{\bfx \in \Omega} f(\bfx)\).
\end{theorem}

The result follows from a minimizing sequence argument and the maximum principle for the elliptic operator \(-\mathcal{L}\) \cite[\S 12]{AMT_Th_2024}. 
From here, we can replace \cref{eq:Euler-Lagrange} with the constraint of the Euler--Lagrange equation, thus revealing that \cref{eq:BilevelProblem_Continuous} is nothing else than a nonlocal optimal control problem. Using a minimizing sequence along the Bolzano--Weierstrass theorem for \(\Lambda\), we obtain the following result:

\begin{theorem}
	The bilevel optimization problem \cref{eq:BilevelProblem_Continuous} admits a solution \((\lambda,u) \in \Lambda \times V\) satisfying the optimality system:
	\begin{equation}\label{eq:Optimality_System_Regularization}
	\left\{
		\begin{aligned}
		-\mu \mathcal{L} u + \lambda u &= \lambda f	&&\quad\emph{in } \Omega,
		\\
		-\mu \mathcal{L} p + \lambda p &= u - u_{\mathsf c}	&&\quad\emph{in } \Omega,
		\\
		\proj_{\Lambda} \big( \lambda + c\langle u-f, p \rangle \big) &= \lambda	&&\quad\forall c > 0.
		\end{aligned}
	\right.
\end{equation}
\end{theorem}

Moreover, if the kernel \(\gamma = \gamma_\theta\) is parameterized by a vector \(\theta\) in a convex and compact set \(\Theta\), then it can be proven \cite{D’Elia2021} that the extended problem
\begin{equation}\label{eq:biparametric_problem_control}
	\min_{u \in V, \lambda \in \Lambda, \theta \in \Theta} J(u;\lambda, \theta) \coloneqq \frac{1}{2} \| u_{\textsf c} - u \|^2_{L^2(\Omega)}
	\qquad
	\text{subject to}
	\qquad
	u = \argmin_{v \in V} \mathcal{E}_\theta (v;\lambda).
\end{equation}
also admits a solution which satisfies \cref{eq:Optimality_System_Regularization} with the additional constraint 
\(
	\proj_{\Theta} \del{ \theta - c \pd{ }{\theta} \langle p, \mathcal{L}_\theta u \rangle } = \theta	
\) for all \( c > 0\).
Moreover, we have the gradient
\[
	\nabla_{\lambda,\theta} \, J (u; \lambda, \theta) = 
	\begin{bmatrix}
		\langle f-u, p \rangle		& \mu \pd{ }{\theta} \langle p, \mathcal{L}_\theta u \rangle
	\end{bmatrix}^\top   \hspace{-0.3em} . 
\]
For example, if \(\theta\) is a scalar inducing a kernel \(\gamma_\theta\), as in the case of the shape parameter in the extended Gaussian ANOVA kernel \cref{eq:Original_ANOVA}, then 
\[
	\mu \pd{ }{\theta} \langle p, -\mathcal{L}_\theta u \rangle
	=
	\mu
	\int\limits_\Omega p(\bfx) 
	\sbr[4]{
	u(\bfx)
	\int\limits_\Omega \tpd{  \gamma_\theta}{\theta}  (\bfx,\bfy) \dif \bfy
	-
	\int\limits_\Omega u(\bfy) \tpd{  \gamma_\theta}{\theta} \dif \bfy
	}
	\dif \bfx .
\]

\begin{remark}
    Note that our discussion has focused on single--channel images. Notwithstanding, this is not a limitation, as the framework can be readily adapted to multichannel signals through a mixing technique. For further details on such approaches, see \cite{Gilboa2007,Buades2009,Buades2008}. Additionally, the review by \cite{Dai2013} provides additional insights on colour mixtures.
\end{remark}

% ----------------------------------------------------- %
\subsection{The choice of weights}\label{ssec:Weights}

The rôle of \(\gamma\), also known as the \emph{nonlocal kernel}, is key to determining the behavior of the regularizing term. This function often incorporates information from the image domain and the features that are relevant in the image (e.g., textures, periodicity) and thus guides the optimization process. 

\begin{wrapfigure}[13]{r}{0.36\textwidth} 
    \centering
    \vspace{-1em}
    \includegraphics[]{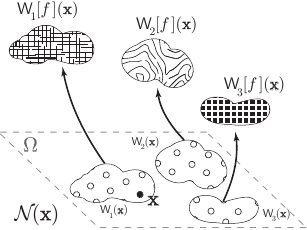}
    \caption{Representation of a neighborhood of features for a point \(\bfx \in \Omega\).}
    \label{fig:Neighbourhood_Features}
\end{wrapfigure}

Inspired by neighborhood filters, one can define the kernel weights based on \emph{affinity} functions. The basic affinity structure is of similarity between image features present already in the noisy sample \( f \in L^2(\Omega)\). To every point \(\bfx\), we can assign a \emph{neighborhood of features} \( \mathcal{N}(\bfx)\) which, in a broad sense, can consist of a finite set of \emph{windows} that can have different sizes and shapes to better adapt to the image; see \cref{fig:Neighbourhood_Features}. Formally, we can write \( \mathcal{N}(\bfx) = \{ \mathsf{W}_\ell (\bfx) \subseteq \Omega \}_{\ell \in \sf{L}}  \) for some \( \sfL \in \N\). Here, each index \(\ell\) can represent a region of interest in the image domain.
The map \(\mathcal{N}: \Omega \to (2^\Omega)^\sfL \) does not necessarily need to define windows that contain every pixel as we can take advantage of the high degree of redundancy of any natural image. This is an extension of the framework of a neighborhood system, which has shown that as long as representative values are assigned to each pixel, then all pixels in the neighborhood can be used to predict the value of such a pixel \cite{Efros1999,Buades2010}. This idea has inspired the design of patch--based priors; see for instance \cite{Piening2024,Altekrueger2023,Zoran2011}. Since our particular prior is given in the form of the linear operator \(-\mathcal{L}\), we focus next on describing several kernel--based filters.

A neighborhood of features allows us to define the following composite similarity kernel:
\begin{equation}
    \label{eq:general_kernel}
    \gamma(\bfx, \bfy) = 
    \begin{cases}
        \sum\limits_{\ell = 1}^{\sfL} \alpha_\ell \exp\Big\{ {-}\sigma^{-1} \mathrm{dist}\big( \mathsf{W}_\ell [f](\bfx), \mathsf{W}_\ell [f](\bfy) \big)  \Big\}^{2}
        &\text{if } (\bfx,\bfy) \in \mathrm{supp}(\gamma),
        \\
        0
        &\text{otherwise},
    \end{cases}
\end{equation}
where \( \{\alpha_\ell\}_{\ell \in \llb 1,\sfL\rrb} \) is a family of positive weights, the \emph{shape} or \emph{filtering} parameter \( \sigma > 0\) controls how much information from the distance between \( \mathsf{W}_\ell [f] (\bfx)\) and \( \mathsf{W}_\ell[f](\bfy)\) is retained, and \( \mathsf{W}_\ell [f] (\bfx) = \{ w[f](\mathbf{z}): \, \mathbf{z} \in \mathsf{W}_\ell (\bfx)  \} \) with \(w\) some vector field that may or not depend on \(\ell\) for \( \ell \in \llb 1,\sfL\rrb\). The formalism \cref{eq:general_kernel} extends other classical kernels like the Sigma or Yaroslavsky \& Lee kernel (\(\sfL = 1\), \( \mathsf{W}_1(\bfx) = \{\bfx\}\), \( \alpha_1 = 1\), \( w[f] = f\), and \( \mathrm{dist} = |\cdot|\)) \cite{Lee1983,Yaroslavsky1985}, the SUSAN kernel (\(\sfL = 1\), \( \mathsf{W}_1(\bfx) = \{\bfx\}\), \( \alpha_1 = 1\), \( w[f](\bfx) = \begin{bsmallmatrix} f(\bfx) & (\nicefrac \sigma \rho)\bfx \end{bsmallmatrix} \), and \( \mathrm{dist} = \|\cdot\|_2\)) \cite{Smith1997}, and the NLM kernel (\(\sfL = 1\), \( \mathsf{W}_1(\bfx) = B(\bfx; \rho)\), \( \alpha_1 = 1\), \( w[f] = f \), and \( \mathrm{dist} = \|\cdot\|_{L^2_{G_a}(\R)}\)) \cite{Buades2005,Buades2005b}. In this work, we are interested in the unnormalized extended Gaussian ANOVA kernel \cite{Nestler2022} which obtains a composite similarity from windows of fixed size. To be precise, fix \(\sfL \geq 1\) and for each \( \ell \in \llb 1,\sfL \rrb\) define \( \alpha_\ell = \sfL^{-1}\), take \( \mathsf{W}_\ell(\bfx) \) as a discrete subset of \(\Omega\) of size \(d_{2,\ell}\), and define \( w[f] = f\). As a result, for any \((\bfx,\bfy) \in \Omega \times \Omega\), we have
\begin{equation}
    \label{eq:continuous_ANOVA}
    \gamma(\bfx,\bfy) = \frac{1}{\sfL} \sum_{\ell=1}^\sfL \exp \Big\{ {-}\sigma^{-2} \big\| \mathsf{W}_\ell[f](\bfx) - \mathsf{W}_\ell[f](\bfy)  \big\|_2^2 \Big\}.
\end{equation}
In \cite{Nestler2022}, \cref{eq:continuous_ANOVA} was introduced to overcome the computational complexity of computing and storing neighborhood kernels over high--dimensional sets of features. This task can be achieved efficiently for windows of size \( d_{2,\ell} \leq 3\) using the NFFT, which we introduce in \cref{sec:NFFT_Gauss}. The core principle lies in transferring the cost of computing the distance between large--sized windows with local information to instead compute additional terms in the composite kernel with relevant information.

% ----------------------------------------------------- %
\subsection{Discretization}

Since we are interested in finding a solution that can display discontinuities, following the approach in \cite{D’Elia2021}, we use a piece--wise constant finite element basis to approximate \(u\) in \(L^2(\Omega)\). Specifically, consider the partition of the image domain \(\Omega\) in the natural pixel grid of \(n_1 \times n_2\) single--pixel squares. For each square we associate a unit step function, the indicator of the pixel. Let us denote the induced finite element method (FEM) functional space by \(V^h\). Here, collect \(\bfu^h\) and \(\bff^h \in \R^n\) as the discretizations of \(u\) and \(f\) in \(V^h\), where \(n \coloneqq n_1 n_2\) is the computational dimension. Let \( \{\bfx_i^h\}_{i\in \llb 1,n\rrb} \subset \Omega \) denote the elements (pixels) in the image domain. Moreover, denote \( \gamma^h_{i,j} \) as the FEM approximation at the pixels \( \bfx_i^h\) and \( \bfx_j^h\), which we collect in the matrix \(\Gamma^h = (\gamma_{i,j}^h)_{(i,j)}\). Then the nonlocal product \( -\mathcal{L} u \) is discretized as
\(
    \big[ \pi_{V^h} (-\mathcal{L} u) \big]_i 
    = 
    u_i \sum\limits_{j=1}^n \gamma^h_{i,j} - \sum\limits_{j=1}^n \gamma^h_{i,j} u_j
\)
for each \( i \in \llb 1,n\rrb\). Let us denote \( \etab^h \coloneqq \Gamma^h \Ones\), then the nonlocal transformation can be represented as \( L^h \coloneqq \diag( \etab^h ) - \Gamma^h\) (notice the sign change). Here, \(L^h\) is nothing else than a graph Laplacian associated to a graph with weights given by \(\Gamma^h\). Then, the lower--level optimality condition \cref{eq:Euler-Lagrange} can be written as
\begin{equation}
\label{eq:discretized-state-h}
	A^h \bfu^h = \big( \lambda \Idn + \mu ( \diag(\etab^h ) - \Gamma^h) \big) \bfu^h = \lambda  \mathbf{f}^h, 
\end{equation}
which we will refer to as the \emph{nonlocal system}. Due to the nature of a graph Laplacian, the system is ill--conditioned for small values of \(\lambda\). We will develop a method for effectively preconditioning \cref{eq:discretized-state-h} in \cref{sec:PreconditioningUnderChangeOfBasis}.

The FEM discretization of \(\gamma\) yields the discrete unnormalized extended Gaussian ANOVA kernel
\begin{equation}
    \label{eq:discrete_ANOVA}
    \gamma^h_{i,j} = \frac{1}{\sfL} \sum_{\ell=1}^\sfL \exp \Big\{ {-}\sigma^{-2} \big\| (\mathsf{W}_{\ell} [\bff^h])_i - (\mathsf{W}_{\ell} [\bff^h])_j  \big\|_2^2 \Big\}.
\end{equation}
Often we will refer to \(\Gamma^h\) just as the ANOVA or the nonlocal kernel and each summand in \cref{eq:discrete_ANOVA} as a \emph{subkernel}. At this point, notice that if \(\gamma^h_{i,i} \neq 0\), then \( (L^h)_{i,i} = \sum\limits_{j=1}^n \gamma^h_{i,j} - \gamma_{i,i} = \sum\limits_{j\neq i} \gamma_{i,j}^h\), thus the diagonal of \(\Gamma^h\) does not contribute to the reconstruction process. As a result, from now on, we will assume that the diagonal of \( \Gamma^h\) only contains zero entries. Moreover, we will drop the super--index \(h\) keeping in mind that all the bold or capitalized quantities are discrete approximations in \(V^h\). We delay the construction of the similarity windows \(\mathsf{W}_\ell\) to \cref{sec:Numerics}.

% ------------------------------------------------------------------------------------------- %
% ------------------------------------------------------------------------------------------- %
% ------------------------------------------------------------------------------------------- %
\section{NFFT--based Fast Gauss Transform}\label{sec:NFFT_Gauss}

In this section, we provide a brief overview of the Nonequispaced Fast Fourier Transform and the fast summation method for the Gauss Transform. The NFFT is an extension of the classical FFT algorithm \cite{Dutt1993,Duijndam1999,Steidl1998,Kunis2006} to evaluate the trigonometric polynomial
\begin{equation}
    \label{eq:NFFT}
    f(\bfx) \coloneqq \sum_{\mathbf{k} \in \mathbf{I}_\mathbf{M}} \widehat{f}_\mathbf{k} e^{-2\pi i \, \mathbf{k}^\top \bfx}
\end{equation}
at the nonequispaced (i.e., arbitrary, often irregularly distributed) nodes \( \{\bfx_j\}_{j \in \llb 0,N-1\rrb} \subset \mathbb{T}^d\) from a finite number of Fourier coefficients \( \{ \widehat{f}_{\mathbf{k}} \}_{ \mathbf{k} \in I_{M} } \subset \C\), where \(\mathbf{M} \in 2 \mathbb{N}^{d} \) and  \( \mathbf{I_{M}} \coloneqq \prod\limits_{s=1}^d \{-\nicefrac{m_s}{2}, \ldots, \nicefrac{m_s}{2} -1 \}\). The Torus is defined as \( \mathbb{T} \coloneqq \R \pmod \Z \simeq [ -\nicefrac{1}{2},\nicefrac{1}{2} )\).  The adjoint NFFT instead interchanges the rôles of the Fourier coefficients and node evaluation in \cref{eq:NFFT} with a sign change in the argument of each exponential term.

The NFFT algorithm makes use of an oversampled FFT and the convolution theorem to yield an approximation of \cref{eq:NFFT} with arithmetic complexity \(\mathcal{O}( |\mathbf{I}_\mathbf{M}| \log |\mathbf{I}_\mathbf{M}| + N ) \) \cite{Kunis2008}. Notice that the computation time of the NFFT increases considerably with \(d\) due to the exponential relationship \( \min \mathbf{M}^d \leq |\mathbf{I}_\mathbf{M}|\). Moreover, the NFFT algorithm deliberately incorporates a systematic error into its calculations to maintain its advantageous computational efficiency. This extra error is manageable and, if required, can be minimized to the level of machine precision \cite{Keiner2009}. For more information on the algorithm and the controlling parameters, we refer to \cite{Keiner2009,Keiner2007,Kunis2006}.

The discrete Gauss Transform computes sums of the form
\begin{equation}\label{eq:DGT}
    g(\bfx) \coloneqq \sum_{j = 0}^{N_y - 1} \alpha_j e^{-\hat{\sigma} \| \bfx - \bfy_j \|^2_2},
\end{equation}
for a given shape parameter \(\hat{\sigma} \in \C\) with \( \Re(\hat{\sigma}) > 0\) and the sets of coefficients \( \{\alpha_j\}_{j\in \llb 1, N_y\rrb}\), source nodes \( \{\bfy_j\}_{j\in \llb 1, N_y\rrb} \subset \R^d\), and target nodes \( \{\bfx_j\}_{j\in \llb 1, N_x\rrb} \subset \R^d\).  In \cite{Kunis2006a}, the Fast Gauss Transform (FGT) was presented as a fast summation method to approximate \cref{eq:DGT} regardless of the distribution of the source and target nodes. The method is based on the application of an adjoint NFFT to \(\alpha\),  followed by a product with Fourier coefficients related to a truncation of a periodic approximation of the shape function \( e^{-(\nicefrac{r}{\hat\sigma})^2} \), and a forward NFFT \cite{Potts2003,Potts2004}. The complexity of the method is \( \mathcal{O}( N_x + N_y + N_d \log N_d )\), where \(N_d\) is the expansion degree of the Fourier coefficients in the NFFT. The value of \(N_d\) is only related to the expected accuracy of the method and is independent of \(N_x\) and \(N_y\); thus, the FGT behaves as \( \mathcal{O}(N_x + N_y)\) \cite[\S4]{Kunis2006}. Notwithstanding, the associated cost of the constant in this complexity computation suffers from the curse of dimensionality. As a result, the FGT is an efficient algorithm for small--sized feature size \(d\). 

Recalling the ANOVA kernel \cref{eq:discrete_ANOVA}, the FGT can be used to efficiently approximate the action of \(\Gamma\) on a vector \(\bfu \in \R^n\) by successive application of \(\sfL\) FGTs. Due to the curse of dimensionality, the authors in \cite{Nestler2022} recommend the use of \(\sfL\) windows of size less than four, where each window samples the feature dimension, thus overcoming the computational limitation of the expansion degree. As a result, the computational complexity of evaluating \(\Gamma\bfu\) is \( \mathcal{O}(\sfL n)\). Error estimates on the approximation are provided in \cite{Potts2004,Kunis2006a,Kunis2006}, where parameter studies are performed.

% ------------------------------------------------------------------------------------------- %
% ------------------------------------------------------------------------------------------- %
% ------------------------------------------------------------------------------------------- %
\section{Preconditioning under a new basis}\label{sec:PreconditioningUnderChangeOfBasis}

Let us motivate this section by considering an ill--con\-dition\-ed system of linear equations, to be solved by an iterative method, including Krylov subspace solvers such as the Conjugate Gradient (CG) \cite{cg} or Minimum Residual (MINRES) \cite{minres} methods. We emphasize that our approach may be embedded within any such solver which requires a fixed (low) number of matrix--vector applications. It is well-known that even if the condition number is large, these methods can perform well if the eigenvalues are clustered in few and small intervals \cite{Kelley1995,Greenbaum1997}. The smaller and further away the clusters are from the origin and each other, the better \cite{Wathen_2015,Alger2019}. Often, if we wish to obtain clustered eigenvalues, we require the use of a \emph{preconditioner}. Here, we are interested in the case where a matrix has an eigenvalue too close to the origin, such that this eigenvalue prevents an `ideal' clustering of the eigenvalues of the preconditioned system because (a) it drags down the remainder of the eigenvalues, or (b) it results in the smallest eigenvalue being lower for the preconditioned system than the original system. This occurs when preconditioning the nonlocal system \cref{eq:discretized-state-h}. In what follows, we propose a methodology to resolve this issue by isolating the smallest eigenvalue of a matrix without modifying the rest of the spectrum. We introduce and derive explicit transformations which can be easily implemented within any iterative framework, and show their computational complexity to be linear in time, hence extending their applicability for large and dense linear systems.

% ----------------------------------------------------- %
\subsection{Change of basis}\label{sec:change_of_basis}

Let \( (\lambda, \bfv) \) be an eigenpair of a Hermitian (or real symmetric) matrix \( A\in \mathcal{M}_n(\F)\)\footnote{%
Most of the following results can be generalized to the non--Hermitian case without loss of generality. However, we focus on the Hermitian case for the relationship between the spectrum of \(A\) and its condition number.}. By definition, we know that \( \dim \bfv^\perp = n-1\); as a result, we can build a matrix \( Q = (\bfq_i)_{i \in \llb 1,n \rrb}\) such that \( \bfq_1 = \bfv\) and \( \{\bfq_i\}_{i \in \llb 2,n \rrb} \) is a basis in \( \bfv^\perp \). Due to linear independence, the matrix \(Q\) is invertible. As a result, we can conjugate \(A\) by the change of basis given by \(Q\) and define \( \UnitSim[Q]{A} \coloneqq Q^{-1} A\, Q \). 

The matrix \( \UnitSim[Q]{A} \) has a particular structure that is useful to solve ill--conditioned systems. To see this, let us start by finding its first column through analyzing the product \( \UnitSim[Q]{A} \bfe_1\). We know that \( Q \bfe_1 = \bfv\), and since this is an eigenvector of \(A\), we have that \( AQ\bfe_1 = \lambda \bfv\). Now, since \( \bfv = \bfq_1\), it holds that \( Q^{-1} \bfv = \bfe_1\), thence \( \UnitSim[Q]{A} \bfe_1 = \lambda \bfe_1\). As a result, we can assert that \( Q \) is a change of basis that maps the eigenvector \(\bfv\) to \( \bfe_1\). Notice that the first row of \( \UnitSim[Q]{A} \) is determined by the first row of \( Q^{-1} \), labeled \( \mathbf{r}_1\), in the following fashion:
\begin{align}
    (\UnitSim[Q]{A})_{1,1:n} = 
    \mathbf{r}_1 A Q &= 
    \begin{bmatrix}
        \langle \lambda \bfv, \bfr_1^\dagger \rangle & \langle A\bfq_2, \bfr_1^\dagger \rangle & \cdots & \langle A\bfv_n, \bfr_1^\dagger \rangle
    \end{bmatrix} 
    \label{eq:FirstRowOfTransformedSystem}
    = 
    \begin{bmatrix}
        \lambda & \langle A\bfq_2, \bfr_1^\dagger \rangle & \cdots & \langle A\bfv_n, \bfr_1^\dagger \rangle
    \end{bmatrix}.
\end{align}
Let us define \( \mathbf{b}_1 \coloneqq \big( \langle A\bfq_i, \bfr_1^\dagger \rangle \big)_{i \in \llb 2,n\rrb } \) and \(B_2 \coloneq (\UnitSim[Q]{A})_{2:n,2:n} \).
We can thus write \(\UnitSim[Q]{A}\) as the block matrix
\begin{equation}
\label{eq:GeneralChangeOfBasisOfA}
	\UnitSim[Q]{A} =
	\left[
	\begin{array}{@{} c ;{2pt/2pt} c c c c c c c}
		 \lambda & \mathbf{b}_1
		\\[0.1em]
		\cdashline{1-2}[2pt/2pt]
		\\[-0.9em]
		\mathbf{0}_{n-1} &  B_2
	\end{array}
	\right].
\end{equation}
Observe that since \(\bfe_1\) is an eigenvector of \(\UnitSim[Q]{A}\), then \( \Sigma (B_2) = \Sigma(A) \setminus \{\lambda\}\) when \(\lambda\) is simple. As a result, the block structure of \cref{eq:GeneralChangeOfBasisOfA} is useful when we try to solve systems of equations of the form \( A \bfu = \bff\) that suffer from ill--conditioning (either by having a very small minimum eigenvalue or a very large spectral radius). To see this, notice that we can select \( \lambda\) accordingly and solve the equivalent system \( \UnitSim[Q]{A} \bfx = Q^{-1} \bff = \bfg\). We can easily recover a solution in the original basis by the transformation \( \bfu = Q \bfx\). By construction, if we let \( \bfu = (u_1 \,\, \bfu_{2:n}^\top)^\top\) and \( \bfg = (g_1 \,\, \bfg_{2:n}^\top)^\top\), then \( \bfu_{2:n} = B_2^{-1} \bfg_{2:n}\) but also \( u_1 = \lambda^{-1} [ g_1 - \mathbf{b}_1 \bfu_{2:n} ]\). The system \( B_2 \bfu_{2:n} = \bfg_{2:n}\) can be solved using either a direct or iterative method.

It turns out that since \(A\) is Hermitian, then \( \mathbf{b}_1\) is always a vector of zeros. To see this, recall that \( \langle \bfv , \bfr_1^\dagger\rangle = 1\), and the orthogonality of \( \bfv^\perp\) implies that we can only have \( \bfr_1 = \|\bfv\|^{-1} \bfv^\dagger\).
Now, for any index \( i \in \llb 2,n\rrb\), it holds that
\[
	\langle A\bfq_i, \bfr_1^\dagger \rangle =  \langle \bfq_i, A \bfr_1^\dagger \rangle 
	= \frac{1}{\|\bfv\|} \langle \bfq_i, A \bfv\rangle = \frac{\lambda}{\|\bfv\|} \langle \bfq_i , \bfv \rangle = 0.
\]

A particular case of interest is when \(Q\) is unitary or orthogonal  whenever \( \F = \R\). Under these assumptions, we obtain that \( \UnitSim[Q]{A}\) and \(B_2\) are then Hermitian (symmetric). 
The unitary case is a special case of \emph{deflation} and is used for solving eigenvalue problems \cite[Lemma 7.1.2]{Golub2013}. In particular, if we find a matrix \(Q\) that is Hermitian (symmetric), we can induce a deflation under an orthogonalization process which simplifies due to the additional structure.
Moreover, the condition number of \(B_2\) is just
\[ 
\kappa (B_2) = 
\begin{cases}
    \left| \frac{\rho(A)}{a(A)} \right| & \text{if } \lambda = \min \big| \Sigma(A) \big|,
    \\[0.5em]
    \frac{\max |\Sigma(B_2)|}{\min |\Sigma(A)|} & \text{if } \lambda = \max \big| \Sigma(A) \big|,
\end{cases}
\]
which can result in better--conditioned systems, depending on the distribution of the spectrum of \(A\). A basic requirement is for \(\lambda\) to be simple.

%%%%%%%%%%%%%%%%%%%%%%%%%%%%%%
\subsubsection{An explicit transformation}
A relevant question might arise at this point: \emph{can we define a general sequence \( \{\bfq_i\}_{i \in \llb 2,n \rrb} \) such that \(Q\) is Hermitian?}
To affirmatively answer this question, assume \( v_1 \) to be real and nonzero. We then propose the following sequence\footnote{The constant \(v_1\) used in \( \bfq_i = -v_1 \bfe_i\) for the case \( v_i = 0\) is optional, as we can select any nontrivial real vector in \(\spn \{\bfe_i\}\).}:
\begin{equation}
\label{eq:Hermitian_Q_Construction}
    \bfq_i \coloneqq
    \overline{v}_i \bfe_1 - v_1 \bfe_i 
    =
    \begin{cases}
        \overline{v}_i \bfe_1 - v_1 \bfe_i & \text{if } v_i \neq 0,
        \\
        -v_1 \bfe_i & \text{otherwise}, 
    \end{cases}
    \qquad \forall i \in \llb 2,n\rrb.
\end{equation}
If \(v_1\) has a nonzero imaginary part, we can instead proceed with the eigenvector \( \bfu \coloneqq \overline{v}_1 \bfv \), which has a real first entry. Moreover, in the case where \(v_1 = 0\), we can just replace \( v_1 \) in \cref{eq:Hermitian_Q_Construction} with \( v_\imath\), where \({\imath = \min\limits_{v_i \neq 0} i} \).

\begin{lemma}
    The matrix \( Q = (\bfq_i)_{i \in \llb 1,n \rrb} \) is Hermitian and invertible.
\end{lemma}
\begin{proof}
    To see this, observe that the matrix \(Q\) is sparse as it only has nonzero values on its diagonal and its first row and column. In particular, the diagonal values are real due to our choice of \(v_1\) and the fact that the canonical basis vectors are real. For the remaining off--diagonal values, we have that \( \bfr_1 = \bfv^\dagger \) by definition. Hence \(Q\) is Hermitian and can be written as the \emph{bordered} or \emph{block} matrix
    \begin{align}
    \label{eq:Q_blocks_structure}
	Q = 
    \left[
	\begin{array}{@{} c ;{2pt/2pt} c c c c @{}}
        v_1
        &  \bfv^\dagger_{2:n} 
		\\[0.1em]
        \cdashline{1-2}[2pt/2pt]
        \multirow{2}{*}{  \vspace{1em}  $\bfv_{2:n}$     }
        \\[-1em]
        & -v_1 \Idn[n-1]
	\end{array}
    \right]
    .
	\end{align}

    To show invertibility, consider a linear combination of the following form (with constants $\alpha_i$):
    \[
        0 = \sum_{i=2}^n \alpha_i \bfq_i = \sum_{i=2}^n \alpha_i ( \beta_{1,i} \bfe_1 + \beta_{2,i} \bfe_i) \eqqcolon \alpha_1 \bfe_1 + \sum_{i=2}^n (\alpha_i \beta_{2,i}) \bfe_i.
    \]
    As the values \( \beta_{2,i}\) are nonzero, we know that by linear independence of the canonical basis that the set \( \{\bfq_i\}_{i \in \llb 2,n \rrb} \) is linearly independent. Moreover, we have that by orthogonality all columns in \(Q\) are linearly independent, thus \(Q\) is invertible. Another way to see this is by the elegant formula \cite{Silvester2000}:
    \begin{equation}
    \label{eq:detQ}
        \det(Q) = 
        \det\Big( v_1 - \bfv_{2:n}^\dagger 
        ( -v_1 \Idn[n-1] )^{-1}
        \bfv_{2:n} \Big) 
        \det( -v_1 \Idn[n-1] ).
    \end{equation}
    Clearly \( \det( -v_1 \Idn[n-1] ) = (-v_1)^{n-1}\), and the first multiplicand in \cref{eq:detQ} is just \( v_1 + v_1^{-1} \| \bfv_{2:n} \|^2\) and none of these terms are zero as \(v_1\) is nonzero and real.
\end{proof}

We can furthermore explicitly compute the inverse of \(Q\) by means of the Schur complement. Letting \( D \coloneqq  -v_1 \Idn[n-1]\), we have that
\[
    Q/D \coloneqq v_1 - \bfv_{2:n}^\dagger D^{-1} \bfv_{2:n} =  v_1 + v_1^{-1} \| \bfv_{2:n} \|^2.
\]
Thence\footnote{Notice that for the case \( v_1 = 0 \), a similar sparse diagonal--block structure as in \cref{eq:Q_blocks_structure} can be derived. The first diagonal block would be \( -v_{\imath} \Idn[\imath-1]\), while the other block is just \cref{eq:Q_blocks_structure} indexed for \({\imath}\) instead. A similar block structure can be found for the inverse of the block diagonal operator that resembles \cref{eq:Q_Inverse} with an additional direct sum.
} we have the inverse block matrix \cite{Taboga2021}:
\begin{align}
    Q^{-1} &=
\left[
\begin{array}{@{} c ;{2pt/2pt} c c c c @{}}
    (Q/D)^{-1}
    &  -(Q/D)^{-1}  \bfv^\dagger_{2:n} D^{-1}
    \\[0.1em]
    \cdashline{1-2}[2pt/2pt]
    \multirow{2}{*}{  \vspace{1em}  $-D^{-1} \bfv_{2:n} (Q/D)^{-1}$     }
    \\[-1em]
    & D^{-1} + D^{-1} \bfv_{2:n} (Q/D)^{-1} \bfv_{2:n}^\dagger D^{-1}
\end{array}
\right]
\notag
\\
&= (Q/D)^{-1}
\left[
\begin{array}{@{} c ;{2pt/2pt} c c c c @{}}
    1
    & -\bfv^\dagger_{2:n} D^{-1}
    \\[0.1em]
    \cdashline{1-2}[2pt/2pt]
    \multirow{2}{*}{  \vspace{1em}  $-D^{-1} \bfv_{2:n} $     }
    \\[-1em]
    & (Q/D) D^{-1} + D^{-1} \bfv_{2:n}  \bfv_{2:n}^\dagger D^{-1}
\end{array}
\right]
\notag
\\
&= (Q/D)^{-1}
\left[
\begin{array}{@{} c ;{2pt/2pt} c c c c @{}}
    1
    & v_1^{-1} \bfv^\dagger_{2:n}
    \\[0.1em]
    \cdashline{1-2}[2pt/2pt]
    \multirow{2}{*}{  \vspace{1em}  $v_1^{-1} \bfv_{2:n} $     }
    \\[-1em]
    & (Q/D) D^{-1} + v_1^{-2} \bfv_{2:n} \bfv_{2:n}^\dagger 
\end{array}
\right],
\label{eq:Q_Inverse}
\end{align}
where we have that \( (Q/D)^{-1} = v_1 \big/ \|\bfv\|^2 \) and \( D^{-1} = -v_1^{-1} \Idn[n-1]\). 

Notice that we can further simplify \(Q^{-1}\) in \cref{eq:Q_Inverse} as
\begin{align}
    Q^{-1} &=
    \frac{v_1}{\|\bfv\|^2}
    \left[
    \begin{array}{@{} c ;{2pt/2pt} c c c c @{}}
        1
        & v_1^{-1} \bfv^\dagger_{2:n}
        \\[0.1em]
        \cdashline{1-2}[2pt/2pt]
        \multirow{2}{*}{  \vspace{1em}  $v_1^{-1} \bfv_{2:n} $     }
        \\[-1em]
        & v_1^{-2} \bfv_{2:n} \bfv_{2:n}^\dagger 
    \end{array}
    \right]
    +
    \left[
    \begin{array}{@{} c ;{2pt/2pt} c c c c @{}}
        \multirow{3}{*}{  \vspace{1em}  $\mathbf{0}_{n} $     } 
        &  \mathbf{0}_{n-1}^\top
        \\[0.1em]
        \cdashline{2-2}[2pt/2pt]
        \\[-1em]
        & D^{-1}
    \end{array}
    \right]
    \notag
    \\
    &=
    \frac{v_1}{\|\bfv\|^2}
    \begin{bmatrix}
        1 \\ v_1^{-1} \bfv_{2:n}
    \end{bmatrix}
    \otimes
    \begin{bmatrix}
        1 \\ v_1^{-1} \bfv_{2:n}
    \end{bmatrix}^\dagger
    +
    \left[
    \begin{array}{@{} c ;{2pt/2pt} c c c c @{}}
        \multirow{3}{*}{  \vspace{1em}  $\mathbf{0}_{n} $     } 
        &  \mathbf{0}_{n-1}^\top
        \\[0.1em]
        \cdashline{2-2}[2pt/2pt]
        \\[-1em]
        & D^{-1}
    \end{array}
    \right]
    \notag
    \\
    &
    =
    \frac{1}{v_1 \|\bfv\|^2}  
    \bfv
    \otimes
    \bfv^\dagger
    - \diag( 0, v_1^{-1} \Ones[n-1] )
    \label{eq:Q_inverse_compact}
    \\&
    =
    \frac{1}{v_1} \cbr{
    \frac{1}{\|\bfv\|^2} 
    \begin{bmatrix}
	v_1 \overline{v}_1 & v_1 \overline{v}_2 & \cdots & v_1 \overline{v}_n \\
	v_2 \overline{v}_1 & v_2 \overline{v}_2 &  \cdots & v_2 \overline{v}_n \\
	\vdots & \vdots & \ddots & \vdots \\
	v_n \overline{v}_1 & v_n \overline{v}_2 & \cdots & v_n \overline{v}_n
    \end{bmatrix}
    -
    \left[
    \begin{array}{@{} c ;{2pt/2pt} c c c c @{}}
        \multirow{3}{*}{  \vspace{1em}  $\mathbf{0}_{n} $     } 
        &  \mathbf{0}_{n-1}^\top
        \\[0.1em]
        \cdashline{2-2}[2pt/2pt]
        \\[-1em]
        & \Idn[n-1] 
    \end{array}
    \right]
    }.
    \notag
\end{align}
We see from this last expression that \( Q^{-1} \) is just a weighted correction of a projection. Geometrically, \( (\bfv \otimes \bfv^\dagger) / \|\bfv\|^2 \) represents an orthogonal projection onto the direction of \( \bfv \). 
However, this alone is not enough to fully correct the change of basis induced by \( \bfv \). The diagonal term, \( \diag(0, \Ones[n-1]) \), acts as an anisotropic corrector that restores any information not captured in the projection.

A perk of the representation \cref{eq:Q_inverse_compact} is that the action of both \(Q\) and \(Q^{-1}\) can be computed in \( \mathcal{O}(n) \) operations,
which can be instantly seen from 
\[
    Q \bfu  = 
    \begin{bmatrix}
        \langle\bfu,\bfv\rangle \\
        u_1 \bfv_{2:n} - v_1 \bfu_{2:n} 
    \end{bmatrix}
    \qquad\text{and}\qquad
    Q^{-1}\bfu = 
    v_1^{-1} \bigg\{
    \frac{ \langle \bfu, \bfv \rangle }{\|\bfv\|^2} \bfv -
    \begin{bmatrix} 0 \\ \bfu_{2:n} \end{bmatrix} \bigg\}
    .
\]

%%%%%%%%%%%%%%%%%%%%%%%%%%%%%%%%%%%%%%%%
\subsubsection{Unitary representation}\label{sec:unitary_change}
After determining \(Q\) and its inverse, it may be pertinent to establish a unitary (orthogonal) representation of this change of basis. Assuming the same hypotheses as those that led to \cref{eq:Q_blocks_structure}, we can orthonormalize the columns in \(Q\) to yield a unitary transformation, denoted as \(U\).
This can be done exactly using the Gram--Schmidt process and the fact that \( \bfv \perp Q_{:,2:n}\). To see this, let us build the family \( \{\bfw_{i} \}_{i \in \llb 1,n\rrb} \) with \( \bfw_1 = \bfv\), and let \( \mathcal{I}_\bfv = \{ i \in \llb 2,n\rrb : \, \imath_{v_i = 0 } \}\); i.e., \( \mathcal{I}_\bfv \) contains the indices of \(\bfv\) corresponding to zero entries. We know that in such cases \( \bfq_i = -v_1 \bfe_i\) according to \cref{eq:Hermitian_Q_Construction}, hence we define \( \bfw_i = -v_1 \bfe_i\) for all \( i \in \mathcal{I}_\bfv\) and notice that in such cases \( \bfv \perp \bfw_i \) and \( \|\bfw_i\|_2 = |v_1| \). 
Observe that \( \bfw_2 = \bfq_2\) as well.
Now, for any other index \(i \in \llb 3,n \rrb \setminus \mathcal{I}_\bfv\), we have that
\begin{align*}
    \bfw_i &= \bfq_i -
    \quad \mathclap{\sum_{ \substack{ j \in \llb 2, i-1  \rrb \\ j \notin \mathcal{I}_\bfv} }} \quad 
    \proj_{\bfw_j} (\bfq_i)
    = \overline{v}_i \bfe_1 - v_1 \bfe_i -
    \,\,\,
    \mathclap{\sum_{j = 2}^{i-1}} \quad
    \frac{\langle \bfq_i, \bfw_j\rangle}{ \langle \bfw_j,\bfw_j \rangle } \bfw_j
    \\[-0.5em]
    &=
    - v_1 \bfe_i + \overline{v}_i \Bigg[ \bfe_1 - \sum_{j=2}^{i-1} \|\bfw_j\|^{-2} \overline{w}_{j,1} \bfw_j \Bigg].
\end{align*}
From here we can conclude that the matrix \( U = (\bfw_i / \|\bfw_i\|_2 )_{i\in \llb 1,n \rrb} \) can be written as
\[
    U = 
    \begin{bmatrix}
        \bfv & T_{\bfv}
    \end{bmatrix} 
    \diag \big( \|\bfv\|_2, \|\bfq\|_2, \|\bfw_3\|_2, \ldots, \|\bfw_n\|_2 \big)^{-1}
    ,
\]
where \( T_{\bfv} \) is an upper triangular matrix with constant first lower diagonal equal to \( -v_1\).

We can further find an explicit recurrence relation to compute the entries of \(U\) \cite[\S IX--6, Th. 2]{Gantmacher1959}. To see this, we can take advantage of the sparsity of \(Q\) and compute the Gram determinant
\begin{align*}
    G_i &\coloneqq 
    \det Q_{1:i,1:i} \, Q_{1:i,1:i}^\dagger
    =
    \left|
	\begin{array}{@{} c ;{2pt/2pt} c c c c @{}}
        \| \bfv \|^2 &  \Zeros[i-1]^\top
        \\[0.1em]
        \cdashline{1-2}[2pt/2pt]
        \multirow{2}{*}{  \vspace{1em}  $\Zeros[i-1]$ }& \\[-1em] &\bfv_{2:i} \otimes \bfv_{2:i}^\dagger + v_1^2 \Idn[i-1]
    \end{array}
    \right|
    \\
    &=
    \|\bfv\|^2 \det \big( \bfv_{2:i} \otimes \bfv_{2:i}^\dagger + v_1^2 \Idn[i-1] \big)
    =
    \|\bfv\|^2 v_1^{2(i-1)} \big( \bfv_{2:i}^\dagger \bfv_{2:i}/v_1^2 + 1 \big)
    \\
    &=
    \|\bfv\|^2 v_1^{2(i-1)} \big[ v_1^{-2} \| \bfv_{2:i} \|^2  + 1 \big]
    = v_1^{2(i-2)} \|\bfv_{1:i}\|^2 \|\bfv\|^2  
    ,
\end{align*}
where we have used the Weinstein--Aronszajn identity, which allows us to compute the determinant of a matrix product shifted by the identity. The orthogonal family \( \{\bfw_i\}_{i\in \llb 1,n\rrb} \) can be found by computing the formal determinant
\[
    \bfw_i = \frac{1}{G_{i-1}}
    \left|
	\begin{array}{@{} c }
        Q_{1:i-1, 1:i}  \, Q_{1:i, 1:i}^\dagger
        \\[0.1em]
        \cdashline{1-1}[2pt/2pt]
        (\bfq_j)_{j\in \llb 1,i\rrb}
    \end{array}
    \right|
    = \frac{1}{G_{i-1}}
    \left|
	\begin{array}{@{} c ;{2pt/2pt} c c c c @{}}
        \| \bfv \|^2 &  \Zeros[i-1]^\top
        \\[0.1em]
        \cdashline{1-2}[2pt/2pt]
        \multirow{2}{*}{  \vspace{1em}  $\Zeros[i-2]$ }& \\[-1em] &\bfv_{2:i-1} \otimes \bfv_{2:i}^\dagger + v_1^2 \big[\Idn[i-2] \,\, \Zeros[i-2] \big]
        \\[0.1em]
        \cdashline{1-2}[2pt/2pt]
        \multirow{2}{*}{  \vspace{1em}  $\bfv$ }& \\[-1em] 
        & (\bfq_j)_{j\in \llb 2,i\rrb}
    \end{array}
    \right|,
\]
where we understand the elements of the last rows within these determinants as symbols.
Using Laplace's expansion for the first row, we have
\begin{equation}
\label{eq:Det_GS_Part_1}
    G_{i-1}^{} \|\bfv\|^{-2} \bfw_i = 
    \left|
	\begin{array}{@{} c }
        \bfv_{2:i-1} \otimes \bfv_{2:i}^\dagger + v_1^2 \big[\Idn[i-2] \,\, \Zeros[i-2] \big]
        \\[0.1em]
        \cdashline{1-1}[2pt/2pt]
        (\bfq_j)_{j\in \llb 2,i\rrb}
    \end{array}
    \right|.
\end{equation}
It turns out that we can compute this determinant exactly in terms of the canonical basis. To see this, let \( \{ p_j \}_{j \in \llb 2,i \rrb } \) be a finite set in \( \F\) and consider the matrix
\[
    \left[
	\begin{array}{@{} c }
        \bfv_{2:i-1} \otimes \bfv_{2:i}^\dagger + v_1^2 \big[\Idn[i-2] \,\, \Zeros[i-2] \big]
        \\[0.1em]
        \cdashline{1-1}[2pt/2pt]
        (p_j)_{j\in \llb 2,i\rrb}
    \end{array}
    \right]
    =
    \begin{bmatrix} \bfv_{2:i-1} \\ 0 \end{bmatrix}
    \bfv_{2:i}^\dagger
    +
    \left[
	\begin{array}{@{} c ;{2pt/2pt} c c c c @{}}
        v_1^2 \Idn[i-2] &  \Zeros[i-2]
        \\[0.1em]
        \cdashline{1-2}[2pt/2pt]
        \multirow{2}{*}{  \vspace{1em}  $(p_j)_{j\in \llb 2,i-1\rrb}$ }& \\[-1em] & p_{i}
    \end{array}
    \right]
    \eqqcolon \mathbf{r} \mathbf{s}^\top + C.
\]
The determinant of this rank--one perturbation can be computed exactly as \cite[Eq. (9)]{Marcus1990}
\begin{equation}
\label{eq:det_Perturbation_1}
    \det( C + \mathbf{r} \mathbf{s}^\top )
    = \det(C) + \mathbf{s}^\top \mathrm{adj}(C) \, \mathbf{r},
\end{equation}
where \( \mathrm{adj}(C)\) is the adjugate of \(C\). Taking advantage of the sparsity of \(C\), we obtain
\[
    \det(C) = p_i v_1^{2(i-2)} 
    \qquad\text{and}\qquad
    \mathrm{adj}(C)
    = v_1^{2(i-3)} 
    \left[
	\begin{array}{@{} c ;{2pt/2pt} c c c c @{}}
        p_i \Idn[i-2] &  \Zeros[i-2]
        \\[0.1em]
        \cdashline{1-2}[2pt/2pt]
        \multirow{2}{*}{  \vspace{1em}  $(-p_j)_{j\in \llb 2,i-1\rrb}$ }& \\[-1em] & v_1^2
    \end{array}
    \right].
\]
As a result, we have that \cref{eq:det_Perturbation_1} yields
\begin{align}
    \det( C + \mathbf{r} \mathbf{s}^\top )
    &=    \notag
    v_1^{2(i-2)} \left[ p_i + v_1^{-2} \bfv_{2:i}^\dagger 
    \left[
    \begin{array}{@{} c }
        p_i \Idn[i-2]
        \\[0.1em]
        \cdashline{1-1}[2pt/2pt]
        (-p_j)_{j\in \llb 2,i-1\rrb}
    \end{array}
    \right]
 \bfv_{2:i-1} \right]
    \\
    &=    \notag
    v_1^{2(i-2)} \left[ p_i + v_1^{-2} \begin{bmatrix} \bfv_{2:i-1}^\dagger & \overline{v}_{i} \end{bmatrix}
    \left[
    \begin{array}{@{} c }
        p_i \bfv_{2:i-1}
        \\[0.1em]
        \cdashline{1-1}[2pt/2pt]
        -\sum\limits_{j=2}^{i-1}  v_j p_j
    \end{array}
    \right]
    \right]
    \\
    &=
    v_1^{2(i-3)} 
    \left[
        \|\bfv_{1:i-1}\|^2 p_i - \overline{v}_{i}  \sum_{j=2}^{i-1} v_j p_j
    \right]
    .
    \label{eq:Determinant_GS_Field}
\end{align}
Returning to \cref{eq:Det_GS_Part_1}, notice that we can replace the finite family \( \{p_j\}_{j \in \llb 2,i \rrb} \) in \cref{eq:Determinant_GS_Field} with the family \( \{\bfq_j\}_{j \in \llb 2,i \rrb} \) as defined in \cref{eq:Hermitian_Q_Construction}. However, by the multilinearity of the determinant applied to the last row of \cref{eq:Det_GS_Part_1} or by the linearity of \cref{eq:Determinant_GS_Field} with respect to the \(\{p_j \}_{j \in \llb 2,i \rrb } \), we can express \cref{eq:Det_GS_Part_1} as the difference between two determinants based on \cref{eq:Determinant_GS_Field}, one with the family \( \{\overline{v}_j \bfe_1\}_{j \in \llb 2,i\rrb}\), and another with the family \( \{v_1 \bfe_j\}_{j \in \llb 2,i\rrb}\).
For the former, we have that
\begin{align*}
    v_1^{2(i-3)} &
    \left[
        \|\bfv_{1:i-1}\|^2 \overline{v}_i \bfe_1 - \overline{v}_{i}  \sum_{j=2}^{i-1} v_j \overline{v}_j \bfe_1
    \right]
    =
    \overline{v}_i v_1^{2(i-3)} 
    \left[
        \|\bfv_{1:i-1}\|^2 -  \sum_{j=2}^{i-1} |v_j|^2
    \right] \bfe_1
    \\
    &= 
    \overline{v}_i v_1^{2(i-3)} 
    \big[
        \|\bfv_{1:i-1}\|^2 -  \| \bfv_{1:i-1} \|^2 + v_1^2
    \big] \bfe_1
    = \overline{v}_i v_1^{2(i-2)} \bfe_1.
\end{align*}
Likewise,
\begin{align*}
    v_1^{2(i-3)} &
    \left[
        \|\bfv_{1:i-1}\|^2 v_1 \bfe_i - \overline{v}_{i}  \sum_{j=2}^{i-1} v_j v_1 \bfe_j
    \right]
    =
    v_1^{2i - 5}      \| \bfv_{1:i-1} \|^2
    \bfe_i
    - \overline{v}_i v_1^{2i-5} \sum_{j = 2}^{i-1} v_j \bfe_j.
\end{align*}
Hence we have that
\begin{align*}
    G_{i-1} \|\bfv\|^{-2} \bfw_i &= 
    v_1^{2(i-3)} \bigg[ 
    \overline{v}_i 
    v_1^2 \bfe_1
    + 
    \overline{v}_i v_1 \sum_{j = 2}^{i-1} v_j \bfe_j
    -
    v_1 \| \bfv_{1:i-1} \|^2 \bfe_i \bigg].
\end{align*}
We can compute \( G_{i-1} \|\bfv\|^{-2} = \|\bfv_{1:i-1}\|^2 v_{1}^{2(i-3)}\), hence
\begin{align*}
    \bfw_i &= 
        \frac{ \overline{v}_i v_1^2 }{\|\bfv_{1:i-1}\|^2}
    \bfe_1
    +  \frac{ \overline{v}_i v_1 }{\|\bfv_{1:i-1}\|^2}
     \sum_{j = 2}^{i-1} v_j \bfe_j
    -
    v_1 \bfe_i
    =
    \frac{ \overline{v}_i v_1 }{\|\bfv_{1:i-1}\|^2}
     \begin{bmatrix}
     	\bfv_{1:i-1} \\ \Zeros[n-(i-1)]
     \end{bmatrix}
    -
    v_1 \bfe_i
    .
\end{align*}
The normalization factors can be computed exactly again using the formal determinant as above, where we have
\begin{align*}
    \bfu_i &\coloneqq \frac{\bfw_i}{\|\bfw_i\|} = \sqrt{\frac{ G_{i-1} }{G_i}} \bfw_i = \frac{ \| \bfv_{1:i-1} \| }{ v_1 \|\bfv_{1:i}\| } \bfw_i 
    =
    \frac{1}{ \|\bfv_{1:i}\| } \bigg[
    \frac{ \overline{v}_i }{\|\bfv_{1:i-1}\|}
     	\begin{bmatrix}
     		\bfv_{1:i-1} \\ \Zeros[n-(i-1)]
	\end{bmatrix}
         -    \|\bfv_{1:i-1}\| \bfe_i \bigg]
\end{align*}
for all \( i \in \llb 2,n \rrb\). Now, let us define the two \((n-1)\)--dimensional vectors
\[ 
    \bfa_{\bfv} \coloneqq \big( \|\bfv_{1:i-1}\| \,  \|\bfv_{1:i}\|^{-1} \big)_{i\in \llb 2,n\rrb } 
    \qquad\text{and}\qquad
    \bfb_{\bfv} \coloneqq \big(  \|\bfv_{1:i-1}\| \, \|\bfv_{1:i}\|  \big)^{\circ -1}_{i \in \llb 2,n \rrb},
\]
where \((\,\cdot\,)^{\circ-1}\) is the component--wise Hadamard multiplicative inverse. 
Then, the unitary matrix \(U \coloneqq (\bfu_i)_{i \in \llb 1,n\rrb}\) can be written in the following form:
\begin{equation}\label{eq:UnitaryGeneralForm}
    U =  
    \diag(\bfv) 
    \left[
	\begin{array}{@{} c ;{2pt/2pt} c}
		 \Ones & \mathrm{Triu}_{0}( \Ones[n,n-1] ) 
    \end{array}  \hspace{-0.35em}
    \right]
    \diag\big( \|\bfv\|^{-1}, \,\overline{\bfv}_{2:n} \circ \bfb_{\bfv} \big)
    -
    \diag\big( 0, \, \bfa_{\bfv} \big).
\end{equation}
Similarly to \(Q^{-1}\), we can understand \(U\) as the difference between a matrix that arranges most of the information in \(\bfv\) and a correction factor:
\begin{align*}
	U &=
    \begin{bmatrix}
	v_1   \\
	 & \hspace{-1.25em}  v_2  &   \\
	 &  & \hspace{-1.25em}  \ddots & \\
	 & &  & \hspace{-1.25em}  v_n 
    \end{bmatrix}
    \left[\arraycolsep=1.6pt
	\begin{array}{@{} c ;{2pt/2pt} r c c c c c c}
		 \multirow{6}{*}{   $ \Ones$     }
		 &  \hphantom{0} 1 & 1 & \cdots & 1
		\\
		& 0 & 1 & \ddots & \multirow{1}{*}{   $ \vdots$     }
		\\
		& \multirow{1}{*}{   $ \vdots$     } & \ddots & \ddots
		\\
		& & &  0 & 1
		\\ \cdashline{2-5}[2pt/2pt]
		& 0 & 0 & \cdots &  0 
	\end{array}
	\right]
    \begin{bmatrix}
	\|\bfv\|^{-1} \vspace{-0.25em}    \\
	 & \hspace{-1.75em} \overline{v}_2   \big(  |v_{1}| \, \|\bfv_{1:2}\|  \big)^{ -1} \vspace{-0.25em}    \\
	 && \hspace{-1.75em} \overline{v}_3   \big(  \|\bfv_{1:2}\| \, \|\bfv_{1:3}\|  \big)^{ -1}  \vspace{-0.5em}  \\
	 &  && \hspace{-1.em} \ddots &    \\
	 &  &  && \hspace{-1.75em} \overline{v}_n \big(  \|\bfv_{1:n-1}\| \, \|\bfv\|  \big)^{ -1}
    \end{bmatrix}
    \\
    &\qquad  - \hspace{1em}
    \begin{bmatrix}
	0   \\
	 & |v_{1}| \,  \|\bfv_{1:2}\|^{-1}   \\
	 && \hspace{-1em} \|\bfv_{1:2}\| \,  \|\bfv_{1:3}\|^{-1} \\
	 &  && \hspace{-1em} \ddots &  \\
	 &  &&& \hspace{-1em} \|\bfv_{1:n-1}\| \,  \|\bfv\|^{-1}
    \end{bmatrix} 
    .
\end{align*}

A nice outcome of finding \cref{eq:UnitaryGeneralForm} is that its action on a vector can be computed in a similar time as \(Q\):
\begin{align*}
    U \bfx  &=
    \bfv \circ
    \left[
	\begin{array}{@{} c ;{2pt/2pt} c c c c c c c}
		 \Ones & \mathrm{Triu}_{0}( \Ones[n,n-1] )
    \end{array}  \hspace{-0.35em}
    \right]
    \begin{bmatrix}
        \| \bfv\|^{-1} 
        \\
        \overline{\bfv}_{2:n} \circ \bfb_{\bfv}
    \end{bmatrix} 
    \circ \bfx
    -
    \begin{bmatrix} 0 \\ \bfa_{\bfv} \end{bmatrix} \circ \bfx
    \\
    &=
    \bfv \circ
    \textrm{SCS}_{\uparrow} \left(
    \begin{matrix}
        \| \bfv\|^{-1} x_1
        \\
        \overline{\bfv}_{2:n} \circ \bfb_{\bfv} \circ \bfx_{2:n}
    \end{matrix} 
    \right)
    -    \begin{bmatrix} 0 \\ \bfa_{\bfv} \circ \bfx_{2:n} \end{bmatrix},
\end{align*}
where \(\textrm{SCS}_{\uparrow}\) computes a translated shifted backwards cumulative sum given by:
\[
    \textrm{SCS}_{\uparrow}(\bfx)
    \coloneqq
    x_1 + 
    \left[
	\begin{array}{@{} c ;{2pt/2pt} c c c c c c c}
    \bigg(\sum\limits_{j=i+1}^n x_j\bigg)_{i\in \llb 1,n-1\rrb} & 0
    \end{array}  \hspace{-0.35em}
    \right]^\top    \hspace{-0.35em}.
\]
We can instantly observe that \(U\bfx\) can be computed in \( \mathcal{O}(n) \) time as it only involves Hadamard products of vectors and a cumulative sum.
Moreover, the adjoint of this linear operator is just a shifted cumulative sum
\[
    \textrm{SCS}_{\downarrow} (\bfx)
    \coloneqq
    \left[
	\begin{array}{@{} c ;{2pt/2pt} c c c c c c c}
        \sum\limits_{j=1}^{n} x_j &
        \bigg(\sum\limits_{j=1}^{i-1} x_j\bigg)_{i\in \llb 2,n-1\rrb}
    \end{array}  \hspace{-0.35em}
    \right]^\top    \hspace{-0.35em}.
\]
The latter expression allows us to obtain the action of the adjoint \( U^\dagger\):
\begin{align*}
    U^\dagger \bfx  &=
    \begin{bmatrix}
        \| \bfv\|^{-1} 
        \\
        \bfv_{2:n} \circ \bfb_{\bfv}
    \end{bmatrix}
    \circ
    \textrm{SCS}_{\downarrow} ( \overline{\bfv} \circ \bfx )
    -    \begin{bmatrix} 0 \\ \bfa_{\bfv} \circ \bfx_{2:n} \end{bmatrix}.
\end{align*}

% ------------------------------------------------------------------------------------------- %
% ------------------------------------------------------------------------------------------- %
% ------------------------------------------------------------------------------------------- %
\subsection{Application to graph Laplacians}
\label{sec:CoB_GL}

In the previous section we constructed an exact change of basis that allowed us to express any Hermitian matrix as a block matrix based on a known eigenpair. In what follows, we will apply our methodology to analyze scalings of graph Laplacians and derive some spectral properties of these operators.

Let us now recall the nonlocal system
\begin{equation}
\label{eq:discretized-state}
	A \bfu = \big( \lambda \Idn + \underbrace{ \mu \overbrace{( \diag(\etab) - \Gamma)}^{\eqqcolon L} }_{\eqqcolon B} \big) \bfu = \lambda  \mathbf{f}, 
\end{equation}
where \(\lambda,\mu > 0\), \( \mathbf{f} \in \R^n\), \( \etab \in \R^n\) is just the sum of \( \Gamma\) by rows; i.e., \( \etab \coloneqq \Gamma \Ones\), and \(\Gamma\) is an indefinite symmetric matrix with entries in \([0,1]\), with zero diagonal and at least one nonzero strictly positive off--diagonal entry per row.

We promptly identify \(L\) as a graph Laplacian, and, by linearity, we have that \(B\) is also a graph Laplacian and \(A\) is a constant diagonal perturbation of \(B\). The following properties follow suit: 
\begin{itemize}
    \item Our assumptions on \(G\) imply that \(L\) is the graph Laplacian of a connected graph\footnote{For the disconnected case, we can start finding each connected component of the graph, then the theory that follows is applicable to each connected component.}.
    \item \(L\), \(B\), and \(A\) are \(M\)--matrices; i.e., with positive diagonal and non--positive off-diagonal entries. 
    \item \(L\) and \(B\) are symmetric, positive semidefinite, and diagonally dominant. They share the smallest eigenvalue \(0 = \min \Sigma(L) = \min \Sigma(B)\), which is simple by the connectivity \cite[Theorem 2.7]{Grigoryan2018}, and its corresponding eigenvector \(\Ones\).
    \item \(A\) is symmetric positive definite, strictly diagonally dominant, and its smallest eigenvalue \( \lambda \) is associated again with the eigenvector \( \Ones\).
\end{itemize}

We can gain some additional information from the spectrum of \(A\). 
Let us revisit the extension \cite[Lemma 2.1]{Reams_1999} of Schur's theorem concerning the spectra of Hadamard products \cite[Theorem \S9 J.1]{Marshall2011}:
\begin{theorem}\label{th:Hadamard}
    Let \( C, D \in \mathcal{M}_n (\R)\) be symmetric. If \( C \succcurlyeq 0\), then 
	\[
        \max \Sigma(D) \diag(C) \succcurlyeq C \circ D \succcurlyeq \min\Sigma(D) \diag(C).
    \]
	Here \( \preccurlyeq \) is the Löwner order over the set of symmetric positive semidefinite matrices.
\end{theorem}

\begin{lemma}\label{r:general_spectral_bounds}
	The spectral radius of \(L\) is contained in a ball centered at \(\max \etab\) and with radius \(\max \etab\); i.e., \( \max \etab \leq \rho(L) \leq 2 \max \etab\) due to symmetry. Moreover, if  \(\Gamma\) is the weight matrix of a non--bipartite graph\footnote{In particular, this occurs whenever we have a complete graph with \( \gamma_{i,j} > 0\) for all \( i\neq j\) and \(n> 2\).}, then the upper bound improves to a strict inequality.
\end{lemma}
\begin{proof}
    The lower bound follows from \cref{th:Hadamard} by selecting \( C = \mathsf{I}_n\) and \(D = L\), then \( \rho(L) \mathsf{I}_n \succcurlyeq \diag(\etab) \succcurlyeq \Zeros[n,n]\); i.e., \( \rho(L) \geq \max \etab\). Then the upper bound follows from \cite[Theorem 2.7]{Grigoryan2018}. In the same theorem, it shown that the upper bound is attained whenever the underlying graph is bipartite. 
\end{proof}

As a result, we have that \( \Sigma(A) \subset [\lambda, \lambda + 2\mu \max \etab) \), which yields the following inclusion for the condition number of \(A\) for non--bipartite based \(L\):
\begin{equation}
\label{eq:CondNumberUnPrecSystem}
    \kappa(A) = 1 + \frac{1}{\lambda} \rho(B) = 1 + \frac{\mu}{\lambda} \rho(L)
    \in 1 + \frac{\mu \max \etab}{\lambda} [1,2).
\end{equation}
We observe that \(\lambda\) has a critical rôle in the conditioning of \(A\); i.e., for values \( \lambda \ll \mu \max \etab\), the system \cref{eq:discretized-state} is ill--conditioned. Let us apply the sparse change of basis we derived in \cref{sec:change_of_basis}. To do this, we identify \( \bfv = \Ones\) as the basis eigenvector for the change of basis \(Q\) in \cref{eq:Hermitian_Q_Construction}, which yields
\begin{equation}
\label{eq:QLap}
    Q =
    \left[
	\begin{array}{@{} c ;{2pt/2pt} c c c c @{}}
		 \multirow{2}{*}{   $\Ones$     }
		 &  \Ones[n-1]^\top 
		\\
		\cdashline{2-2}[2pt/2pt]
		& -\Idn[n-1]
	\end{array}
	\right]
    \qquad\text{and}\qquad
    Q^{-1} = 
    n^{-1}\Ones[n,n] - \diag \big( 0, \Ones[n-1] \big).
\end{equation}
Their action on a vector \( \bfu \in \R^n\) with mean \( \overline{u} = \frac{1}{n} \sum\limits_{i=1}^n u_i\) can be easily computed through
\(
  Q\bfu = 
  \begin{bsmallmatrix}
      n\overline{u} &\,\,  u_1 \Ones[n-1]^\top  - \bfu_{2:n}^\top
  \end{bsmallmatrix}^\top\hspace{-.3em}
\)
and
\(
    Q^{-1}\bfu = 
    \begin{bsmallmatrix}
        \overline{u} &\,\, \overline{u} \Ones[n-1]^\top - \bfu_{2:n}^\top
    \end{bsmallmatrix}^\top \hspace{-.3em}
\).
Following \cref{eq:UnitaryGeneralForm}, the unitary representation of the change of basis encoded in \(Q\) is given by
\begin{align}
    U &=  
    \left[
	\begin{array}{@{} c ;{2pt/2pt} c}
		 \Ones & \mathrm{Triu}_{0}( \Ones[n,n-1] ) 
    \end{array}  \hspace{-0.35em}
    \right]
    \diag\Big( n^{-\nicefrac 1 2},  \big( \tfrac{1}{k(k-1)} \big)^{\circ \nicefrac 1 2}_{i \in \llb 2,n \rrb} \Big)
    -
    \diag\Big( 0, \big( \tfrac{k-1}{k} \big)^{\circ \nicefrac 1 2}_{i \in \llb 2,n \rrb} \Big)
    \notag
    \\
    &=
	\left[\arraycolsep=1.6pt
	\begin{array}{@{} c ;{2pt/2pt} r r r r r r r}
		 \multirow{6}{*}{   $ \Ones$     }
		 &  \phantom{-} 1 & 1 & 1 & \cdots & 
		 \multirow{5}{*}{   $  \Ones[n-1] $}
		\\
		& -1 & 1 & 1
		\\
		\cdashline{2-2}[2pt/2pt]
		& 0 & -2 & 1
		\\
		\cdashline{3-3}[2pt/2pt]
		& \vdots & 0 & -3
		\\
		\cdashline{4-4}[2pt/2pt]
		& & & 0 & \ddots
		\\
		& 0 & \cdots &  & 0 &-(n-1)
	\end{array}
	\right]
	\diag\Big( n^{-\nicefrac 1 2},  \big( \sqrt{\nicefrac{1}{k(k-1)}} \big)_{i \in \llb 2,n \rrb} \Big).
    \label{eq:ULap}
\end{align}
Again, it is possible to obtain simplified expressions for the action of \(U\) and \(U^\top\), which are real matrices, as at the end of \cref{sec:unitary_change}.

We know that as a result of the symmetry of \(Q\), the conjugate forms \(\UnitSim[Q]{A}\) and \(\UnitSim{A}\)\footnote{Recall the notation from \cref{sec:change_of_basis}: \(\UnitSim[X]{Y} \coloneqq X^{-1} Y X\) for any two square matrices \(Y\) and \(X\).} have a block diagonal structure as in \cref{eq:GeneralChangeOfBasisOfA} with \( \mathbf{b}_1 = \Zeros[n-1]^\top\). After some algebraic manipulations, it can be derived that the sub--blocks \((\UnitSim[Q]{A})_{2:n,2:n}\) and \((\UnitSim[U]{A})_{2:n,2:n}\) can have non--negative off--diagonal entries. In the case of \(\UnitSim[Q]{A}\), we can readily see this since
\[
	Q^{-1} L = 
	\left[
	\begin{array}{@{} c c}
		\mathbf{0}_{n}^\top
		\\[0.1em]
		\cdashline{1-2}[2pt/2pt]
		\\[-0.9em]
		-L_{2:n, 1:n}
	\end{array}
	\right]
	\qquad\text{and}\qquad
	\UnitSim[Q]{L} = 
	\left[
	\begin{array}{@{} c ;{2pt/2pt} c c c c c c c}
		 0 & \mathbf{0}_{n-1}^\top
		\\[0.1em]
		\cdashline{1-2}[2pt/2pt]
		\\[-0.9em]
		\mathbf{0}_{n-1} & L_{2:n, 2:n} + \Gamma_{2:n,1} \Ones[n-1]^\top
	\end{array}
	\right]
	.
\]
As a result, we have that
\begin{equation}
\label{eq:ChangeOfBasisOfA}
	\UnitSim[Q]{A} =
	\left[
	\begin{array}{@{} c ;{2pt/2pt} c c c c c c c}
		 \lambda & \mathbf{0}_{n-1}^\top
		\\[0.1em]
		\cdashline{1-2}[2pt/2pt]
		\\[-0.9em]
		\mathbf{0}_{n-1} &  \mu L_{2:n, 2:n} + \mu \Gamma_{2:n,1} \Ones[n-1]^\top + \lambda \Idn[n-1]
	\end{array}
	\right].
\end{equation}
By our analysis in \cref{sec:change_of_basis}, we know that \( \Sigma(A) = \Sigma(\UnitSim[Q]{A}) \), moreover
\(
    \Sigma(\UnitSim[Q]{A}) \setminus \{\lambda\} = \mu \Sigma\big(  L_{2:n, 2:n} + \Gamma_{2:n,1} \Ones[n-1]^\top  \big) + \lambda
\).

At this point, we introduce the \emph{leverage operator} \( \mathsf{P}_{\ell(\lambda)} \) defined as 
\(
    \mathsf{P}_{ \ell(\lambda) } \coloneqq \diag\big( \ell_1 (\lambda), \ell_2 (\lambda) \Ones[n-1] \big).
\)
When we premultiply \( \UnitSim{A}\) or \(\UnitSim[Q]{A}\) by such an operator, we modify the spectrum in the following way:
\[
	\Sigma\big( \mathsf{P}_{ \ell(\lambda) } \UnitSim{A} \big) =
	\Sigma\big( \mathsf{P}_{ \ell(\lambda) } \UnitSim[Q]{A} \big)
	=
	\ell_2(\lambda) \big[ \Sigma(B) + \lambda \big] \setminus \{\lambda\}
	\, \cup \,
	\{ \ell_1(\lambda) \lambda \}.
\]
Notice here that \( \mathsf{P}_{ \ell(\lambda) } \UnitSim{A}\) is a symmetric matrix, as the leverage \(\mathsf{P}_{ \ell(\lambda) }\) acts on each block of \( \UnitSim{A}\) as a constant.

To understand the name and motivation of this operator, observe that if we select the functions \( \ell_1(\lambda) = \lambda^{-1}  \big(  \delta \rho(B) + (1-\delta) a(B) +\lambda \big) \) with \(\delta \in [0,1]\) and also \( \ell_2(\lambda) \equiv 1\), then the spectrum of \( \mathsf{P}_{ \ell(\lambda) } \UnitSim{A} \) will be contained in the convex hull of \( \Sigma(A) \setminus \{\lambda\} \). Hence we have \emph{leveraged} the smallest eigenvalue \(\lambda\). 

Let us consider how the operator affects the matrix in the original basis. To see this, consider the conjugation by \(Q\) and the same leverage action \( \mathsf{P}_{\ell(\lambda)} \) that reallocates the smallest eigenvalue \(\lambda\) inside the spectrum of the sub--block \( (\UnitSim[Q]{A})_{2:n,2:n} \). Then we have that
\begin{align}
    Q \mathsf{P}_{\ell(\lambda)} \UnitSim[Q]{A} Q^{-1} 
    &= Q \lambda \mathsf{P}_{\ell(\lambda)} Q^{-1} + B
    \notag
    \\
	&= \frac{1}{n} ( \lambda \ell_1(\lambda) - \lambda ) \Ones[n,n] + \lambda \Idn
	+ B = \frac{1}{n} ( \lambda \ell_1(\lambda) - \lambda ) \Ones[n,n] + A.
    \label{eq:LeverageBack}
\end{align}
As the change of basis does not modify the spectrum, the spectrum of \eqref{eq:LeverageBack} lies in the interval \( \big[ a(A), \rho(A) \big]\). Moreover, notice that now we are dealing with a positive definite matrix that might not necessarily be an \(M\)--matrix, yet it is still strictly diagonally dominant. 

The leverage operator allows us to retrieve information about the algebraic connectivity of \(L\). 
\begin{lemma}\label{lemma:UpperConnectivity}
    For any graph Laplacian \(L\), it holds that
    \(
        a(L) \leq \frac{n}{n-1} \min \etab.
    \)
\end{lemma}
\begin{proof}
    Notice that we can apply the same procedure as in \cref{r:general_spectral_bounds} by instead letting \( C = \Idn\) and \( D = Q \mathsf{P}_{\ell(\lambda)} \UnitSim[Q]{A} Q^{-1}\). The diagonal of \(D\) is just \( \frac{1}{n} \big( \lambda \ell_1(\lambda) - \lambda \big) \Ones + \lambda \Ones + \mu\etab\), and its minimum eigenvalue is \( \min\big\{ a(B) + \lambda, \lambda \ell_1(\lambda) \big\} \). The latter can be seen by computing a Rayleigh quotient with the corresponding vector expanded as a linear combination of eigenvectors of $A$, which include $\Ones[n]$ and $n-1$ vectors orthogonal to $\Ones[n]$ By the choice \( \ell_1(\lambda) = \lambda^{-1}  \big(  \delta \rho(B) + (1-\delta) a(B) +\lambda \big) \) with \( \delta = 0\), we obtain that \( \min \Sigma(D) = a(B) + \lambda\), and then by \cref{th:Hadamard}:
    \[
    	\big( a(B) + \lambda \big) \Idn \preccurlyeq \diag\Big( \frac{1}{n} \big( \lambda \ell_1(\lambda) - \lambda \big)\Ones + \lambda \Idn + \mu\etab \Big)
    	=
    	\diag\Big( \frac{1}{n} a(B) \Ones + \lambda \Idn + \mu\etab \Big)
    	,
    \]
    which means that 
    \( a(B) + \lambda \leq \frac{1}{n} a(B) + \lambda + \mu \min \etab\) 
    and the bound follows.
\end{proof}
The inequality in \cref{lemma:UpperConnectivity} is as sharp as possible without employing specific properties from the underlying graph; for graph--based bounds, see \cite{Abreu2007,Brouwer2012}. A simple example arises from the unweighted complete graph \(K_{n}\): its unscaled graph Laplacian is given by \( L = n \Idn - \Ones[n,n]\), and its spectrum is just \(\{0,n\}\) where \(n\) has algebraic multiplicity \(n-1\). As a result, \( a(B) = n\) but \( \min \diag(B) = n-1\), so we have that \( a(B) \geq \min \diag(B)\) and the inequality \( a(B) \leq \frac{n}{n-1} \min \diag(B)\) is exactly attained.

Finally, we can arrive at a new inequality for bounding the algebraic connectivity from below by noticing that the matrix \( E = L_{2:n, 2:n} + \Gamma_{2:n,1} \Ones[n-1]^\top\) is a rank--one perturbation of a definite \(M\)--matrix. The perturbation is positive semidefinite in the sense that \( \Sigma(\Gamma_{2:n,1} \Ones[n-1]^\top) = \Big\{0, \sum\limits_{i=2}^n\Gamma_{i,1} \Big\}\). Moreover, we obtain the following result:

\begin{lemma}\label{lem:lower_bound_algebraic_connectivity}
    It holds that \( \Sigma(E) \geq \min \Gamma_{2:n,1}\).
\end{lemma}
\begin{proof}
    Let us denote \( \bfu \coloneqq \Gamma_{2:n,1}\) and \(S \coloneqq L_{2:n,2:n}\). Notice that \( S \Ones[n-1] = \bfu\). Thus we can factor \(E\) as the product of two symmetric matrices \( E = S( \Idn[n-1] + \Ones[n-1,n-1] ) \eqqcolon S F\). Observe that the matrix \(F\) is positive definite with \( \Sigma(F) = \{1,n\}\) and diagonalizable by nothing else than the change of basis \(U\) in \cref{eq:ULap}. As a result, \(E\) is the product of two positive definite matrices, hence we have by \cite[\S 20 A.1.a]{Marshall2011} that, for any \(k \in \llb 1,n-1\rrb\), it holds that \(\lambda_k(S) \leq \lambda_k(E) \leq n \lambda_k(S)\), for which we can assert that the eigenvalues of \(S\) satisfy \(\frac{1}{n} \lambda_k(E) \leq \lambda_k(S) \leq \lambda_k(E)\). Here we have used the notation \(\lambda_k(A)\) to refer to the ascending \(k\)--th eigenvalue (accounting for multiplicity) of a positive definite matrix \(A \in \mathcal{M}_n (\F) \).
    
    Now, it is clear that the sum of the off--diagonal entries of \(S\) by rows is nothing else than \( \etab_{2:n} - \bfu = \etab_{2:n} - \Gamma_{2:n,1}\). By Gershgorin's circle theorem \cite{Varah1975,Golub2013}, we have that \( \min \Sigma(S) \geq \min \Gamma_{2:n,1}\).
\end{proof}

% ⌘ + ⇧ + 7

The lower bound obtained in \cref{lem:lower_bound_algebraic_connectivity} allows us to obtain the following bound on the algebraic connectivity of the original system via \cref{eq:ChangeOfBasisOfA}:
\begin{equation}
    \label{eq:AlgCon_LB_on_A}
    a(A) \geq \mu \min \Gamma_{1:n,1} + \lambda.
\end{equation}
If \(L\) is the graph Laplacian of a complete graph, the bound \cref{eq:AlgCon_LB_on_A} is particularly useful for estimating when the system \cref{eq:discretized-state} is ill--conditioned, and estimating the condition number of the sub--block \( (\UnitSim[Q]{A})_{2:n,2:n}\). Notice that the bound gives a sharper criteria for ill--conditioning and at the expense of computing \( L\bfe_1\).

% ----------------------------------------------------- %
\subsection{Applications for preconditioning graph Laplacians}\label{sec:Prec_GLs}

In this section we propose two families of preconditioners for solving the system \cref{eq:discretized-state}. The first family consists of two sparse diagonal preconditioners. Their spectra are analyzed and criteria for their applicability are discussed. These preconditioners will be extended to a second family of dense preconditioners. Examples will showcase the effects of each family on preconditioning the system \cref{eq:discretized-state}. We note that, alongside research on solvers for graph Laplacians (see e.g., \cite{Spielman2011,Napov2016,Gao2023,D’Elia2021}), recent work on interior point methods for support vector machines, involving the second author \cite{WaPeSt2023}, addressed preconditioning matrices of the form \( -\Gamma \) (including analogous diagonal entries), summed with diagonal matrices from logarithmic barrier terms with entries that may be very large or small. However, under our current framework in this work, the replacement of the diagonal with \(\diag(\etab)\), which may lead to many of the smallest eigenvalues being moved further away from the origin, and their regularization by a (possibly very small) constant $\lambda$, necessitates a bespoke strategy for the matrix \(A\).

\subsubsection{Diagonal preconditioners}\label{sec:Diagonal_Preconditioners}
As \(A\) is a positive definite operator, the Jacobi preconditioner \( \Prec[a] = \diag(A) \) is an out--of--the--box suitable choice. By our choice of \(\Gamma\), we have that
\[
    \Prec[a] = \lambda \Idn + \mu \diag(\etab) - \mu \diag(\Gamma) = \lambda \Idn + \mu \diag(\etab).
\]
We have the following result for the numerical range of the preconditioned system \( \Prec[a]^{-1} A\):
\begin{lemma}\label{lem:Rayleigh_Diag_Jac}
The eigenvalues of the preconditioned system satisfy
\begin{equation}
    \label{eq:Bounds_on_prec_jacobi_diagonal}
    \frac{\lambda}{\lambda + \mu \max \etab} \leq
    \Sigma( \Prec[a]^{-1} A )
    \leq 
    \min \left\{ 2, \frac{n}{n-1} \frac{\rho(L)}{a(L)} \right\} .
\end{equation}
\end{lemma}
\begin{proof}
    Let us start by noticing that \( \Prec[a]^{-1} A\) can be written as the difference between an identity matrix and a stochastic matrix. Hence by \cite[Theorem 2.7]{Grigoryan2018}, we have that \( \Sigma(\Prec[a]^{-1} A) \subset [0,2]\).

    Now let us analyse the generalized Rayleigh quotient of \( \Prec[a]^{-1} A\). Here, let \( \bfx \neq \Zeros[n]\) and
    \begin{equation}
        \label{eq:Rayleigh_Pa}
        \Rayleigh(A, \Prec[a], \bfx)
        = 
        \frac{ \langle \bfx, A \bfx \rangle }{ \langle \bfx, \Prec[a] \bfx\rangle }
        =
        \frac{ \lambda \langle \bfx, \bfx\rangle + \mu \langle \bfx, \etab\circ \bfx \rangle - \mu \langle \bfx, \Gamma \bfx \rangle }{\lambda \langle \bfx, \bfx\rangle + \mu \langle \bfx, \etab\circ \bfx \rangle }.
    \end{equation}
    Now consider the inequality \( \frac{a+b}{a+c} \leq \max\{ 1, \nicefrac{b}{c} \} \), which holds for all \(a,b \geq 0\) and \( c>0\). Hence
    \[
        \Rayleigh(A, \Prec[a], \bfx)
        \leq
        \max \left\{
            \frac{ \langle \bfx, \etab\circ \bfx \rangle - \langle \bfx, \Gamma \bfx \rangle }{\langle \bfx, \etab\circ \bfx \rangle }
        ,1 \right\}.
    \]
    Notice that \( \langle \bfx, \etab\circ \bfx - \Gamma \bfx \rangle \leq \rho(L) \|\bfx\|^2\) and \( \langle \bfx, \etab \circ \bfx \rangle \geq \min \etab \|\bfx\|^2 \). Thus, we can further bound the above by
    \[
        \Rayleigh(A, \Prec[a], \bfx) \leq 
        \max \left\{
            \frac{\rho(L)}{\min \etab} ,  1
        \right\}.
    \]
    From \cref{lemma:UpperConnectivity}, we may conclude that
    \[
        \Rayleigh(A, \Prec[a], \bfx) \leq 
        \min \left\{
        2,
            \frac{n}{n-1} \frac{\rho(L)}{a(L)} 
        \right\}.
    \]

    For the lower bound, by the positive definiteness of $L$, we have that
    \[
        \Rayleigh(A, \Prec[a], \bfx)
        \geq
        \frac{ \lambda \langle \bfx, \bfx\rangle }{\lambda \langle \bfx, \bfx\rangle + \mu \langle \bfx, \etab\circ \bfx \rangle } 
        \geq \min_{i\in \llb 1,n\rrb} \frac{\lambda}{\lambda + \mu \eta_i}
        = \frac{\lambda}{\lambda + \mu \max \etab}
        .
    \]
\end{proof}

\begin{remark}
    From \cref{eq:Rayleigh_Pa} we also see, by the positivity of \(\etab\), that
    \[
        \Rayleigh(A, \Prec[a], \bfx) \leq 1 + \frac{\mu}{\lambda} \frac{\langle \bfx, L \bfx \rangle}{\langle \bfx, \bfx \rangle} 
        \leq 1 + \frac{2\mu}{\lambda} \max \etab.
    \]
    This is nothing else than the upper bound from the condition number of \(A\) as in \cref{eq:CondNumberUnPrecSystem}. Also, notice that if \( \lambda \gg 2\mu \max\etab\), then the spectral radius of the preconditioned system is very close to \(1\).
    On the other hand, the lower bound on \cref{eq:Bounds_on_prec_jacobi_diagonal} tells us that the smallest eigenvalue of \( \Prec[a]^{-1} A\) can be even smaller than the smallest eigenvalue of \(A\). In particular, this becomes critical in the case \(\lambda \ll \mu \max\etab\), as often \(\etab \geq \mathbf{1}_n\), for which we can suspect that the smallest eigenvalue can be closer to zero with preconditioning, hence damping the performance of Krylov subspace solvers.
\end{remark}

Now let us consider the preconditioner proposed in \cite{D’Elia2021} for the nonlocal means kernel:
\[
    \Prec[b] = \diag \bigg( \Big[ \sum_{j} a_{i,j}^2 \Big]^{1/2} \bigg)_{i \in \llb 1,n\rrb}. 
\]
Taking into account the definitions of \(A\), \(\Gamma\), and \(\etab\), we have that \( a_{i,i}^2 = (\mu\etab + \lambda)_i^2\), while \( a_{i,j}^2 = \mu^2 \gamma_{i,j}^2\) for \( i\neq j\) (here, we employ the convention of overwriting $\gamma_{i,i} = 0$).
Summing up either by rows or columns, we obtain the following:
\begin{equation}
\label{eq:Prec_rel_a_b}
       ( \Prec[b])_{i,i}^2 = \sum_{j=1}^n a_{i,j}^2 =  \mu^2 \sum_{j=1}^n \gamma_{i,j}^2 + (\mu \eta_i + \lambda)^2
        = \mu^2 [ (\Gamma \circ \Gamma) \mathbf{1}_n ]_i + ( \Prec[a] )_{i,i}^2,
\end{equation}
or alternatively
\(
	\Prec[b] = \big[ \mu^2 \diag( \Gamma^{\circ 2} \, \mathbf{1}_n ) + ( \Prec[a] )^2 \big]^{\nicefrac 1 2} \hspace{-0.3em}.
\)
\begin{lemma}
    \label{lem:Equivalence_Diag_Precs}
    The operators \( \Prec[a]\) and \( \Prec[b]\) are spectrally equivalent; specifically, \( \Prec[a] \preccurlyeq \Prec[b] \preccurlyeq \sqrt{2} \Prec[a]\).
\end{lemma}
\begin{proof}
    Due to the \(\ell_1\)--\(\ell_2\) norm equivalence, we have that
    \[
	   \Big( \sum_{j=1}^n \gamma_{i,j}^2 \Big)^{\nicefrac 1 2} = \|\gamma_{i,(:)}\|_2 \leq \|\gamma_{i,(:)}\|_1 = \sum_{j=1}^n \gamma_{i,j} = \eta_i.
    \]
    As a result, we have that \( \sum_{j=1}^n \gamma_{i,j}^2  \leq \eta_i^2\), and by basic manipulation
    \begin{align*}
           ( \Prec[b])_{i,i}^2 = \sum_{j=1}^n a_{i,j}^2 
            \leq \mu^2 \eta_i^2 + ( \Prec[a] )_{i,i}^2
            \leq 2 ( \Prec[a] )_{i,i}^2.
    \end{align*}
    Taking the square root of both sides, we obtain \( (\Prec[b])_{i,i} \leq \sqrt{2}  (\Prec[a] )_{i,i}.\)
    Further, by the non--negativity of all the terms involved in \cref{eq:Prec_rel_a_b}, we also have that \( (\Prec[a])_{i,i} \leq (\Prec[b])_{i,i}\).
\end{proof}

\begin{lemma}
\label{lem:Prec_Norm_2_Diag}
The eigenvalues of the preconditioned matrix \( \Prec[b]^{-1} A\) satisfy
\begin{equation*}
    2^{-\nicefrac 1 2} \lambda_{\min} (\Prec[a]^{-1} A) \leq
    \Sigma( \Prec[b]^{-1} A )
    \leq \lambda_{\max} (\Prec[a]^{-1} A).
\end{equation*}
\end{lemma}
\begin{proof}
    The result follows directly from \cref{lem:Equivalence_Diag_Precs}. To see this, consider the Rayleigh quotient
    \(
        \Rayleigh(A, \Prec[b], \bfx) = \frac{ \langle \bfx, A \bfx \rangle }{ \langle \bfx, \Prec[a] \bfx \rangle } \frac{ \langle \bfx, \Prec[a] \bfx \rangle }{ \langle \bfx, \Prec[b] \bfx \rangle }. 
    \)
    The second quotient is bounded in \( [2^{-\nicefrac 1 2}, 1] \); hence, 
    \( 
        2^{-\nicefrac 1 2} \Rayleigh(A, \Prec[a], \bfx) \leq \Rayleigh(A, \Prec[b], \bfx) \leq \Rayleigh(A, \Prec[a], \bfx)
    \)
    for all \(\bfx \in \R^n\).
\end{proof}

\begin{remark}[Computing \(\Gamma^{\circ 2}\)]
    Note that if computing \(\Gamma\) is computationally inefficient, this will also be the case with \( \Gamma^{\circ 2}\). However, we are interested in particular cases of \(\Gamma\): (i) if the entries of \(\Gamma\) are moderated by a variance parameter \(\sigma\) as in \( \gamma_{i,j} = \exp( -\sigma^{-2} r_{i,j} ) \), then \( \gamma_{i,j}^2 = \exp(-(\sigma / \sqrt{2} )^{-2} r_{i,j}) \), so computing \( \Gamma^{\circ 2} \, \Ones\) is as expensive as computing \( \Gamma  \Ones\) with a different \(\sigma\); (ii) if \(\Gamma\) is a composite operator of the form \( \Gamma = \frac{1}{\mathsf L} \sum\limits_{\ell=1}^{\mathsf L} \Gamma_{\ell}\) and \( \Gamma_\ell\) is as in case (i), then by Pascal's simplex we obtain the expansion
    \begin{align}
        \label{eq:Pascal_simplex}
        \gamma_{i,j}^2 &= \frac{1}{\mathsf L^2} \sum_{\ell = 1}^{\mathsf L} (\gamma_\ell)_{i,j}^2 + \frac{2}{\mathsf{L}^2} \sum_{\ell < \kappa} (\gamma_\ell)_{i,j} (\gamma_\kappa)_{i,j}
        \\
        &= 
        \frac{1}{\mathsf{L}^2}\sum_{\ell = 1}^{\mathsf L} \exp\big( {-}(\sigma / \sqrt{2} )^{-2} (r_\ell)_{i,j} \big)
        + \frac{2}{\mathsf L^2} \sum_{\ell < \kappa} 
        \exp\big( {-}\sigma^{-2} (r_\ell + r_\kappa)_{i,j} \big)
        \notag
        .
    \end{align}
    Out of the \( \frac{1}{2} \mathsf{L} (\mathsf{L} + 1) \) terms, the \( \frac{1}{2} \mathsf{L} (\mathsf{L} - 1) \) pairwise products can exceed the computational limitations of the NFFT. Hence, assembling \( \Prec[b] \) would require more computational time than assembling \(A\) itself, making it an unpractical preconditioner. 

    As a result, let us redefine the preconditioner \(\Prec[b]\) using just the first \(\mathsf L\) terms in \cref{eq:Pascal_simplex}:
    \begin{equation}
    \label{eq:ANOVA_Norm_2_Prec}
	\Prec[b'] = \bigg[ \big(\nicefrac{\mu}{\mathsf L}\big)^2 \sum_{\ell=1}^{\mathsf L} \diag( \Gamma_\ell^{\circ 2} \, \mathbf{1}_n ) + ( \Prec[a] )^2 \bigg]^{\nicefrac 1 2} \hspace{-0.3em}.
    \end{equation}
    Hence, we obtain that
    \(
        ( \Prec[a] )_{i,i}^2 \leq
       ( \Prec[b'])_{i,i}^2 
            \leq ( \Prec[b])_{i,i}^2
            \leq 2 ( \Prec[a] )_{i,i}^2.
    \)
    Furthermore, the spectral results in \cref{lem:Prec_Norm_2_Diag} also apply to \(\Prec[b']\). We will use \cref{eq:ANOVA_Norm_2_Prec} whenever we precondition systems based on the ANOVA kernel.
    
\end{remark}

% ----------------------------------------------------- %
\subsubsection{Dense preconditioners}\label{sec:Dense_Preconditioners}

As proven in \cref{lem:Rayleigh_Diag_Jac} and extended by \cref{lem:Prec_Norm_2_Diag}, the preconditioners \( \Prec[a]\) and \( \Prec[b]\) introduced in the previous section can result in an eigenvalue that is smaller than \(\lambda\) when applied to \(A\). We aim to mitigate this effect by applying the change of basis explored in \cref{sec:CoB_GL}.

Let \( X \in \{Q,U\}\) be any of the change of basis matrices in \cref{eq:QLap,eq:ULap}. Recall that the matrix \( \UnitSim[X]{A}\) that has the form
\begin{equation}
\label{eq:ChangeOfBasisOfAByItsFirstEigenvector}
	\UnitSim[X]{A} =
	\left[
	\begin{array}{@{} c ;{2pt/2pt} c c c c c c c}
		 \lambda & \mathbf{0}_{n-1}^\top
		\\[0.1em]
		\cdashline{1-2}[2pt/2pt]
		\\[-0.9em]
		\mathbf{0}_{n-1} &  \ast
	\end{array}
	\right] = X^{-1} A\, X.
\end{equation}
In this sense, as discussed in \cref{sec:change_of_basis}, we can solve the system \( \UnitSim[X]{A} \bfx = \bfg\) by splitting it so that \( x_1 = g_1/\lambda\) and \( (\UnitSim[X]{A})_{2:n, 2:n} \bfx_{2:n} = \bfg_{2:n}\). For solving the latter system, we would benefit from the action of the preconditioner \( \Prec[] \in \{ \Prec[a], \Prec[b] \}\). However, we need to express the preconditioner under the new basis; i.e., we introduce \( \UnitSim[X]{D} \coloneqq X^{-1} \Prec[]  X\). As a result \( \UnitSim[X]{D}^{-1} \UnitSim[X]{A} = X^{-1} \Prec[]^{-1} A \, X \) has the same range of eigenvalues as the preconditioned system in the original basis. Moreover, notice that \( \UnitSim[X]{D}^{-1} \UnitSim[X]{A}\) is dense as \( \UnitSim[X]{D} \) is dense, otherwise we would not be able to decouple the preconditioned system as for \cref{eq:ChangeOfBasisOfAByItsFirstEigenvector}. 

Our interest turns instead towards the projection of the preconditioner on the second diagonal sub--block of \( \UnitSim[X]{A}\). At this point, let us define the projection operator \(\pi_2: \mathcal{M}_n(\F) \to \mathcal{M}_{n-1}(\F)\) such that \( \pi_2(A) = A_{2:n, 2:n}\).
Now, consider the equivalent system
\begin{equation*}
    \underbrace{
	\left[
	\begin{array}{@{} c ;{2pt/2pt} c c c c c c c}
		 \multirow{2}{*}{$\,\mathbf{0}_{n}$} & \mathbf{0}_{n-1}^\top
		\\[0.1em]
		\cdashline{2-2}[2pt/2pt]
		\\[-0.9em]
		   &  
        \pi_2 ( \UnitSim[X]{A} )
	\end{array}
	\right] 
    }_{ \eqqcolon \UnitSim[X]{S} }
    \left[
	\begin{array}{@{} c }
		 0
		\\[0.1em]
		\cdashline{1-1}[2pt/2pt]
		\\[-0.9em]
		\bfx_{2:n}
	\end{array}
	\right] 
    =
    \left[
    \begin{array}{@{} c }
		 0
		\\[0.1em]
		\cdashline{1-1}[2pt/2pt]
		\\[-0.9em]
		\bfg_{2:n}
	\end{array}
	\right].
\end{equation*}
When we apply \( \UnitSim[X]{D}^{-1}\) to \( \UnitSim[X]{S}\), we obtain a system of the form
\begin{equation}
\label{eq:Higher_Dimension_Operator}
	\left[
	\begin{array}{@{} c ;{2pt/2pt} c c c c c c c}
		 \multirow{2}{*}{$\,\mathbf{0}_{n}$} & \begin{smallmatrix} \ast & \ast & \ast & \ast & \ast & \ast \end{smallmatrix}
		\\[0.1em]
		\cdashline{2-2}[2pt/2pt]
		\\[-0.9em]
		   &  %\pi_{2:n, 2:n} ( \UnitSim[X]{A}  )
        \pi_2 ( \UnitSim[X]{D}^{-1})  \pi_2(\UnitSim[X]{A} )
	\end{array}
	\right] 
    \left[
	\begin{array}{@{} c }
		 0
		\\[0.1em]
		\cdashline{1-1}[2pt/2pt]
		\\[-0.9em]
		\bfx_{2:n}
	\end{array}
	\right] 
    =
    \left[
    \begin{array}{@{} c }
		 \ast
		\\[0.1em]
		\cdashline{1-1}[2pt/2pt]
		\\[-0.9em]
		\pi_2 (\UnitSim[X]{D}^{-1}) \bfg_{2:n}
	\end{array}
	\right] .
\end{equation}
However, it is not difficult to also see that \( \pi_2 ( \UnitSim[X]{D}^{-1} \UnitSim[X]{A} ) = \pi_2 (\UnitSim[X]{D}^{-1}) \pi_2 ( \UnitSim[X]{A} )  \) (this is a consequence of the block structure). Hence we define a preconditioner for \( \UnitSim[X]{A}\) by its inverse form:
\[
    \mathcal{P}_{X, \Prec[]}^{-1} \coloneqq
    \left[
	\begin{array}{@{} c ;{2pt/2pt} c c c c c c c}
		 1 & \mathbf{0}_{n-1}^\top
		\\[0.1em]
		\cdashline{1-2}[2pt/2pt]
		\\[-0.9em]
		\mathbf{0}_{n-1} &  \pi_2( \UnitSim[X]{D}^{-1} )
	\end{array}
	\right].
\]
As \( \mathcal{P}_{X, \Prec[]}^{-1} \) is a block matrix, its action on \( \UnitSim[X]{A}\) retains the block structure, and hence we can decouple \( \mathcal{P}_{X, \Prec[]}^{-1} \UnitSim[X]{A}\) and focus just on the sub--block \( \pi_2 ( \UnitSim[X]{D}^{-1} \UnitSim[X]{A} )\). 

The spectral properties of \( \pi_2 ( \UnitSim[X]{D}^{-1} \UnitSim[X]{A} )\) can be studied from the higher--dimensional system \cref{eq:Higher_Dimension_Operator} as its lower diagonal block contains all of the relevant eigen--information. 

\begin{lemma}\label{lemma:Spectral_Control_Pa_on_D}
    Let \(\Prec\) be any of the diagonal preconditioners introduced in \cref{sec:Diagonal_Preconditioners} and \( \UnitSim[X]{D} = X^{-1} \Prec[] X\). Further, let \( X \in \{Q,U\}\) be any of the change of basis matrices in \cref{eq:QLap,eq:ULap}. Then \(\pi_2 ( \UnitSim[X]{D}^{-1} \UnitSim[X]{A}) \) has a strictly positive smallest eigenvalue, and the following holds:
    \[
        a(\Prec[]^{-1} A) - \frac{\lambda}{\lambda + \mu \min\etab} \leq \Sigma\big( \pi_2 ( \UnitSim[X]{D}^{-1} \UnitSim[X]{A} ) \big) \leq \rho(\Prec[]^{-1} A).
    \]
    Moreover, the upper bounds of \cref{lem:Rayleigh_Diag_Jac} are satisfied.
\end{lemma}
\begin{proof}
    Start by noticing that \( \UnitSim[X]{D}^{-1} \UnitSim[X]{S}\) is equivalent to \( \UnitSim[X]{D}^{-\nicefrac 1 2} \UnitSim[X]{S} \UnitSim[X]{D}^{-\nicefrac 1 2}\). 
    Moreover, we have that \(\UnitSim[X]{S}\) is nothing more than a perturbation of \(\UnitSim[X]{A}\):
    \[ 
        \UnitSim[X]{S} = \UnitSim[X]{A} - \diag(\lambda, \Zeros[n-1]) 
        = X^{-1} \Big( A - \frac{\lambda}{n} \Ones[n,n] \Big) X \eqqcolon X^{-1} T X.
    \]
    The last equality comes from the fact that conjugation by \(X\) diagonalizes the matrix \( (\nicefrac{\lambda}{n}) \Ones[n,n] \) by construction of \(Q\) and \(U\) for the graph Laplacian.
    
    At this point, notice that \(T\) is a deflation of \(A\). To see this, observe that
    \( T \Ones = \lambda \Ones - \lambda (\nicefrac n n) \Ones = 0 \Ones\). 
    Further, for any other eigenpair of \(A\), \( (\beta, \bfv) \), we have that \( \Ones \perp \bfv\) and thus 
    \( T \bfv = \beta \bfv \). 
    As a result, the rank--one perturbation of \(A\) by \( (\nicefrac{\lambda}{n}) \Ones[n,n] \) does not modify the eigenvalues of \(A\) that are different than \(\lambda\): \( \Sigma( \UnitSim[X]{S} ) = \{0\} \cup \Sigma(A)\setminus \{\lambda\} \). 
    
    From the fact that 
    \( 
        \UnitSim[X]{D}^{-1} \UnitSim[X]{S} = X^{-1} \Prec[]^{-1} T X 
    \), 
    we can describe the spectrum of the preconditioned operator \( \pi_2 ( \UnitSim[X]{D}^{-1} \UnitSim[X]{A} )\) by a sub--range of the Rayleigh quotient
    \(
       \Rayleigh( T, \Prec[], \bfy ).
    \)
    By the form of \cref{eq:Higher_Dimension_Operator}, the smallest value this quotient can take is \(0\), this value is attained by selecting \( \bfy = \Ones\). We also have that
    \begin{align*}
        \Rayleigh( T, \Prec[], \bfy ) 
        &= 
        \frac{ \langle \bfy, T \bfy \rangle }{ \langle \bfy, \Prec[] \bfy \rangle }
        = \frac{ \langle \bfy, A \bfy \rangle }{ \langle \bfy, \Prec[] \bfy \rangle }
        - \frac{\lambda}{n} \frac{ \langle \bfy, \Ones[n,n] \bfy \rangle }{ \langle \bfy, \Prec[] \bfy \rangle }
        \leq \frac{ \langle \bfy, A \bfy \rangle }{ \langle \bfy, \Prec[] \bfy \rangle },
    \intertext{which is bounded by the spectral radius of \( \Prec[]^{-1} A\). Further, taking into account \cref{lem:Equivalence_Diag_Precs}, we get}
    \Rayleigh( T, \Prec[], \bfy )
        &\leq 
        \frac{ \langle \bfy, A \bfy \rangle }{ \langle \bfy, \Prec[] \bfy \rangle }
        =
        \frac{ \langle \bfy, A\bfy \rangle }{ \langle \bfy, \Prec[a] \bfy \rangle } \frac{ \langle \bfy, \Prec[a] \bfy \rangle }{ \langle \bfy, \Prec[] \bfy \rangle }
        \leq 
        \frac{ \langle \bfy, A\bfy \rangle }{ \langle \bfy, \Prec[a] \bfy \rangle }.
    \end{align*}
    As a result, the upper bound \cref{eq:Bounds_on_prec_jacobi_diagonal} holds.
    
    Now, the smallest eigenvalue of \( \pi_2 ( \UnitSim[X]{D}^{-1} \UnitSim[X]{A} )\) is just the algebraic connectivity of \( \UnitSim[X]{D}^{-1} \UnitSim[X]{S} \). Here, notice that
\[
	a(\Prec[]^{-1} T)  
	= a( \Prec[]^{-\nicefrac 1 2} T \Prec[]^{-\nicefrac 1 2} ) 
	= a\del{ \Prec[]^{-\nicefrac 1 2} A \Prec[]^{-\nicefrac 1 2} - \frac{\lambda}{n} \Prec[]^{-\nicefrac 1 2} \Ones[n,n] \Prec[]^{-\nicefrac 1 2} },
\]
and $a(\Prec[]^{-1} A) =
	a( \Prec[]^{-\nicefrac 1 2} A \Prec[]^{-\nicefrac 1 2} )$. Now, by Weyl's inequality \cite[Corollary 8.1.6]{Golub2013}, we have that
\[
	a(\Prec[]^{-1} A) \leq a(\Prec[]^{-1} T) + \frac{\lambda}{n} \rho\del{\Prec[]^{-\nicefrac 1 2} \Ones[n,n] \Prec[]^{-\nicefrac 1 2} }.
\]
Thus, we obtain 
\[
	a(\Prec[]^{-1} T)  \geq  a(\Prec[]^{-1} A) -  \frac{\lambda}{n} \rho\del{\Prec[]^{-\nicefrac 1 2} \Ones[n,n] \Prec[]^{-\nicefrac 1 2} }
	=
	a(\Prec[]^{-1} A) -  \frac{\lambda}{n} \rho\del{\Prec[]^{-1} \Ones[n,n] } .
\]
    At this point, let us define \( \bfu \coloneqq \diag(\Prec[a]^{-\nicefrac 1 2})\), and observe that from \cref{lem:Equivalence_Diag_Precs}:
    \[
        \max_{ \bfy }
        \frac{ \langle \bfy, \Ones[n,n] \bfy \rangle }{\langle \bfy, \Prec[a] \bfy \rangle}
        \frac{\langle \bfy, \Prec[a] \bfy \rangle}{\langle \bfy, \Prec \bfy \rangle}
        \leq
        \max_{ \bfx } 
        \frac{ \langle \bfx, \Prec[a]^{-\nicefrac 1 2} \Ones[n,n] \Prec[a]^{-\nicefrac 1 2} \bfx \rangle }{\langle \bfx, \bfx \rangle}
        =
        \max_{ \bfx } 
        \frac{ \langle \bfx, \bfu \bfu^\top \bfx \rangle }{\langle \bfx, \bfx \rangle}
        \leq \|\bfu\|_2^2.
    \]
    Here \( \bfu = (\lambda + \mu \etab)^{\circ -\nicefrac 1 2} \) and by the \(\ell_2\)--\(\ell_\infty\) norm equivalence, we have that
    \( \| \bfu \|_2^2 \leq n \|\bfu\|_\infty^2 = n (\lambda + \mu \min\etab)^{-1} \). Going back to the algebraic connectivity of \( \Prec[]^{-1} T\), we obtain
    \begin{equation}
        \label{eq:Lower_bound_Alg_Conect}
        a(\Prec[]^{-1} T) \geq a(\Prec[]^{-1} A) - \frac{\lambda}{\lambda + \mu \min\etab}.
    \end{equation}
\end{proof}
\begin{remark}
    Notice that if \( \lambda \ll \mu \min \etab\), then the last fraction in \cref{eq:Lower_bound_Alg_Conect} is negligible. Something similar happens as well when \( n \gg 1\), as the term \( (\nicefrac{\lambda}{n} \Ones[n,n]) \) could stop being significant in floating point arithmetic. 
\end{remark}

% ----------------------------------------------------- %
\subsubsection{Comparison of preconditioners}\label{sec:Initial_Comparison}

At this point, let us observe how the preconditioners analyzed in \cref{sec:Diagonal_Preconditioners,sec:Dense_Preconditioners} transform the eigenvalues of a particular instance of the ANOVA kernel. For this, we assembled a kernel using the image in \cref{sub:Imgs_Cell}, and computed the eigenvalues of the exact unnormalized extended Gaussian ANOVA kernel based on a scaled sample with \(n = 3\,630\) and \(\mathsf{L} = 41\).

\textbf{Effects of \(\sigma\).}
To build any given kernel, we need to select a real value of \(\sigma > 0\). This parameter plays a critical rôle in shaping the behavior of each kernel. It determines the smoothness and spread of the function \( s(\sigma; r) \coloneqq \exp\{- (\nicefrac{r}{\sigma} )^2\} \), which in turn influences the eigenvalues of the kernel. Observe that \( s(\sigma; r) \) is bounded with values in \( [0,1]\). Moreover, for any nontrivial \( r \) it holds that
\[
	\lim_{\sigma \to \infty} s(\sigma;r) =  \lim_{y \downarrow 0} e^{-y} = e^{0} = 1
	\qquad\text{and}\qquad
	\lim_{\sigma \to 0} s(\sigma;r) = \lim_{Y \uparrow \infty} e^{-Y} = 0.
% \vspace{-1em}
\]
As \(\sigma\) increases, \(s\) approaches \(1\), indicating a smoother, broader interaction between features. Conversely, as \(\sigma\) approaches \(0\), \(s\) trends towards \(0\), reflecting sharper, more localized interactions. In the trivial case \(r = 0\), the value of \(\sigma\) does not have an effect on the interaction.

Since the off--diagonal entries of \(\Gamma\) are given by terms of the form \( \frac{1}{{\mathsf L}} \sum\limits_{\ell=1}^{\mathsf L} s(\sigma; (r_\ell)_{i,j})\) for all \( i\neq j\), we can conclude that \( \sigma\) is a parameter that interpolates between one of the matrices in the set \( \{\Zeros[n,n], \mathbf{E} \} \) and the matrix of ones \( \Ones[n,n]\); i.e., the map \(\sigma \mapsto A\) is an interpolator in the set of graph Laplacians. Here, the symmetric matrix \(\mathbf{E}\) has nonnegative entries in the range \([\mathsf{L}^{-1},1]\) which only occur in the entries where at least one \((\gamma_\ell)_{i,j}\) is equal to one. From the upper end of the interpolation, the uniform bound \( \rho(A) \leq \lambda + \mu n \) holds for any \(\sigma > 0\). Let us label the resulting limit operators as \( A_{\sigma \to \infty} \coloneqq \lambda \Idn + \mu( n \Idn - \Ones[n,n] ) \), \( A_{1,\sigma \to 0} \coloneqq \lambda \Idn \) and \( A_{2,\sigma \to 0} \coloneqq \lambda \Idn + \mu ( \diag(\mathbf{E}\Ones) - \mathbf{E}) \). Notice that for the last limit case, it is possible for the matrix \( \mathbf{E} \) to be associated with a disconnected graph. In fact, any \(\Gamma_\ell\) within such a regime will either weight a complete graph or a disconnected graph; thus, \(\mathbf{E}\) represents a convex combination of weights associated with disconnected and complete graphs.

The smoothness of \(s\) suggests that the eigenvalues of \(A\) and the preconditioned operators will also vary smoothly with respect to \(\sigma \in [10, 1.4 \times 10^3]\). Effectively, this is showcased in \cref{fig:Sigma_Eigs}. Here, four panels depict how different values of \(\sigma\) result in different spectra and condition numbers for \(A\) and its preconditioned variants. The values of \(\sigma\) were chosen such that the gap between \( \rho(A) \) and \(a(A)\) was larger than 1. The eigenvalue curves were subsampled to improve the visualization. Furthermore, the eigenvalue maps of \( \Prec[b]^{-1}A\) and \( \pi_2 ( \mathsf{D}^{-1}_{b,X} \UnitSim[X]{A} ) \) were not included as they behave in a very similar way as \( \Prec[a]^{-1}A\) and \( \pi_2 ( \mathsf{D}^{-1}_{a,X} \UnitSim[X]{A} ) \), presenting almost identical curves and ranges of values. All eigenvalues were computed to machine precision.

% ------------------------------------- %
\begin{figure}[ht!]
\centering
    \subcaptionbox{Eigenvalues of $A$ with varying \(\sigma\) \label{sub:Eigs_A_Sigma}}
    {
    \includegraphics[scale=0.6,draft=false]{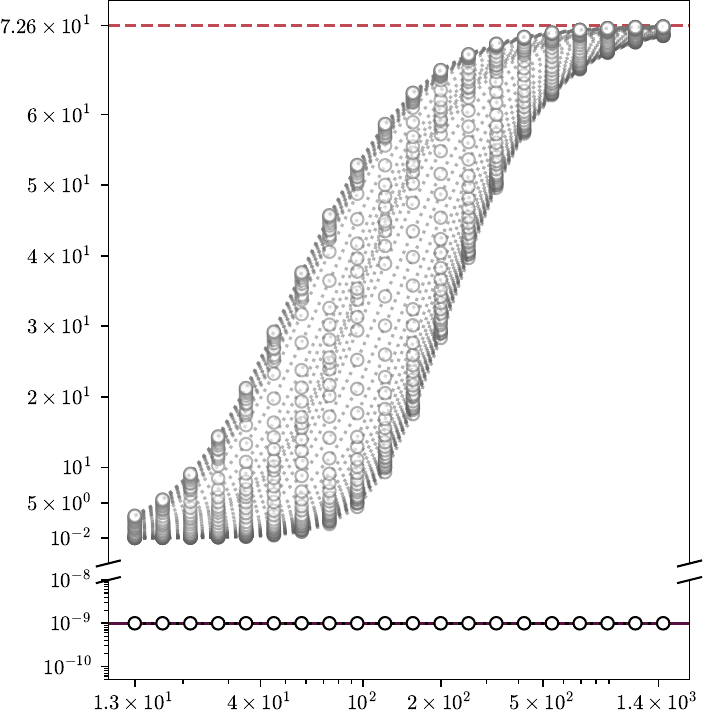}
    }
    \hspace{0.2em}
    \subcaptionbox{Eigenvalues of \(\Prec[a]^{-1} A\) with varying \(\sigma\) 
    \label{sub:Eigs_P_a_A_Sigma}}
       {\includegraphics[scale=0.6,draft=false]{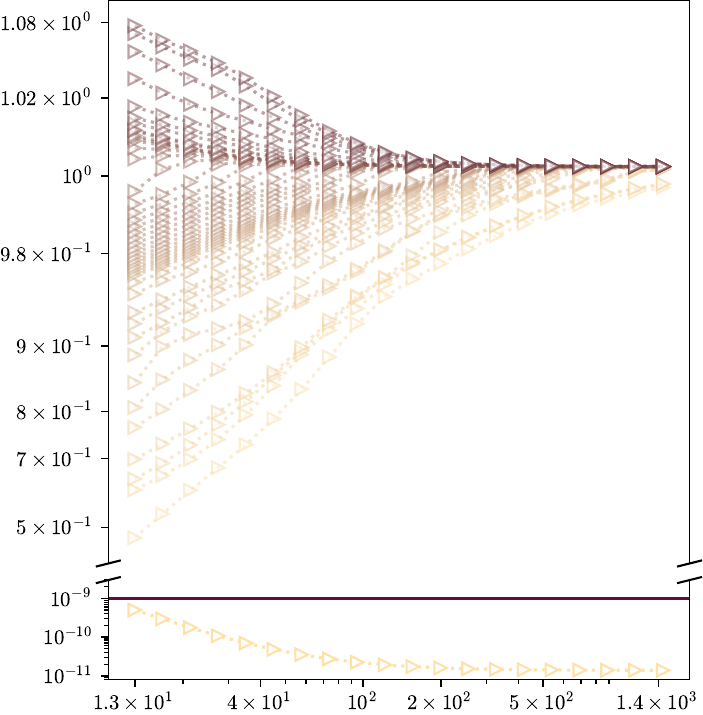}}
    \\
    \subcaptionbox{Comparison of condition number with varying \(\sigma\) \label{sub:Cond_A_Sigma}}
    {
    \includegraphics[scale=0.6,draft=false]{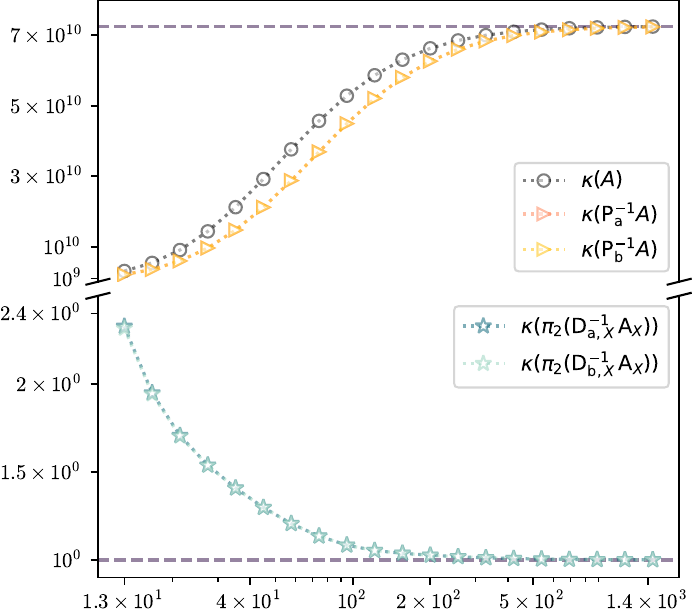}
    }
    \hspace{0.2em}
    \subcaptionbox{Eigenvalues of \( \pi_2 ( \mathsf{D}^{-1}_{a,X} \UnitSim[X]{A} ) \) with varying \(\sigma\) 
    \label{sub:Eigs_D_a_A_Sigma}}
       {\includegraphics[scale=0.6,draft=false]{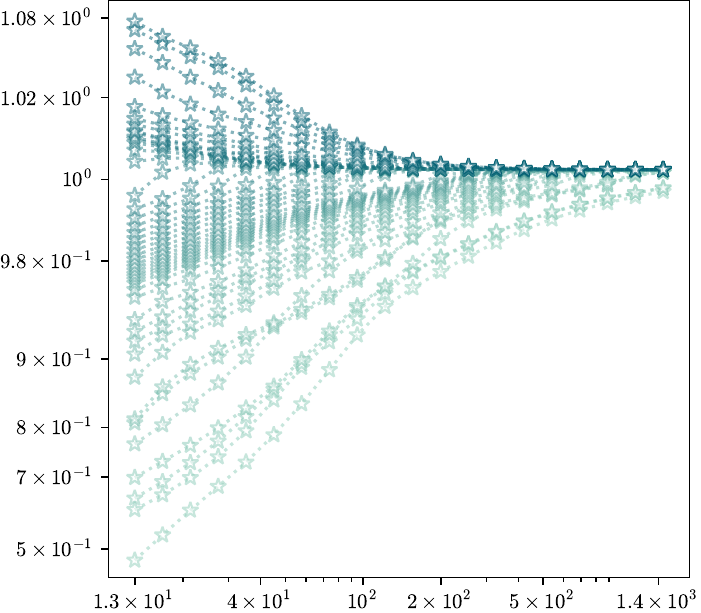}}
    \vspace{-1em}
    \caption{
    Evolution of eigenvalues and condition number as \(\sigma\) varies in the interval \([10,1.4 \times 10^3]\) and with \(\lambda = 10^{-9}\). All the horizontal axes are displayed in logarithmic scale.
    Each ordered quantity is connected by a dotted line as \(\sigma\) increases, suggesting a continuous effect of \(\sigma\).
    \newline
    \textbf{Panel a} displays each point of \(\Sigma(A)\) corresponding to a fixed value of \(\sigma\). A plum horizontal solid line at \(\lambda\) highlights the smallest eigenvalue. The dashed red line at the top represents the upper bound \(\lambda + \mu n\). 
    \textbf{Panel b} represents \(\Sigma(\Prec[a]^{-1} A)\) as \(\sigma\) evolves. The vertical axis features a scaling around 1 for improved visualization.
    \textbf{Panel c} showcases the evolution of the condition number of \( A\) and its preconditioned variants. The upper and lower bounds in dashed purple lines represent \( 1\) and \( 1 + \lambda^{-1} \mu n \).
    \textbf{Panel d} represents \(\Sigma( \pi_2 ( \mathsf{D}^{-1}_{a,X} \UnitSim[X]{A} ) )\) as \(\sigma\) evolves. The vertical axis features a scaling around 1.
    }
\label{fig:Sigma_Eigs}
\vspace{-1\baselineskip}
\end{figure}

In panel (a), we observe that small and large values of \(\sigma\) result in clustered eigenvalues, as more information of the underlying noisy samples is trimmed by the bounds of \(s\); i.e., \(s\) acts as a cutoff or activation function. Notice that several choices of \(\sigma\) result in spectral radii that are close to the upper bound \(\lambda + \mu n \). For moderate values of \(\sigma\), we instead observe that \(\Sigma(A)\) can be spread between one and two orders of magnitude.

In panels (b) and (d), we observe the eigenvalues of the preconditioned operators \( \Prec[a]^{-1} A\) and \( \pi_2 ( \mathsf{D}^{-1}_{a,X} \UnitSim[X]{A} ) \). We observe that all the eigenvalues with preconditioning are effectively clustered within the interval \( [10^{-11}, 1.08]\). The clustering effect improves for \( \sigma > 50\) by defining a cluster around \(1\) for both operators, with a gap below \(10^{-1}\), and a second cluster for \( \Prec[a]^{-1} A\) isolating the smallest eigenvalue, which is closer to \(\lambda\). Let us focus on \( \Prec[a]^{-1} A\), as its spectrum controls the eigenvalues of \( \pi_2 ( \mathsf{D}^{-1}_{a,X} \UnitSim[X]{A} ) \) due to \cref{lemma:Spectral_Control_Pa_on_D}. Per our discussion on \(\sigma\), the limit operator \( A_{\sigma \to \infty} \) has a two point spectrum given by \( \Sigma(A_\infty) = \lambda + \mu \{0,n\}\), and its spectral radius presents a geometric multiplicity of \(n-1\). The limit preconditioner is just a scaling of the identity, yielding 
\( 
    \Sigma( \Prec[a]^{-1} A_\infty) =  \big\{ \lambda [\lambda + \mu (n-1)]^{-1}, (\lambda + \mu n)[\lambda + \mu (n-1)]^{-1}  \big\} .
\)
Here, the preconditioned spectral radius is bounded within \( [1,2]\) for \(n \geq 2\) and quickly approaches \(1\) even for small values of \(n\), regardless of the values of \(\mu\) and \(\lambda\). As a result, it is not surprising that the first cluster quickly pushes the second--to--last eigenvalues towards \(1\), while on the other hand the first preconditioned eigenvalue approaches \( \lambda [\lambda + \mu (n-1)]^{-1}\), which in this case is approximately \( 10^{-11} \). Notice that this result is independent of the off--diagonal values of \( \Gamma_\ell\), as any zero value will remain unchanged as \(\sigma\) increases. By contrast, when \(\sigma\) becomes small, the behavior of \(A\) and its preconditioned variants is based on the limit cases \(A_{\sigma\to 0}\). In one setting, if there is no matrix \( \Gamma_\ell \) with off--diagonal values equal to \(1\), then we have the limiting case \( A_{1,\sigma\to 0} = \lambda \Idn\). As a result, the eigenvalues of the preconditioned system will be clustered around \(1\) when \(\sigma \ll 1\). On the other hand, if there are off--diagonal entries equating \(1\) for any \( \Gamma_\ell \), then \(A\) will get close to \( A_{2,\sigma\to 0} \) as \(\sigma\) becomes smaller. From \(\mathbf{E}\) we can identify the connected components of the underlying graph and its isolated nodes are also given by any zero entries in \( \mathbf{E}\Ones\). Let us suppose that there are \( d_1\) isolated nodes and \( d_2\) connected components with at least two nodes. Then \( \lambda \) will have geometric multiplicity \( d_1 + d_2 + 1\) for \( A_{2,\sigma\to 0} \). With preconditioning, any eigenvalue associated with an isolated node will have value \(1\). We can then decompose \( \mathbf{E} \) into \(d_2\) connected components of sizes \((n_j)_{j \in \llb 1,d_2\rrb}\), and work independently with the submatrices \( \lambda \Idn[n_j] + \mu \diag(\mathbf{E}_j \Ones[n_j] ) - \mu \mathbf{E}_j \). Here, the eigenvalues under preconditioning will span \((0, 2]\). An interesting case, which is depicted in panels (b) and (d), occurs when \(\sigma\) has an intermediate value, say \( 1 < \sigma < 50\): here, we observe that the entries in \( \mathbf{E}\) become relevant, hence the spectrum of \( \Prec[a]^{-1} A \) starts to spread out in \( (0,2]\). For cases when \( A_{1,\sigma \to 0} \) is the limit case, intermediate values of \(\sigma\) yield a similar behavior, as some entries in each \(\Gamma_\ell\) get pulled down to \(0\) but entries close to \(1\) remain; thus, a spread of the spectrum of the preconditioned system in \((0,2]\) is expected.

In panel (c), the condition number of \(A\) and its preconditioned variants are included. In this case, the condition number when applying \( \Prec[b]^{-1} \) and its projection through \(X\) was included as well. For large values of \(\sigma\), we can observe that the unpreconditioned operator and the diagonal preconditioners display a condition number that steadily increases up to \( 1 + \lambda^{-1} \mu n \). This behavior is determined by the value of the smallest eigenvalue. In contrast, the dense preconditioners display a condition number approaching \(1\). This can be explained as \(\Prec[a]^{-1} A\) displays a decreasing gap between the algebraic connectivity and its spectral radius when \(\sigma\) is large, as seen in panel (d). It is without surprise that the condition number for moderate and small values of \(\sigma\) will depend on the limit \( A_{\sigma \to 0}\). For moderate values, the spread of the eigenvalues with preconditioning in \( (0,2]\) will increase the condition number for the dense preconditioners, while the opposite behavior occurs with the diagonal preconditioners as the smallest eigenvalue grows. As \(\sigma\) goes below \(1\), we can expect a limiting condition number strictly greater than \(1\) due to the underlying presence of \( \mathbf{E}\), or a condition number equal to \(1\) for the limit \( A_{1,\sigma\to 0}\).

% ------------------------------------- %
\begin{figure}[h!]
\centering
	\includegraphics[scale=0.65,draft=false]{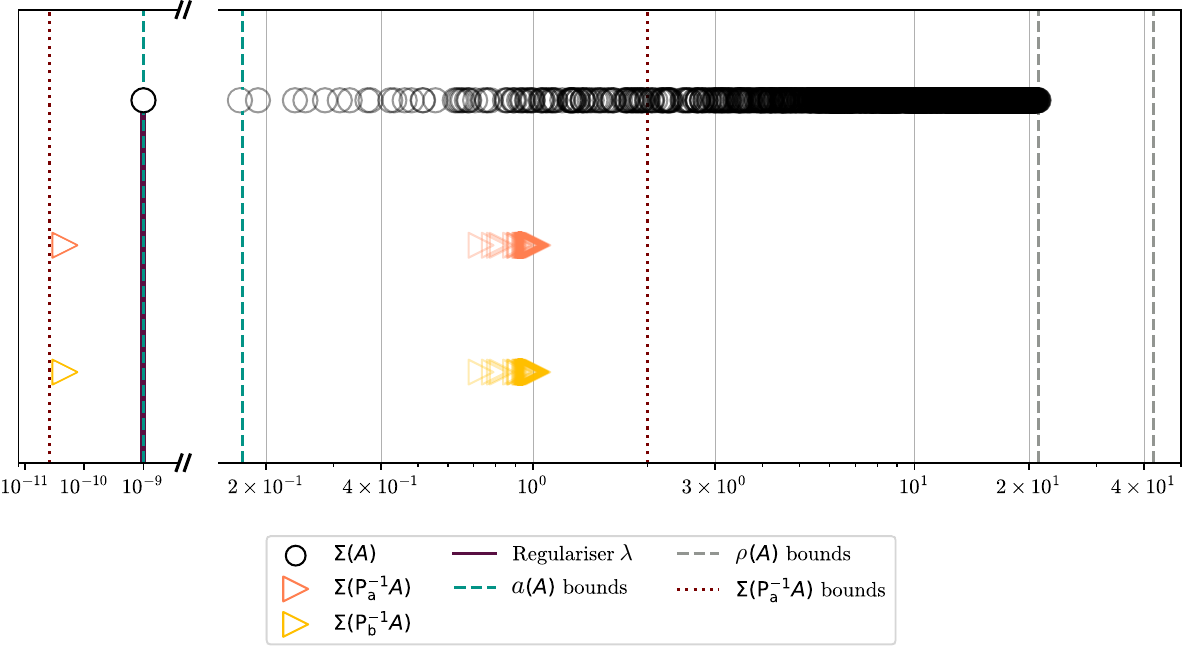}
\caption{Comparative display of the eigenvalues of \(A\), \(\Prec[a]^{-1} A\), and \(\Prec[b]^{-1} A\) with \( \lambda = 10^{-9}\). The horizontal axis quantifies each eigenvalue, while the vertical alignment is non--quantitative and solely for visual separation. 
\newline
\textbf{Top row (black $\ocircle$ markers):} \(\Sigma(A)\) is represented. The grey dashed lines on the right are the bounds for \(\rho(A)\) as derived in \cref{r:general_spectral_bounds}. The bounds for \(a(A)\), as per \cref{lemma:UpperConnectivity,lem:lower_bound_algebraic_connectivity}, are depicted with teal dashed lines. A plum vertical solid line at \(\lambda\) highlights the smallest eigenvalue.
\textbf{Middle row (orange $\vartriangleright$ markers):} \(\Sigma(\Prec[a]^{-1} A)\) is represented. The upper and lower bounds on this set based on \cref{lem:Rayleigh_Diag_Jac} are included in dotted red lines.
\textbf{Bottom row (yellow $\vartriangleright$ markers):} \(\Sigma(\Prec[b]^{-1} A)\) is represented. The bounds from \cref{lem:Prec_Norm_2_Diag} are omitted for visual clarity.
}
\label{fig:Eigs_and_Diagonal_Precs}
\vspace{-1.5\baselineskip}
\end{figure}
% ------------------------------------- %

\vspace{0.5\baselineskip}
%\newpage
\textbf{Comparison at working range.}
In \cref{sub:Eigs_A_Sigma}, we depicted the effects of \(\sigma\) on the spectrum of \(A\) whenever the condition \( \rho(A) - a(A) > 1\) is satisfied. This condition allows \(A\) to retain as much information as possible from the noisy data used to build the ANOVA kernel, without the caveats of either reducing \(\Gamma\) to a matrix of ones minus the identity (\(\sigma\) large) or allowing \(\lambda\) to have a greater geometric multiplicity (\(\sigma\) small). We will call the range of values of \(\sigma\) that satisfy the condition \( \rho(A) - a(A) > 1\) as the \emph{working range}. In what follows, we will fix \(\sigma\) under this regime and focus instead on the effectiveness of the preconditioners.

\cref{fig:Eigs_and_Diagonal_Precs} showcases the effects of the diagonal preconditioners \(\Prec[a]\) and \(\Prec[b]\) when applied to \(A\). The top row displays the spectrum of \(A\) for \(\lambda = 10^{-9}\). Here, the second--to--last eigenvalues of \(A\) are roughly spread in the interval \( 2 \times [10^{-1}, 10^1]\). The derived bounds on \(a(A)\) and \(\rho(A)\) from \cref{r:general_spectral_bounds,lemma:UpperConnectivity,lem:lower_bound_algebraic_connectivity} are included and suggest we could give an approximate criteria on whether we should precondition \(A\) based on the magnitude of \(\lambda\), as a value of \(\lambda\) larger than the upper bound on \( a(A) \) would already result in a well--conditioned system. The middle and bottom rows depict the eigenvalues of \(A\) preconditioned with \(\Prec[a]\) and \(\Prec[b]\). Here we observe that the bounds \cref{lem:Rayleigh_Diag_Jac} are close to the smallest and largest eigenvalues with preconditioning. Moreover, the algebraic connectivity and spectral radius of the preconditioned matrices tend to \(1\), reflecting the effectiveness of the diagonal preconditioners for moderate and large values of \(\lambda\). In this case, however, as \(\lambda\) is close to zero, the smallest eigenvalue with preconditioning is actually below \(\lambda\), which significantly increases the condition number as observed in \cref{sub:Cond_A_Sigma}.

% ------------------------------------- %
\begin{figure}[htbp]
\centering
	\includegraphics[scale=0.65,draft=false]{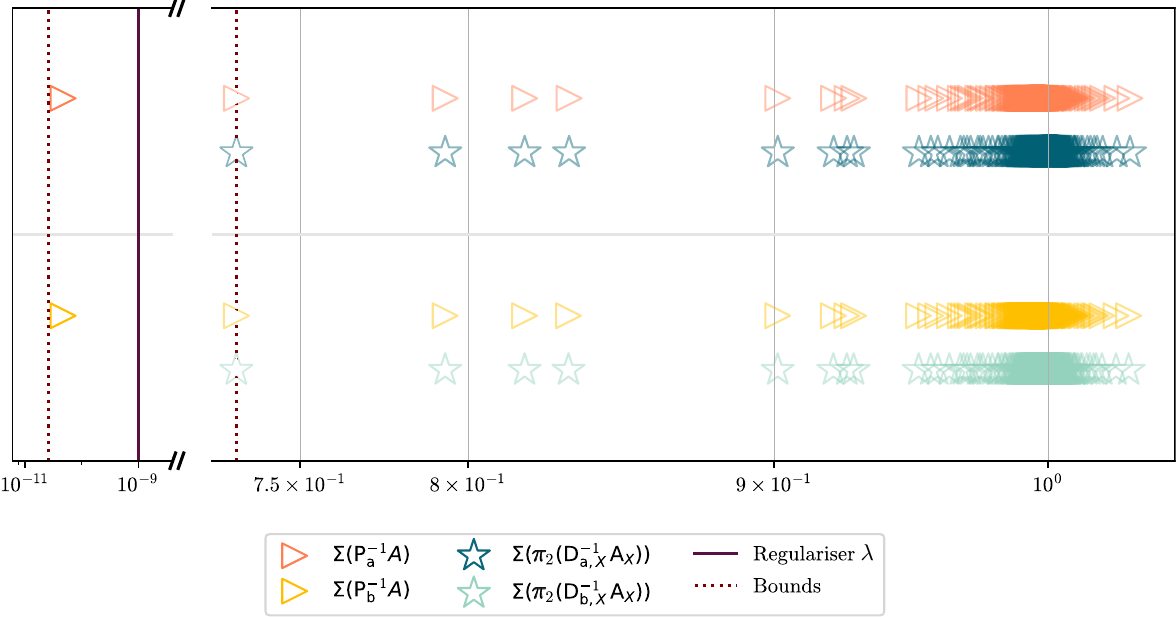}
	%\vspace{1em}
\caption{Comparative display of the eigenvalues of \(\Prec[a]^{-1} A\), \( \pi_2 ( \mathsf{D}^{-1}_{a,X} \UnitSim[X]{A} ) \), \(\Prec[b]^{-1} A\), and \( \pi_2 ( \mathsf{D}^{-1}_{b,X} \UnitSim[X]{A} ) \) with \( \lambda = 10^{-9}\). The horizontal axis quantifies each eigenvalue (in logarithmic scale), while the vertical alignment is non--quantitative and solely for visual separation. 
\newline
\textbf{Top two rows (orange $\vartriangleright$ and dark blue \(\star\) markers):} \(\Sigma(\Prec[a]^{-1} A)\) and \( \Sigma \big( \pi_2 ( \mathsf{D}^{-1}_{a,X} \UnitSim[X]{A} ) \big) \) are represented. 
From left to right, the two dotted lines in red represent the lower bounds in \cref{lem:Rayleigh_Diag_Jac,lemma:Spectral_Control_Pa_on_D}. A plum vertical solid line at \(\lambda\) highlights the smallest eigenvalue of the unpreconditioned operator \(A\).
\textbf{Bottom last two rows (yellow $\vartriangleright$ and light mint \(\star\) markers):} \(\Sigma(\Prec[b]^{-1} A)\) and \( \Sigma \big( \pi_2 ( \mathsf{D}^{-1}_{b,X} \UnitSim[X]{A} ) \big) \) are represented. 
}
\label{fig:Eigs_Proj_and_Diags}
\vspace{-1.25\baselineskip}
\end{figure}
% ------------------------------------- %

In \cref{fig:Eigs_Proj_and_Diags}, the diagonal preconditioners from \cref{sec:Diagonal_Preconditioners} are contrasted against their dense variants introduced in \cref{sec:Dense_Preconditioners}. Let us focus on \( \Prec[a]\) and its dense variant. Visually, we do not observe a significant difference between the two sets of eigenvalues except that the smallest eigenvalue of \( \pi_2 ( \mathsf{D}^{-1}_{a,X} \UnitSim[X]{A} ) \) is now larger than \(\lambda\) and very close to the algebraic connectivity of \( \Prec[a]^{-1} A\). As a result, the dense preconditioner and projection technique are able to compress the eigenvalues of \(A\) in a similar way as the diagonal preconditioner, without the caveat of producing a smaller minimum eigenvalue. Now, if we compare the sets \( \Sigma( \Prec[a]^{-1} A ) \setminus \big\{ \min \Sigma( \Prec[a]^{-1} A ) \big\} \) and \( \Sigma\big( \pi_2 ( \mathsf{D}^{-1}_{a,X} \UnitSim[X]{A} ) \big) \), we do not observe a significant visual difference. This is due to the fact that numerically they are separated by a distance of order \( 10^{-11}\). Let \( c = \min \Sigma\big( \pi_2 ( \mathsf{D}^{-1}_{a,X} \UnitSim[X]{A} ) \big) \), then \( c \geq a(\Prec[a]^{-1} A) - \lambda [\lambda + \mu \min \etab]^{-1} \) by \cref{lemma:Spectral_Control_Pa_on_D}. For the particular choice of parameters and data to generate this example, we obtained that \( \lambda [\lambda + \mu \min \etab]^{-1} \approx 10^{-6} \), \( c < a(\Prec[a]^{-1} A) \), and indeed the inequality was numerically satisfied. A similar behavior is displayed for the second preconditioner \(\Prec[b]\) and its dense variant. Motivated by the spectral equivalence of the preconditioners, as in \cref{lem:Equivalence_Diag_Precs}, we also checked numerically \cref{lemma:Spectral_Control_Pa_on_D}, and found that the pairwise distance between the two sets of eigenvalues is of order \(10^{-11}\).

% ------------------------------------------------------------------------------------------- %
% ------------------------------------------------------------------------------------------- %
% ------------------------------------------------------------------------------------------- %
\section{Numerical experiments}\label{sec:Numerics}

In this section we test and showcase the effectiveness of using NFFT for fast summations when solving systems of the form \cref{eq:discretized-state} and the performance of the preconditioners studied in the previous section.

The unnormalized extended Gaussian ANOVA kernel \(\Gamma\) has the form of a composite nonlocal filter, hence it is natural to test it in the context of image denoising. To assemble \(\Gamma\) and each subkernel \(\Gamma_\ell\), we will consider a similar approach to \cite{D’Elia2021}. Given a clean image \(\bfu_{\textsf c}\) with \(n\) pixels, we generate some additive noise \( \mathbf{h} \sim \mathcal{N}(\Zeros,s^2\Idn)\) to obtain a noisy variant \( \bff \coloneqq \bfu_{\textsf c} + \mathbf{h}\). Then, a set of patches based on \(\bff\) is formed by tracing closed square regions of radius \(\rho\) for each pixel. We pad with zeros as needed in the areas of pixels that are below a distance $\rho$ to the border. This results in a set of features of size \( n\times (2\rho+1)^2\).

Due to the limitations of the NFFT regarding the size of the feature space, we process the features into windows of size \( n \times \hat{d}\), where \( \hat{d} \in \{1,2,3\}\). For the generation of such windows of features, we follow the approach of \cite{Nestler2022,WaPeSt2023} which relies on the Mutual Information Score (MIS) to rank and identify the most relevant features before separating them into \(\mathsf{L} = \lceil (2\rho+1)^2 /3 \rceil\) windows. Then we assemble \(\Gamma\) as an operator that weights the action of each subkernel \(\Gamma_\ell\) for \(\ell \in \llb 1,\mathsf{L} \rrb\). If the available memory allows it, we also store \(\Gamma\) in its matrix form for comparison purposes.

The feature engineering step is crucial for reducing the computational complexity of kernel evaluation for large--scale instances while ensuring that the most informative components are retained. While the use of small--sized features allows us to use the NFFT, more windows lead to more kernels what leads to more large--scale kernel evaluations for solving dense linear systems; as a result, we ought not incorporate all possible feature interactions. Specifically, by computing MIS, we can take the noisy pixel values in \(\bff\) as a target that acts as a proxy for the underlying clean image, and these are contrasted against the patches. Thus we obtain features that are statistically predictive of the observed data with a low computational cost. However, the MIS is a univariate measure that does not examine the impact of a combination of several features on predicting the target. In order to study the computational tradeoff of feature combination and selection to capture linear and nonlinear relationships, \cite{Wagner2024} provides a thorough study on alternative approaches for feature engineering. Moreover, the study of the patch distribution \cite{Piening2024} could potentially be used in future work to develop alternative purpose--designed windows targeted for denoising and other imaging tasks.

A nonlocal image denoising task approximates \( \bfu_{\textsf c}\) with the solution of \( (\lambda \Idn + \mu L) \bfu = \lambda \bff\). The value of the regularizer \(\lambda\) determines how much of the information in \(\bff\) is retained. The value of \(\mu\) works as a scaling parameter that weights the action of the ANOVA kernel and, in principle, can be omitted. However, in practice this parameter is useful to control the range of intensity values of the product \(\mu L\bfu\) and prevent cancellation errors. Notice that, if we fix \(\mu\), then a different choice for \(\lambda\) will result in a different approximation of \(\bfu_{\textsf c}\).

This section is structured as follows. In \cref{ssec:Pics}, two image datasets are presented, which we use to set up and run our experiments. \Cref{ssec:Op_Time} showcases the computational gain of using the NFFT for fast matrix--vector products. It is quickly seen that, as the dimension of the problem grows, explicitly computing and storing each subkernel becomes a computational burden that the NFFT easily overcomes. This is again exemplified in \cref{ssec:LinSys}, where the two versions of the operator are employed to solve a denoising problem using direct and iterative methods. In \cref{ssec:CG_Compare} the different preconditioning techniques developed in \cref{sec:Prec_GLs} are tested when using preconditioned CG to solve a denoising task. Finally, \cref{ssec:Learning} solves a parameter learning problem for recovering a dataset of images using the preconditioned CG iteration and the NFFT fast summation method.

All the experiments were performed on a MacBook Pro 2020 M1 with 16 GB RAM. The code was written in Python 3.9 and relies mainly on \texttt{NumPy} 1.23, \texttt{SciPy} 1.11.1, and \texttt{FastAdj} 0.2, \texttt{NFFT4ANOVA}, and \texttt{FastAdjacency2} which yield an interface with the \texttt{NFFT3} library \cite{Keiner2009}. The NFFT does not require the use of costly hardware or access to HPC systems, which facilitates its broad and accessible application.

% ----------------------------------------------------- %
\subsection{Images}\label{ssec:Pics}

Our set of experiments will be based on two image datasets. We denote by \(n_1\) and \(n_2\) the vertical and horizontal size of a given picture in the image domain, from which we obtain the computational dimension \(n \coloneqq n_1 n_2\). We considered single channel grayscale images for illustrating the method; however, in principle, color information can be added to the feature set resulting only in additional windows but not dampening the performance of the NFFT--based algorithms. The first dataset, used for the numerical experiments in \cref{sec:Initial_Comparison} and \cref{ssec:Op_Time} to \cref{ssec:CG_Compare}, is displayed in \cref{fig:Experimental_images}. These images are sourced from the \emph{Test and Example} data module of the \texttt{scikit-image} package \cite{scikit-image}. \cref{sub:Imgs_Cell} and \cref{sub:Imgs_Brain} feature low contrast with close--to--flat regions. \cref{sub:Imgs_Kidney} instead has high contrast with several jumps in intensity in the depicted patterns. Finally, \cref{sub:Imgs_Cat} features low and high contrast regions, patterns, and textures. The second dataset, sourced instead for parameter learning in \cref{ssec:Learning}, comprises pictures of cats generously provided by individuals who granted consent to their use for this research. The variety of cats and backgrounds implies that any recovered image will feature multiple jumps in accordance to the presence of patterns, textures, and areas of low and high contrast mixed all together.

% ------------------------------------- %
\begin{figure}[ht!]
\centering
    \subcaptionbox{Cell \label{sub:Imgs_Cell}}
    {
    \includegraphics[height=0.17\textwidth,draft=false]{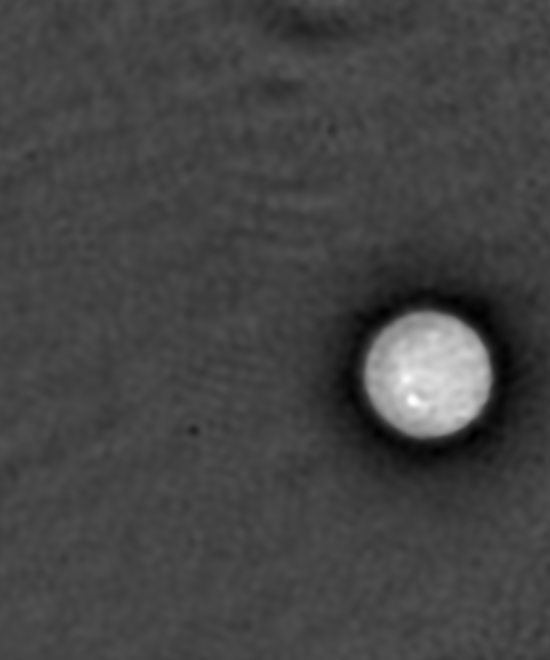}
    }
    \hspace{0.2em}
    \subcaptionbox{Brain \label{sub:Imgs_Brain}}
       {\includegraphics[width=0.17\textwidth,draft=false]{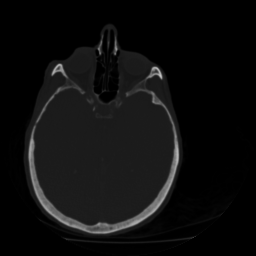}}
    \hspace{0.2em}
    \subcaptionbox{Kidney \label{sub:Imgs_Kidney}}
       {\includegraphics[width=0.17\textwidth,draft=false]{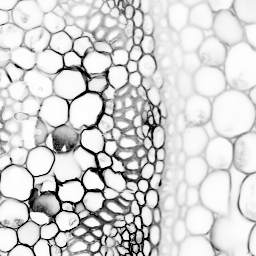}}
    \hspace{0.2em}
    \subcaptionbox{Chelsea the cat \label{sub:Imgs_Cat}}
       {\includegraphics[width=0.2125\textwidth,draft=false]{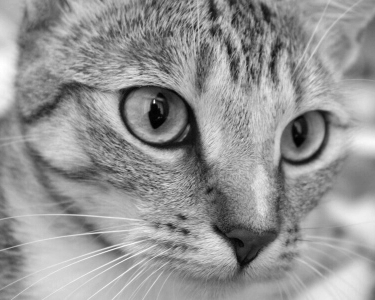}}
    %\\
    %
    %
    %
    %
    \vspace{-1em}
    \caption{Images used for numerical tests.}
\label{fig:Experimental_images}
\vspace{-2\baselineskip}
\end{figure}
% ------------------------------------- %

% ----------------------------------------------------- %
\subsection{NFFT -- Operator time}\label{ssec:Op_Time}

This section presents a comparative performance analysis of two distinct methodologies for initializing an instance of the unnormalized extended Gaussian ANOVA kernel, \(\Gamma \in \mathcal{M}_{n}(\R)\). Specifically, we contrast a NFFT--based fast summation method against precomputing and storing each entry of a dense matrix. The comparison focuses on the computational time required to calculate the matrix--vector product \(\Gamma \bfv\) for a given vector \(\bfv \in \R^n\) as the problem size, \(n\), increases.

For these tests, our evaluation is based on a set of scalings of \cref{sub:Imgs_Brain}, originally \(256 \times 256\) pixels. The kernel \(\Gamma\) is assembled with fixed parameters \(\sigma = 30\) and \( \mathsf{L} = 17\) windows. The number of windows comes from tracing patches with \(49\) pixels around each individual pixel of the image; these features are subdivided into \(16\) windows of size three and one window of size one. This particular patch size enables a detailed yet computationally efficient local representation of the image. Notwithstanding, if we increase the patch size, the computation times will increase approximately linearly with respect to \(\mathsf L\).

% ------------------------------------- %
\begin{table}[h!]
\setlength{\tabcolsep}{0.4em}
\centering
%\fontsize{6.5}{7.5}\selectfont
%\footnotesize
\fontsize{8}{9}\selectfont
\def\arraystretch{1.15}
\begin{tabular}{cr rrrrrrrrrr}
\toprule
 \bf Dimensions & \multirow{2}{*}{\bf Size \(n\)} & \multirow{2}{*}{\bf MIS} &&  \multicolumn{3}{c}{\bf NFFT--based kernel}
 &&  \multicolumn{3}{c}{\bf Exact kernel}
 \\  \cline{5-7} \cline{9-11}
 \((n_1,n_2)\) &&&& \bf Setup & \bf Fast \(\Gamma \bfv\) & \bf Total && \bf Setup & \bf Exact \(\Gamma \bfv\) & \bf Total
 \\[-0.2em]
\midrule
 (16, 16) &      256 & 1.46 \sf s &&  19 \sf ms & 100 \sf ms &  120 \sf ms &&  19 \sf ms &  13 \(\mu\)\sf s &  19 \sf ms \\
 (21, 21) &      441 & 1.74 \sf s &&  22 \sf ms & 105 \sf ms &  128 \sf ms &&  51 \sf ms & 14 \(\mu\)\sf s &  52 \sf ms \\
 (37, 37) &   1\,369 & 1.97 \sf s &&    34 \sf ms &  224 \sf ms & 259 \sf ms && 410 \sf ms &  1 \sf ms &   411 \sf ms \\
(49, 49) &        2\,401 & 2.32 \sf s && 46 \sf ms & 290 \sf ms & 337 \sf ms && 1.15 \sf s & 2 \sf ms &     1.15 \sf s \\
 (64, 64) &       4\,096 & 2.74 \sf s && 60 \sf ms & 337 \sf ms & 397 \sf ms && 4.17 \sf s &  4 \sf ms &   4.17 \sf s \\
 (84, 84) &       7\,056 & 3.80 \sf s && 95 \sf ms &  469 \sf ms & 564 \sf ms && 50.09 \sf s &  14 \sf ms &  50.10 \sf s \\
 (111, 111) &    12\,321 & 4.54 \sf s && 147 \sf ms &  665 \sf ms & 812 \sf ms && 3 {\sf min} 46 \sf s & 26 \sf ms & 3 {\sf min} 46 \sf s \\
 (147, 147) &    21\,609 & 6.13 \sf s && 242 \sf ms &  1.03 \sf s &    1.28 \sf s && \MaxT & \MaxT & \MaxT \\
 (194, 194) &    37\,636 & 8.78 \sf s && 412 \sf ms &   1.63 \sf s &    2.05 \sf s && \MaxT & \MaxT & \MaxT \\
 (256, 256) &    65\,536 & 13.44 \sf s && 716 \sf ms &   2.77 \sf s &    3.48 \sf s && \MaxT & \MaxT & \MaxT \\
 (338, 338) &   114\,244 & 21.54 \sf s && 1.21 \sf s &   4.76 \sf s &    5.97 \sf s && \MaxT & \MaxT & \MaxT \\
 (446, 446) &   198\,916 & 36.14 \sf s && 2.12 \sf s &   8.22 \sf s &   10.34 \sf s && \MaxT & \MaxT & \MaxT \\
 (588, 588) &   345\,744 & 1 {\sf min} 5 \sf s && 4.02 \sf s &  14.27 \sf s &   18.29 \sf s && \MaxT & \MaxT & \MaxT \\
 (776, 776) &   602\,176 & 1 {\sf min} 54 \sf s && 7.32 \sf s &  24.67 \sf s &   31.99 \sf s && \MaxT & \MaxT & \MaxT \\
 (1\,001, 1\,001) & 1\,002\,001 & 3 {\sf min} 20 \sf s && 12.24 \sf s &  47.47 \sf s &   59.71 \sf s && \MaxT & \MaxT & \MaxT \\
\bottomrule
\end{tabular}
\\[0.5em]
%\footnotetext
\flushleft
\caption{Time comparison of setting up a kernel via a NFFT--based fast summation method versus computing the kernel exactly as the dimension grows. The symbol {\MaxT} is reported whenever the exact kernel was not computable.}
\vspace{-2\baselineskip}
\label{tb:Operator_Times}
\end{table}
% ------------------------------------- %

The results of our comparative tests are summarized in \cref{tb:Operator_Times}. The first two columns describe the size of the problem. The third column reports the time required to compute the mutual information score (MIS) used to assemble the windows of features. This procedure is only required once, with its output used to build both versions of the operator. Subsequent columns (four--to--six and seven--to--nine) compare the two approaches to compute \(\Gamma\) and its action: these sections of the table report the average time associated with the setup process (one--time requirement), the computation of \(\Gamma \bfv\), and the aggregate of these times.

The most time--intensive task for each \(n\) is the calculation of the MIS. However, performing this task is independent of \(\lambda\), \(\mu\), and \(\sigma\), hence there is no need for a recalculation for varying kernel configurations. We observe that both the MIS and the NFFT--based summation grow approximately linearly with respect to \(n\). In contrast, the time required to apply the summation using a dense matrix exhibits quadratic growth \(n^2\), significantly limiting the computational feasibility of this approach. As a result, the exact kernel becomes uncomputable rather quickly before \(n \sim 20\,000\), when a storage of \(50\) GB would already be needed merely to assemble the matrix using all the subkernels. Such a requirement quickly renders the exact kernel computation impractical for large dimensions. Consequently, while it is evident that the computational cost of both methods increases with \(n\), the NFFT--based method remains computationally feasible.

% ----------------------------------------------------- %
\subsection{NFFT -- Solution time}\label{ssec:LinSys}

Here we test the efficiency of solving a linear system by using the NFFT for fast summation and explicitly computing the exact kernel. Similarly as in the previous section, we used a set of scalings of \cref{sub:Imgs_Kidney}, originally \(256\times 256\) pixels, to study how the method behaves as the problem size \(n\) increases. Here we considered \(\sigma = 50\), \( \mu = 10^{-2}\), \(\lambda = 10^{-1}\), and \(\mathsf{L} = 17\). The choice for the regularization parameter is based on the fact that if we scale up the weights of a graph and add nodes and edges, then the magnitudes of the eigenvalues will increase as well. As a result, it is enough for \(\lambda\) to be below the algebraic connectivity of the smallest graph. The system \( A \bfu = \lambda \bff\) was solved using a direct and an iterative method. The exact solution (direct method) was computed using a Cholesky factorization via the \texttt{dposv} routine of \texttt{LAPACK}, which is accessible through \texttt{SciPy}'s linear system solver \cite{2020SciPy-NMeth,lapack}. The iterative method of choice for \(A\) is CG, as it only requires one matrix--vector multiplication per iteration and takes into consideration the symmetry and positive definiteness of the operator. While setting up the solver, we fixed a tolerance of \(10^{-7}\), a maximum of \(200\) iterations, and set \(\bff\) as the initial guess.

% ------------------------------------- %
\begin{table}[h!]
\setlength{\tabcolsep}{0.4em}
\centering
%\fontsize{6.5}{7.5}\selectfont
%\footnotesize
\fontsize{8}{9}\selectfont
\def\arraystretch{1.15}
\begin{tabular}{cr rrrrrrrrrrr}
\toprule
\multirow{3}{*}{\begin{tabular}[c]{@{}c@{}} \bf Dimensions\\ \((n_1,n_2)\) \end{tabular}}
& \multirow{3}{*}{\bf Size \(n\)} & \multirow{3}{*}{ \(\tilde{\kappa}(A)\) } &&  \multicolumn{3}{c}{\bf NFFT kernel}
 &  \multicolumn{4}{c}{\bf Exact kernel}
 \\   \cline{8-11}
 &&&& \multicolumn{2}{c}{\bf with CG}   && \texttt{dposv} && \multicolumn{2}{c}{\bf CG} 
 \\ \cline{5-6}  \cline{10-11}
 &&&& \bf Iters & \bf Time && \bf Time && \bf Iters & \bf Time
 \\[-0.2em]
\midrule
 (16, 16)   &       256 &            \EE[1] &&   10 & 1.27 \sf s && 3 \sf ms && 11 & 1.38 \sf s  \\
 (21, 21)   &       441 & \(2\times\)\EE[1] &&   12 & 1.57 \sf s && 5 \sf ms && 14 & 1.79 \sf s  \\
 (37, 37)   &      1\,369 & \(6\times\)\EE[1] &&   25 & 3.90 \sf s && 28 \sf ms && 29 & 4.47 \sf s  \\
 (49, 49)   &      2\,401 &            \EE[2] &&   26 & 5.13 \sf s && 92 \sf ms  && 31 & 6.01 \sf s  \\
 (64, 64)   &      4\,096 & \(2\times\)\EE[2] &&   35 & 9.21 \sf s && 363 \sf ms  && 41 & 10.71 \sf s  \\
 (84, 84)   &      7\,056 & \(4\times\)\EE[2] &&   38 & 14.59 \sf s && 1.108 \sf s  && 46 & 17.28 \sf s  \\
 (111, 111) &     12\,321 & \(7\times\)\EE[2] &&   53 & 31.46 \sf s && 4.942 \sf s  && 64 & 37.38 \sf s  \\
 (147, 147) &     21\,609 &            \EE[3] &&   55 & 53.12 \sf s && \MaxT &&  \MaxT &   \MaxT \\
 (194, 194) &     37\,636 & \(2\times\)\EE[3] &&   65 & 1 {\sf min} 45 \sf s && \MaxT &&  \MaxT &   \MaxT \\
 (256, 256) &     65\,536 & \(4\times\)\EE[3] &&   71 & 3 {\sf min} 17 \sf s && \MaxT &&  \MaxT &   \MaxT \\
 (338, 338) &    114\,244 & \(8\times\)\EE[3] &&   86 & 6 {\sf min} 54 \sf s && \MaxT &&  \MaxT &   \MaxT \\
 (446, 446) &    198\,916 &            \EE[4] &&   91 & 12 {\sf min} 46 \sf s && \MaxT &&  \MaxT &   \MaxT \\
 (588, 588) &    345\,744 & \(3\times\)\EE[4] &&   97 & 23 {\sf min} 41 \sf s && \MaxT &&  \MaxT &   \MaxT \\
 (776, 776) &    602\,176 & \(4\times\)\EE[4] && 119 & 49 {\sf min} 24 \sf s && \MaxT &&  \MaxT &   \MaxT \\
 (1\,001, 1\,001) & 1\,002\,001 & \(8\times\)\EE[4] && 116 & 1 {\sf h} 29 \sf min && \MaxT &&  \MaxT &   \MaxT \\
\bottomrule
\end{tabular}
\\[0.5em]
\flushleft
\caption{Time comparison for solving the nonlocal system without preconditioning, where the nonlocal operator \(\Gamma\) is either given by the NFFT--based fast summation method or by exact kernel computation. 
The symbol {\MaxT} is reported whenever the exact kernel was not computable.}
\vspace{-2\baselineskip}
\label{tb:Solution_Times}
\end{table}
% ------------------------------------- %

The results of our tests are presented in \cref{tb:Solution_Times}. Similarly as in \cref{tb:Operator_Times}, the two initial columns describe the size of the problem \(n\). The third column displays the approximation of the condition number \( \tilde{\kappa}(A) = 2\lambda^{-1} \mu \max \etab  + 1\), which comes from \cref{eq:CondNumberUnPrecSystem}. The approximation gives us an idea of how difficult it is to solve the system of equations up to a given numerical precision. Subsequent columns, specifically, four--to--five and six--to--seven, compare the two approaches for computing \(\Gamma\) and its action: For the NFFT--based kernel, two columns are included to describe the number of iterations and average time spent with the CG method. A similar structure follows for the exact kernel computation, but an additional column is included to report the average time spent on the direct method \texttt{dposv}.

Naturally, the exact kernel approach is quickly outperformed by the NFFT--based summation method. As storing the whole matrix in memory becomes unpractical, the NFFT operator can easily be used to solve large--scale problems. In terms of the computational times, \texttt{dposv} outperforms CG; however, the rapid cubic growth of this method clearly indicates its dimensional limitations. If we focus our attention on CG, we notice that even for small dimensions the NFFT approach yields faster solution times with a reduced number of iterations. This iteration discrepancy is explained by the computational error of performing the Fast Gauss Transform; see \cref{sec:NFFT_Gauss}. The default parameters in \texttt{FastAdj} yield an approximation error in order to gain computational performance. Parameter studies in \cite{Potts2004,Kunis2006} indicate, based on an estimation of the \(\ell_\infty\) norm, that the error is bounded by \(10^{-5}\) for this choice of parameters. Certainly, parameter tuning of the algorithm can decrease this effect, which would yield similar iteration counts in CG. However, such a study is outside the scope of this paper.

As the time spent in each CG iteration is controlled by the time required to compute \(\Gamma \bfv\), it is of no surprise that both CG times grow at least linearly with respect to \(n\); see \cref{sec:NFFT_Gauss}. In terms of the iteration count, we also observe an increase which is directly explained by the approximation \(\tilde{\kappa}(A)\) and the problem size. The increasing value of the condition number of \(A\) is not only explained by the fact that \(\lambda\) is small but also the fact that, as \(n\) increases, there is more information collected in \(\etab\), yielding a larger spectral radius. To conclude, the increase of the condition number and the computational cost of the NFFT operator pinpoints the clear need for a preconditioner to improve solution times.

% ----------------------------------------------------- %
\subsection{Comparison of preconditioned iterative methods}\label{ssec:CG_Compare}

For this set of experiments, we test the performance of CG for different preconditioner choices in combination with the NFFT fast summation scheme. The considered setup is as follows: for each problem size \(n\), we solved the system \(A \bfu = \lambda \bff\) using CG for different values of \(\lambda\) in decreasing order. In principle, as \(n\) increased and \(\lambda\) became smaller, we could expect each test to become more ill--conditioned than the last. To precondition CG, we considered two bases to work with, namely the original basis where \(A\) is defined and the conjugated basis under the unitary transformation \(U\); i.e., where \(\UnitSim{A}\) is decomposed into a block diagonal system. For the first basis, we considered the unpreconditioned case represented by the preconditioner \(\mathsf{M} = \Idn\), the two diagonal preconditioners from \cref{sec:Diagonal_Preconditioners}; i.e., \(\mathsf{M} \in \{ \Prec[a], \Prec[b] \}\). For comparison with a different class of preconditioning strategies, specifically low--rank methods, we also tested a recently--developed strategy built from the inverse of the Nyström--based randomly pivoted Cholesky approximation \(\widehat{B} \approx B\) \cite{Chen2024}, which we denote \(\Prec[Ch]\). The approximation \( \widehat{B} \) allows us to use the Sherman--Morrison--Woodbury formula \cite[\S 2.1.4]{Golub2013} to approximate the inverse of \(A\) through \(\Prec[Ch]\); see \cite{Cutajar2016,Sun2015,Frangella2023,Abedsoltan2024,Wenger2022} for examples using this technique. For the basis induced by \(U\), we considered similar preconditioners as in the previous case; however, first we decoupled the system \( \UnitSim{A} \bfx = \lambda U^\top \bff \) into its diagonal subblocks and then applied CG to the largest subblock. The unpreconditioned method is labeled, for notational convenience, by \(\mathsf{I}_U\), and the two dense preconditioners developed in \cref{sec:Dense_Preconditioners} are labeled \( \mathcal{P}_{U,\mathsf a} \) and \(\mathcal{P}_{U,\mathsf b}\), where the second subindex indicates which diagonal preconditioner was transformed with \(U\). To sum up, the preconditioners used under the change of basis induced by \(U\) are given by \( \mathsf{M}\in \{\mathsf{I}_U, \mathcal{P}_{U,\mathsf a}, \mathcal{P}_{U,\mathsf b} \} \). Notice that to find a solution in the original basis, the approximate solution \( \bfx_{2:n}\) has to be concatenated with the solution of the one--dimensional block \( x_1\), and then \( \bfu = U \bfx\) transforms the solution back to the original basis.

% ------------------------------------- %
\begin{table}[h!]
\setlength{\tabcolsep}{0.25em}
\centering
%\fontsize{6.5}{7.5}\selectfont
%\footnotesize
\fontsize{8}{9}\selectfont
\def\arraystretch{1.15}
\begin{tabular}{rrr rrrrrr c rrrrrr c rrrrrr c rrrrrr}
\toprule
& \multirow{3}{*}{ \begin{tabular}[r]{@{}c@{}} \bf Est.\\ \bf Memory\end{tabular} } && \multicolumn{27}{c}{\bf Regularization value \(\lambda\)}
\\  \cline{4-30}
\\[-1.1em]
\multicolumn{1}{c}{\textbf{Size} $n$} 
&&& \multicolumn{6}{c}{$ 1 $} && \multicolumn{6}{c}{$ 10^{-3} $} && \multicolumn{6}{c}{$ 10^{-6} $} && \multicolumn{6}{c}{$ 10^{-9} $}
\\ %\cline{1-1} 
\cline{4-9} \cline{11-16} \cline{18-23} \cline{25-30}
&&& \multicolumn{6}{c}{\bf CG Iters} && \multicolumn{6}{c}{\bf CG Iters} && \multicolumn{6}{c}{\bf CG Iters} && \multicolumn{6}{c}{\bf CG Iters}
\\[-0.2em]
\midrule
108 		& 4 \textsf{MB}
    && 5 & 5 & 5 & 6 & 5 & 5     	&&  15 & 12 & 12 & 13 & 9 & 9 		&&  19 & 16 & 16 & 13 & 9 & 9 		&& $\times$ & $\times$ & $\times$ & 13 & 9 & 9 \\
208 		& 14 \textsf{MB} 
    && 5 & 5 & 5 & 7 & 6 & 6     	&&  19 & 13 & 13 & 17 & 10 & 10 	&&  25 & 18 & 18 & 17 & 10 & 10 	&& $\times$ & $\times$ & $\times$ & 17 & 10 & 10 \\
414 		& 54 \textsf{MB} 
    && 6 & 6 & 6 & 7 & 6 & 6     	&&  17 & 13 & 13 & 14 & 10 & 10 	&&  22 & 18 & 18 & 15 & 10 & 10 	&& $\times$ & $\times$ & $\times$ & 15 & 10 & 10 \\
775 		&  188 \textsf{MB} 
    && 6 & 6 & 6 & 7 & 7 & 7     	&&  16 & 14 & 14 & 14 & 10 & 10 	&&  20 & 19 & 19 & 14 & 10 & 10 	&& $\times$ & $\times$ & $\times$ & 14 & 10 & 10 \\
1\,505 	&  709 \textsf{MB} 
    && 8 & 7 & 7 & 10 & 7 & 7   	&&  19 & 14 & 14 & 16 & 10 & 10 	&&  24 & 18 & 18 & 16 & 10 & 10 	&& $\times$ & $\times$ & $\times$ & 16 & 10 & 10 \\
2\,880	& 3 \textsf{GB} 
    && 11 & 8 & 8 & 13 & 8 & 8  	&&  25 & 15 & 15 & 20 & 10 & 10 	&& $\times$ & 19 & 19 & 20 & 10 & 10   && $\times$ & $\times$ & $\times$ & 20 & 10 & 10 
\\
5\,561	& 9 \textsf{GB} 
    && 15 & 9 & 9 & 17 & 8 & 8  	&& $\times$ & 15 & 15 & 26 & 9 & 9 	&& $\times$ & 19 & 19 & 26 & 9 & 9 	&& $\times$ & $\times$ & $\times$ & 26 & 9 & 9 \\
10\,580	&  34 \textsf{GB} 
    && 21 & 9 & 9 & 24 & 9 & 9     	&& $\times$ & 15 & 15 & $\times$ & 10 & 10 	&& $\times$ & 20 & 20 & $\times$ & 10 & 10 	&& $\times$ & $\times$ & $\times$ & $\times$ & 10 & 10 \\
20\,193	&  125 \textsf{GB} 
    && 29 & 10 & 10 & $\times$ & 9 & 9			&& $\times$ & 16 & 16 & $\times$ & 10 & 10 	&& $\times$ & 20 & 20 & $\times$ & 10 & 10	&& $\times$ & $\times$ & $\times$ & $\times$ & 10 & 10 \\
38\,720	&  458 \textsf{GB} 
    && $\times$ & 11 & 11 & $\times$ & 10 & 10	&& $\times$ & 16 & 16 & $\times$ & 10 & 10 	&& $\times$ & 21 & 21 & $\times$ & 10 & 10	&& $\times$ & $\times$ & $\times$ & $\times$ & 10 & 10 \\
74\,420	& 2 \textsf{TB} 
    && $\times$ & 12 & 12 & $\times$ & 11 & 11	&& $\times$ & 17 & 17 & $\times$ & 11 & 11 	&& $\times$ & $\times$ & $\times$ & $\times$ & 11 & 11 	&& $\times$ & $\times$ & $\times$ & $\times$ & 11 & 11 \\
142\,636	& 6 \textsf{TB} 
    && $\times$ & 13 & 13 & $\times$ & 11 & 11	&& $\times$ & 17 & 17 & $\times$ & 12 & 12 	&& $\times$ & $\times$ & $\times$ & $\times$ & 12 & 12 	&& $\times$ & $\times$ & $\times$ & $\times$ & 12 & 12 \\
272\,728	& 22 \textsf{TB} 
    && $\times$ & 14 & 14 & $\times$ & 12 & 12	&& $\times$ & 18 & 18 & $\times$ & 12 & 12 	&& $\times$ & $\times$ & $\times$ & $\times$ & 12 & 12 	&& $\times$ & $\times$ & $\times$ & $\times$ & 12 & 12 \\
1\,000\,610 & 299 \textsf{TB} 
    && $\times$ & 15 & 15 & $\times$ & 13 & 13	&& $\times$ & 20 & 20 & $\times$ & 13 & 13 	&& $\times$ & $\times$ & $\times$ & $\times$ & 13 & 13 	&& $\times$ & $\times$ & $\times$ & $\times$ & 13 & 13
\\
\bottomrule
\end{tabular}
\\[0.5em]
%\footnotetext
\flushleft
\caption{Number of CG iterations for different regularization values \(\lambda \in \{1, 10^{-3}, 10^{-6}, 10^{-9}\}\) and each choice of preconditioner \( \mathsf{M} \in \{ \Idn, \Prec[a], \Prec[b],\mathsf{I}_U, \mathcal{P}_{U,\mathsf a}, \mathcal{P}_{U,\mathsf b} \} \) when solving the nonlocal system via the NFFT--based fast summation method. 
The symbol {\MaxT} is reported whenever a solution was not found within the maximum number of iterations.}
\vspace{-2\baselineskip}
\label{tb:Its_CG}
\end{table}
% ------------------------------------- %

Once more, we considered different scalings of an image, namely \cref{sub:Imgs_Cat} with base size \( 375\times 300\), resulting in problem sizes \(n\) ranging from a few hundred to up to a million. Additionally, we fixed \( \sigma = 30\), \(\mu = 10^{-2}\), and \( \mathsf{L} = 41\). The value of the regularizer \(\lambda\) was taken in the set \( \Lambda = \{ 10^{-\imath}:\, \imath\in \llb 0,9\rrb \} \). For CG, we limited the number of iterations to \(30\), as it allowed for a reasonable solution time that could be inferred from the results of \cref{tb:Solution_Times}. Moreover, we tightened the tolerance of the iterative solver to \(10^{-8}\).

Some of the results are included in \cref{tb:Its_CG}. The first column displays the problem size, the second column displays the estimated memory that would be required to store the nonlocal subkernels in double--precision floating--point format, and the third--to--last columns display the number of CG iterations required for each test, excluding the use of \(\Prec[Ch]\) due to its poor performance. To be precise, each set of six columns corresponds to a fixed value of \(\lambda \in \{ 1, 10^{-3}, 10^{-6}, 10^{-9} \}\), and each column within a pack represents a different choice of preconditioner \( \mathsf{M} \in \{ \Idn, \Prec[a], \Prec[b], \mathsf{I}_U, \mathcal{P}_{U,\mathsf a}, \mathcal{P}_{U,\mathsf b} \} \).
Although tests were conducted for additional values of \(\lambda\), the presented results capture a clear trend that summarizes the behavior of CG and the preconditioners as \(\lambda\) becomes smaller and \(n\) grows.

Let us begin by analysing the rows of \cref{tb:Its_CG}. In terms of the dimension, it seems that, in a similar way as in \cref{tb:Solution_Times}, the problem becomes more ill--conditioned as \(n\) grows, hence the number of iterations for all methods grows for any \(\lambda\). We note that for small instances \((n < 2\,000)\), all methods were able to run under the maximum number of iterations except for the very ill--conditioned case \( \lambda = 10^{-9}\), where only the preconditioners under the change of basis induced by \(U\) were able to produce a solution. Let us focus on the cases \( \lambda \geq 10^{-6}\). Here we notice that as \(n\) grows, unpreconditioned CG in the original basis is the first method to fail in finding a solution. The behavior is closely followed by the unpreconditioned case in the basis induced by \(U\) and then by the two diagonal preconditioners. Now let us consider the performance of the preconditioners by carrying out a column--wise comparison. First, notice that the two diagonal preconditioners and their dense counterparts always took the same number of iterations. This behavior was expected from the equivalence results of \cref{lem:Equivalence_Diag_Precs,lemma:Spectral_Control_Pa_on_D}, suggesting that the Jacobi preconditioner, and its dense variant, is a preconditioner--of--choice due to its simplicity in contrast to the \(\ell_2\) choice \cref{eq:ANOVA_Norm_2_Prec}. Now focus on the range \(\lambda \leq 10^{-3}\). Here we observe that the number of CG iterations in the original basis always exceeds the number of iterations for the basis induced by \(U\). In particular, we can see that the case \( \mathsf M = \UnitSim[U]{I}\) already acts as a preconditioned method as the decoupling removes the ill--conditioned diagonal block given by \(\lambda\), while the second diagonal block is better conditioned and controlled only by the spectral radius and the algebraic connectivity of \(B\) shifted by \(\lambda\). It is interesting to notice that the number of preconditioned iterations does not change for the decoupled system as \(\lambda\) becomes smaller. This behavior is explained by the quality of the approximation \( \nicefrac{ [\rho(B) + \lambda] }{ [a(B) + \lambda] } \approx \nicefrac{ \rho(B) }{ a(B) } \). In other words, if \(\lambda\) is small enough in terms of the order of magnitude of \( a(B) \), then the decoupled method will have roughly the same condition number regardless of how small \(\lambda\) can be. The opposite behavior occurs when \( \lambda\) lies close to \( a(B) \) or has a larger value. This can be seen for \(n < 20\,000\) and \(\lambda = 1\), where we had that \( a(B) < \lambda\). Here, the number of iterations of the diagonal preconditioners and their dense variants was identical, and the unpreconditioned case under the change of basis induced by \(U\) no longer affected the eigenvalues, yielding the only case where CG with \( \mathsf M = \UnitSim[U]{I}\) was outperformed by CG with \( \mathsf M = \Idn\). However, notice that as \(n\) grew, it held that \( \lambda \lesssim a(B) \), and the problem became more ill--conditioned, then once more the dense preconditioners were able to take over and outperform the diagonal preconditioners.

% n = 775 for plot
% ------------------------------------- %
\begin{figure}[h!]
\centering
	\includegraphics[scale=0.65,draft=false]{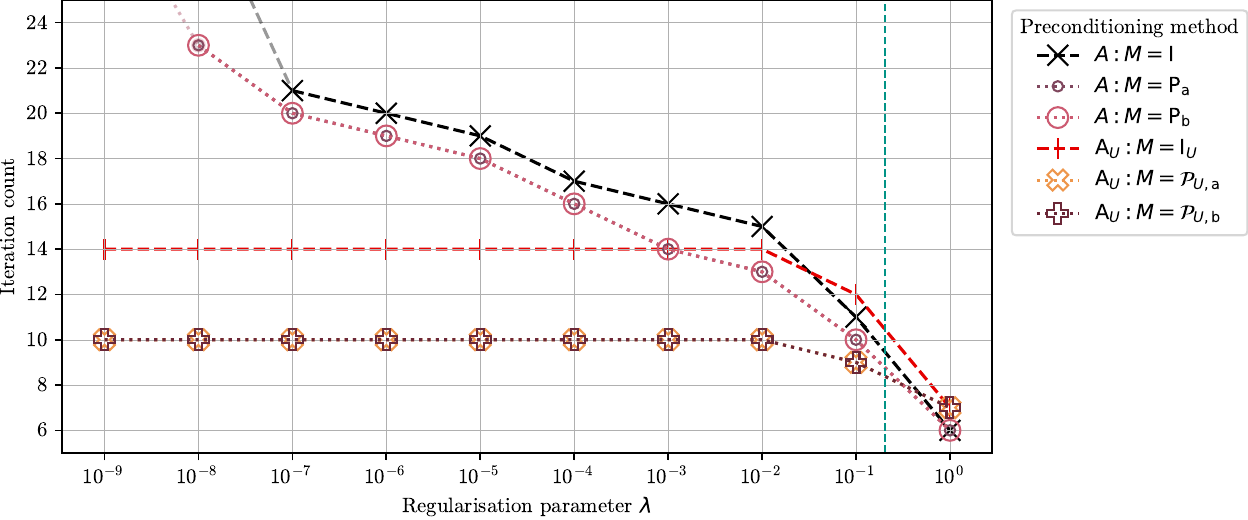}
\caption{Comparative display of the number of CG iterations for different regularization values \(\lambda \in \Lambda\) and each choice of preconditioner. Here \( \mathsf{M} \in \{ \Idn, \Prec[a], \Prec[b]\} \) was used for the system \( A\bfu = \lambda \bff\), while \( \mathsf{M} \in \{\mathsf{I}_U, \mathcal{P}_{U,\mathsf a}, \mathcal{P}_{U,\mathsf b} \} \) was used for the decoupled version of the system \( \UnitSim{A} \bfx = \lambda U^\top \bff\). The strong black and red dashed lines correspond to the unpreconditioned cases for both bases, while the other four dashed lines correspond to the preconditioned cases, respectively. If a test failed, the lines are extrapolated to a larger quantity in the vertical axes which is not depicted to indicate that the required number of iterations grows, for that method takes a greater number of iterations than the maximum allowed. The teal dashed vertical line represents the approximation of the algebraic connectivity \( a(B) \approx \frac{n}{n-1} \mu \min \etab \). 
}
\label{fig:CG_Its}
\vspace{-2\baselineskip}
\end{figure}
% ------------------------------------- %

The behavior just described is also depicted for more values of \(\lambda\) in \cref{fig:CG_Its}. By fixing \(n = 775\), we are able to see the significantly improved performance of the dense preconditioners in conjunction with the change of basis and decoupling. Under this configuration, we see that for the cases where \(\lambda\) is smaller than \(a(B)\) by an order of magnitude, the iteration count remains constant. In contrast, for values that are close or larger than the algebraic connectivity of \(B\), the iteration count starts to vary and resembles the behavior observed for the cases under the original basis. The preconditioners for the original basis are clearly impacted by the value of the regularizer, requiring an increasing number of iterations for the worse--conditioned configurations up to the point where \(30\) iterations was no longer enough to solve the problem. Notice that as \( \lambda \) gets closer from the left to \( [a(B), \rho(B)]\), then the conditioning of the problem is better described by the quotient \( \nicefrac{ [\rho(B) + \lambda] }{ [a(B) + \lambda] } \). As a result, we can expect that for larger values of \(\lambda\) the problem will become better conditioned and the number of iterations will decrease to a constant value. This latter behavior amounts to the choice of a regularizer with value greater than \( \rho(B)\), as the system will essentially behave like a scaling of the identity.

We do not include the results for \(\Prec[Ch]\) due to its poor performance. For \(\lambda = 10^{-9}\) and \(\lambda = 10^{-6}\), \(\Prec[Ch]\) failed to produce an error--acceptable solution in all cases. For \(\lambda = 10^{-3}\), CG consistently required more iterations under this choice of preconditioner than the unpreconditioned strategy \( \mathsf{M} =\Idn\), and similarly for \(\lambda = 1\), it took at most two additional iterations compared to the identity. Moreover, the computational cost of employing \(\Prec[Ch]\) is prohibitive: its usage in CG succeeded in only 13 out of 56 attempts. In fact, even the preconditioners with the richest sampled data, accumulating a total count of 150 matrix--vector evaluations, were not able to compete against the other preconditioners. The time for building \(\Prec[Ch]\) ranged from around \(7\) seconds per sample for smaller problems to over \(3\) hours to the largest problem. As a result, by the time the preconditioner was ready to use, many other preconditioners had already solved the problem entirely. These results suggest that \(\Prec[Ch]\) is not an optimal choice of preconditioner for this class of problems: we believe this is due to the full--rank and relatively well--conditioned nature of this problem structure better captured with diagonal information than a low--rank approximation.

% ----------------------------------------------------- %
\subsection{Parameter learning application} \label{ssec:Learning}

As a proof of concept, here we are interested in the solving the inverse problem of determining \(\lambda\) from a set of clean and noisy images \( \big\{ (\bfu_{t,\textsf c}, \bff_t) \big\}_{t \in \llb 1,\mathsf T\rrb} \). Following the approach of \cite{D’Elia2021}, we aim to solve the bilevel problem, which we state in fully discretized form, given by
\begin{subequations}\label{eq:BilevelProblem}
\begin{align}
    &\qquad \min J(\bfu;\lambda) \coloneqq \frac{1}{\mathsf T} \sum\limits_{t=1}^{\mathsf T} \| \bfu_{t,\textsf c} - \bfu_t \|_2^2
    \\[-0.5em]
    \text{subject to} &  \notag
    \\
    & \bfu_t = \argmin_{ \bfv } \frac{\mu}{2} \langle\bfv, L_t\bfv\rangle  + \frac{\lambda}{2} \| \bfv - \bff_t \|_2^2
    \qquad \forall t\in \llb 1,\mathsf T \rrb,   \label{Prob:LowerLevel}
    \\
    & \lambda \in \Lambda \coloneqq [\Lambda_{\min}, \Lambda_{\max}].
\end{align}
\end{subequations}
Here, for any index \(t\in \llb 1,\mathsf T\rrb\), the Laplacian operator \(L_t\) is associated with the ANOVA kernel \(\Gamma_t\) built using features of the noisy variant \( \bff_t \), which in turn was constructed using additive noise with high variance. The bounds of \( \Lambda \) are positive to always ensure that any feasible \( \bfu_t\) resembles \( \bff_t\). For more details on the analysis of problems like \cref{eq:BilevelProblem}, we refer the reader to \cite{D’Elia2021}. In particular, we know that the problem has a unique solution for any \( \lambda \in \Lambda\), and that we can optimize the alternative reduced function \( j(\lambda) \coloneqq J\big(\bfu(\lambda);\lambda\big)\). Moreover, the solution of the lower--level problem is completely characterized by the family of systems \( (\lambda \Idn + \mu L_t) \bfu_t = \lambda \bff_t\) for all \(t\in \llb 1,\mathsf T\rrb\).

Given a solution of \cref{eq:BilevelProblem}, which we label \(\lambda_{\mathsf T}\), we validate the parameter by solving the lower--level problem \cref{Prob:LowerLevel} adapted to a new set of noisy samples from the validation set   \( \big\{ (\bfu_{v,\textsf c}, \bff_v) \big\}_{v \in \llb 1,\mathsf V\rrb} \). To quantify the effectiveness of a reconstruction, we use the structural similarity index (SSIM), which measures the similarity of the recovered image against its corresponding clean version.

The training set is composed of clean images corresponding to their noisy variants in \cref{sub:Imgs_Train_Noisy}. Similarly, the clean images corresponding to their noisy variants in \cref{sub:Imgs_Val_Noisy} form the validation set. Overall, the datasets contained $\mathsf{T} =15$ and $\mathsf{V} = 11$ images of different sizes. The problem sizes range from \(6\,794\) to \(10\,560\) for the training set and from \( 7\,566\) to \(11\,160\) for the validation set, respectively.

% ------------------------------------- %
\begin{figure}[ht!]
\centering
    \subcaptionbox{Noisy samples \label{sub:Imgs_Train_Noisy}}
    {
    \includegraphics[scale=0.5,page=2]{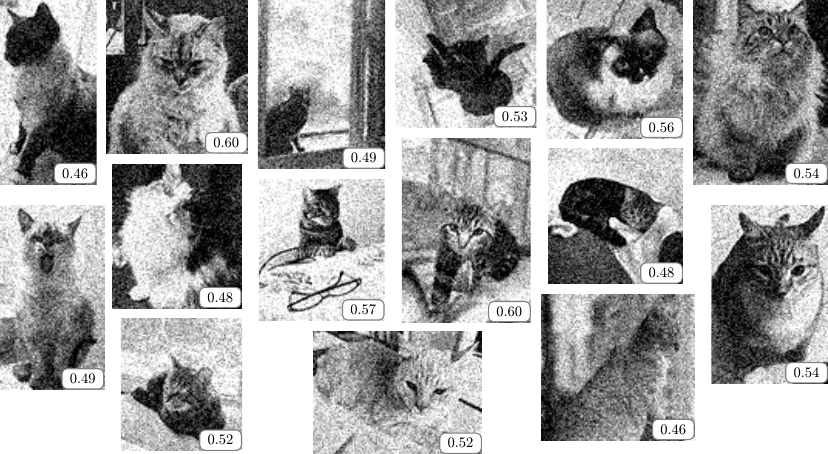}
    }
% \end{figure}
% \begin{figure}[ht!]\ContinuedFloat
% \centering
    \subcaptionbox{Recovered images \label{sub:Imgs_Train_Recovered}}
    {
    \includegraphics[scale=0.5,page=3]{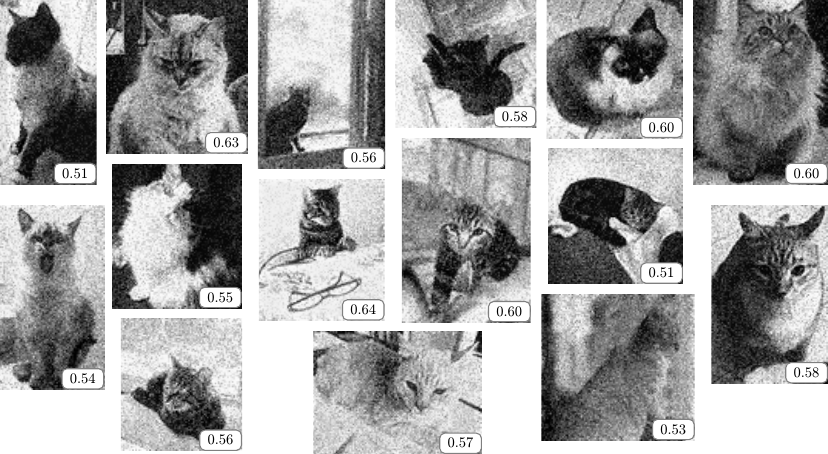}
    }
    \vspace{-1em}
    \caption{Noisy inputs and outputs of the training phase. At the bottom--left corner of each image, we include the SSIM value, rounded to two digits, of the image with respect to the clean variant \(\bfu_{t,\textsf c}\) for all \( t\in \llb 1, \mathsf T \rrb\).}
\label{fig:Training}
\vspace{-1\baselineskip}
\end{figure}
% ------------------------------------- %

%\vspace{-1\baselineskip}

% ------------------------------------- %
\begin{figure}[ht!]
\centering
    \subcaptionbox{Noisy samples \label{sub:Imgs_Val_Noisy}}
    {
    \includegraphics[scale=0.5,page=5]{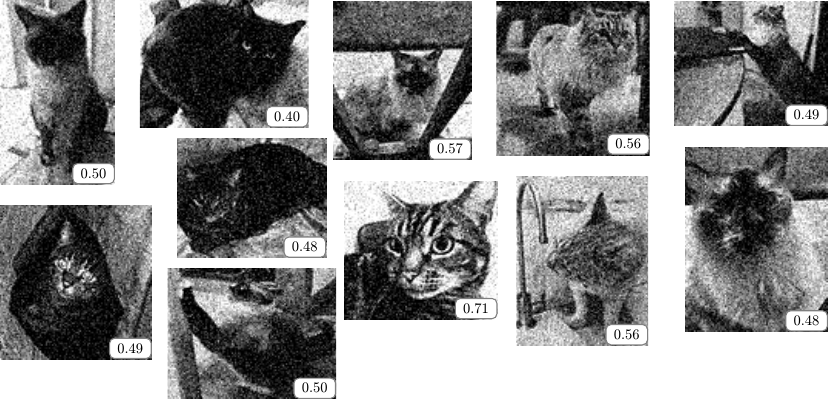}
    }
% \end{figure}
% \begin{figure}[ht!]\ContinuedFloat
% \centering
    \subcaptionbox{Recovered images \label{sub:Imgs_Val_Recovered}}
    {
    \includegraphics[scale=0.5,page=6]{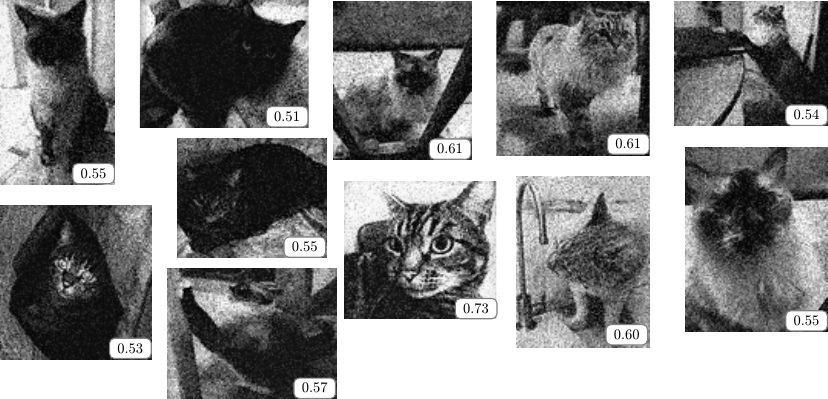}
    }
    \vspace{-1em}
    \caption{Noisy inputs and outputs of the validation phase. At the bottom--left corner of each image, we include the SSIM value, rounded to two digits, of the image with respect to the clean variant \(\bfu_{v,\textsf c}\) for all \( v\in \llb 1,\mathsf V\rrb\).}
\label{fig:Validation}
\vspace{-2\baselineskip}
\end{figure}
% ------------------------------------- %

To solve the bilevel problem \cref{eq:BilevelProblem}, we used the Brent method which is developed to minimize a scalar function of a scalar variable efficiently \cite[\S5]{brent2013algorithms}. Moreover, we parallelized the subkernel evaluation to speed up the overall kernel fast summation. By distributing the work of each subkernel when reconstructing a particular image, we obtained an efficient way to evaluate the composite objective functional. We fixed \(\mathsf{L} = 41\), \(\sigma = 40\), \( \mu = 10^{-2}\), and \(\Lambda = [10^{-9},10]\). The value of the upper bound on \(\lambda\) was selected to test the preconditioner for the ill--conditioned regime. By the nature of the training data, we would already need 335 \textsf{GB} in memory just to store all the kernels used for the evaluation of the lower--level systems. For solving the lower--level systems, we used preconditioned CG limited to 25 iterations and a relative tolerance of \(10^{-10}\). In particular, we employed the dense preconditioner \( \mathcal{P}_{U,\mathsf a}\) tailored to each image. Notice that the cost of assembling this preconditioner is absorbed by the computation of \(\etab\).

The optimization method was able to find a solution \( \lambda_{\mathsf T} \approx 7.48\) in 26 iterations with a tolerance of \(10^{-10}\). The average number of CG iterations per step was \(4\), highlighting the advantage of the our preconditioning scheme as the evaluation of the lower level problem was almost automatic. The results of the training are depicted in \cref{sub:Imgs_Train_Recovered}, where the average SSIM increased from 0.5221 to 0.5740. In general, we can observe that the method was able to smoothen the images and remove noise. However, in some cases the noise removal was not aggressive. This can be explained either by the choice of parameters and the clear limitations of using a one--parameter--fits--all approach. Notwithstanding, the method was able to find a solution in less than ten minutes that could later be postprocessed to remove the remaining artifacts. The method did not struggle with recovering textures and patterns and only slightly lowered the contrast of the output images. The validation test, depicted in \cref{sub:Imgs_Val_Recovered}, also showed an increase of the average SSIM from 0.5214 to 0.5576 while also featuring the preservation of some noise. Overall, this exercise suggests that the method can be used as the backbone of more specialized imaging tasks by allowing the computation and efficient solution of otherwise dense and ill--conditioned systems.

At this point, it is reasonable to question whether the reconstruction can be improved by jointly optimizing the regularization parameter \(\lambda\) and the shape parameter \(\sigma\) in \cref{eq:BilevelProblem}. By defining \( \theta \coloneqq \sigma^{-2}\), we adapt the bilevel training setup from \cref{eq:biparametric_problem_control} to guarantee the existence of solutions for the bivariate case. We employed L--BFGS--B \cite{Byrd1995,Zhu1997} with a two--point finite difference estimation of the Jacobian to avoid the additional computational cost of evaluating and storing the derivative kernel. The optimization was performed with fixed parameters: \(\mathsf{L} = 41\), \(\mu = 10^{-2}\), \(\Lambda = [10^{-9}, 255]\), and \(\Theta = [10^{-10}, 15]\). We also set a maximum of 50 iterations, a function discrepancy tolerance of \(10^{-15}\), and a limited memory size of \(5\). Under this setup, the algorithm converged to a minimizing pair \(( \lambda_{\mathsf T}, \theta_{\mathsf T}) \approx (15.87, 2.86 \times 10^{-4})\) in nine iterations, taking approximately \(38\) minutes. Once more, the average number of CG iterations was \(4\). Notice that \( \theta_{\mathsf T}^{-\nicefrac{1}{2}} \approx 59 \), which does not differ much from the selection of \(\sigma\) in the previous experiment. The objective function achieved a reduction of \(1.84\%\), while the mean SSIM increased to \(0.5758\), although we note that this is not the quantity that the problem formulation is minimizing. For the validation set, the reduction was \(4.15\%\), though the average SSIM remained unchanged compared to the one--dimensional problem in \cref{eq:BilevelProblem}. Overall, while the joint optimization process led to marginal improvements in reconstruction quality, it quadrupled the training time. This suggests that the additional computational effort may not be justified by the gain in the reconstruction performance, depending on the problem being examined, but it highlights the feasibility of our approach within such settings.

\section{Conclusions}\label{sec:Conc}

We have introduced a new NFFT--based framework for tackling bilevel optimization problems arising from image denoising applications. We utilized an ANOVA kernel, which may be readily applied in a matrix--free way within Krylov subspace solvers. To accelerate the solution algorithm for linear systems, we employed diagonal approximations coupled with a bespoke change of basis routine. Theoretical and numerical results underlined the potency of our methodology on a wide range of denoising and parameter learning models. Future work will involve applying similar frameworks to other classes of imaging problems, and investigating alternative preconditioners for problems where the eigenvalue distributions suggest alternative strategies to diagonal preconditioning.

\section*{Acknowledgements} A.M-T. thanks Bernhard Heinzelreiter (Maxwell Institute for Mathematical Sciences), Jennifer Buchberger, Belen Santacruz Reyes, and Alejandro Saenz Ortega for providing pictures of their cats and consenting for their use in publicly--visible research.

\section*{Data, code, and materials}
All the images and code for the experiments (and additional computational tests) are included at
% \begin{center}
%     \noindent 
    \href{https://github.com/andresrmt/Prec_GLs_NFFT_BLO}{\texttt{https://edin.ac/3zh86hT}}.
% \end{center}

% ----------------------------------------------------- %
% ----------------------------------------------------- %
% ----------------------------------------------------- %

\footnotesize
\bibliographystyle{siamplain}
%\bibliography{references}
\bibliography{ex_article.bbl}
\end{document}